\documentclass[reqno]{amsart} 
\usepackage{graphicx, amsmath, amssymb, amsfonts, amsthm, stmaryrd, amscd}
\usepackage[usenames, dvipsnames]{xcolor}
\usepackage{tikz}

\usepackage{enumerate}
\usepackage[normalem]{ulem}

\usepackage{ifpdf}
\ifpdf \RequirePackage[pdftex]{hyperref} \fi
\usepackage{xparse}

\usepackage[all]{xy}
\usepackage{enumerate}
\usetikzlibrary{matrix,arrows,decorations.pathmorphing}

\makeatletter
\newcommand*{\transpose}{%
  {\mathpalette\@transpose{}}%
}
\newcommand*{\@transpose}[2]{%
  \raisebox{\depth}{$\m@th#1\intercal$}%
}
\makeatother

\makeatletter
\newcommand*{\da@rightarrow}{\mathchar"0\hexnumber@\symAMSa 4B }
\newcommand*{\da@leftarrow}{\mathchar"0\hexnumber@\symAMSa 4C }
\newcommand*{\xdashrightarrow}[2][]{%
  \mathrel{%
    \mathpalette{\da@xarrow{#1}{#2}{}\da@rightarrow{\,}{}}{}%
  }%
}
\newcommand{\xdashleftarrow}[2][]{%
  \mathrel{%
    \mathpalette{\da@xarrow{#1}{#2}\da@leftarrow{}{}{\,}}{}%
  }%
}
\newcommand*{\da@xarrow}[7]{%
  \sbox0{$\ifx#7\scriptstyle\scriptscriptstyle\else\scriptstyle\fi#5#1#6\m@th$}%
  \sbox2{$\ifx#7\scriptstyle\scriptscriptstyle\else\scriptstyle\fi#5#2#6\m@th$}%
  \sbox4{$#7\dabar@\m@th$}%
  \dimen@=\wd0 %
  \ifdim\wd2 >\dimen@
    \dimen@=\wd2 %
  \fi
  \count@=2 %
  \def\da@bars{\dabar@\dabar@}%
  \@whiledim\count@\wd4<\dimen@\do{%
    \advance\count@\@ne
    \expandafter\def\expandafter\da@bars\expandafter{%
      \da@bars
      \dabar@ 
    }%
  }%
  \mathrel{#3}%
  \mathrel{%
    \mathop{\da@bars}\limits
    \ifx\\#1\\%
    \else
      _{\copy0}%
    \fi
    \ifx\\#2\\%
    \else
      ^{\copy2}%
    \fi
  }%
  \mathrel{#4}%
}
\makeatother

\usepackage{mathtools}

\DeclareMathOperator{\GL}{GL}

\DeclareMathOperator{\Ad}{Ad}

\DeclareMathOperator{\Hom}{Hom}

\DeclareMathOperator{\End}{End}

\def\eps{\varepsilon}

\DeclareMathOperator{\U}{U}
\DeclareMathOperator{\trace}{trace}

\DeclareMathOperator{\prim}{prim}

\DeclareMathOperator{\diag}{diag}

\def\O{\operatorname{O}}

\DeclareMathOperator{\fin}{fin}

\DeclareMathOperator{\rank}{rank}

\DeclareMathOperator{\Eis}{Eis}

\DeclareMathOperator{\Ind}{Ind}

\DeclareMathOperator{\vol}{vol}

\theoremstyle{plain} \newtheorem{theorem} {Theorem}  \newtheorem{corollary} [theorem] {Corollary} \newtheorem{proposition} [theorem] {Proposition} 

\theoremstyle{definition} \newtheorem{definition} [theorem] {Definition}  
\newtheorem{example} [theorem] {Example}

            \newtheorem{remark} [theorem] {Remark}

\newtheoremstyle{itplain} 
{6pt}                    
{5pt\topsep}                    
{\itshape}                   
{}                           
{\itshape}                   
{.}                          
{5pt plus 1pt minus 1pt}                       
{}  

\theoremstyle{itplain} 
\newtheorem{lemma}[theorem]{Lemma}

\newtheorem*{lemma*}{Lemma}
\newtheorem*{proposition*}{Proposition}
\newtheorem*{definition*}{Definition}
\newtheorem*{example*}{Example}

\newtheorem*{results*}{Results}

\usepackage[displaymath,textmath,sections,graphics]{preview}
\PreviewEnvironment{align*}
\PreviewEnvironment{multline*}
\PreviewEnvironment{tabular}
\PreviewEnvironment{verbatim}
\PreviewEnvironment{lstlisting}
\PreviewEnvironment*{frame}
\PreviewEnvironment*{alert}
\PreviewEnvironment*{emph}
\PreviewEnvironment*{textbf}

\newcommand{\mfq}{\mathfrak{q}}

\newcommand{\RR}{R} \newcommand{\rn}{A}

\newcommand{\psiT}{\psi_{\tilde{T}}}

\PreviewEnvironment*{yh}

\numberwithin{equation}{section} \numberwithin{theorem}{section}

\begin{document}
\title{The subconvexity bound for standard L-function in level aspect}

\author{Yueke Hu} \address{YMSC, Tsinghua University, Beijing, China} \email{yhumath@tsinghua.edu.cn}

\author{Paul D. Nelson} \address{Aarhus University, Aarhus, Denmark} \email{paul.nelson@math.au.dk}
\begin{abstract}
  In this paper we prove a new subconvexity result  for the standard L-function of a unitary cuspidal automorphic representation $\pi$ of $\GL_n$, where the finite set of  places $S$ with large conductors is allowed to vary,
  provided that the local parameters at every place in $S$ satisfy certain uniform growth condition.
\end{abstract}
\maketitle
\tableofcontents

\section{Introduction}\label{sec:cnq7tfnatf}

\subsection{Background}\label{sec:cnq7tfm86o}

Let $\pi$ be a cuspidal automorphic representation of $\GL_n$ over $\mathbb{Q}$.  We may write the archimedean $L$-factor for $\pi$ in the form $L(\pi_\infty, s) = \prod_{i = 1 }^n \Gamma_{\mathbb{R} } (s + \lambda_j)$.  We refer to the $\lambda_j$ as the \emph{archimedean $L$-function parameters} of $\pi$ (or of its local component $\pi_\infty$).

We recall the main result of \cite{2021arXiv210915230N}.

\begin{theorem}\label{theorem:cnpy9gjptx}
  Let $n$ be a natural number, and let $c_0 > 0$.  For each $\delta < \delta_n^\sharp$, there exists $C_1 \geq 0$ with the following property.  Let $\pi$ be a unitary cuspidal automorphic representation of $\GL_n$ over $\mathbb{Q}$ such that each archimedean $L$-function parameter $\lambda_j$ of $\pi$ satisfies the ``uniform growth'' condition
  \begin{equation}\label{eqn:uniform-growth-archimedean}
    c_0 T_\infty  \leq 1 + \lvert \lambda_j  \rvert \leq  T_\infty.
  \end{equation}
  Then
  \begin{equation*}
    \lvert L(\pi, \tfrac{1}{2}) \rvert \leq C_1 (T_\infty^n)^{1/4 - \delta} C(\pi_{\fin})^{1/2}.
  \end{equation*}
\end{theorem}

This estimate is nontrivial when the ramification of $\pi$ concentrates at the infinite place, with parameters satisfying the uniform growth condition \eqref{eqn:uniform-growth-archimedean}.

This result built on \cite{2020arXiv201202187N}, which established a similar conclusion for $L$-functions on $\U_{n + 1} \times \U_n$, with $\U_{n}$ anisotropic (so that it defines a compact quotient).  Liyang Yang \cite{2023arXiv2309.07534} obtained related results for certain Rankin--Selberg convolutions.  Marshall \cite{2023arXiv2309.16667} established a subconvex bound in the depth aspect at a fixed $p$-adic place, assuming the analogue of \eqref{eqn:uniform-growth-archimedean} together with an additional ``wall-avoidance'' condition.

In \cite{2023arXiv2309.06314}, we initiated the study of higher rank subconvexity in ``horizontal aspects'', where the set of ramified places itself varies.  This required a quite different approach to the ``volume bound'', which is the main technical result driving the ``geometric'' estimates of prior results.  That paper, like \cite{2020arXiv201202187N}, treated the setting $U_{n+1}\times U_n$, with $\U_n$ anisotropic.  In this paper, we combine the tools and results from \cite{2021arXiv210915230N} and \cite{2023arXiv2309.06314} to derive a subconvexity bound for the standard L-functions, with possibly varying set of ramifications, assuming a parity condition on the ramification at finite places.  We summarize with the following diagram:

\begin{table}[htbp]
  \centering
  \caption{Higher-rank subconvex bounds by setting and ramification pattern}
  \begin{tabular}{|c|c|c|}
    \hline
    & \multicolumn{2}{c|}{\textbf{Set of ramified places}} \\
    \cline{2-3}
    & Fixed & Varying \\
    \hline
    \textbf{Compact quotient} & \cite{2020arXiv201202187N, 2023arXiv2309.16667} & \cite{2023arXiv2309.06314} \\
    \hline
    \textbf{GL$_\mathbf{n}$} & \cite{2021arXiv210915230N, 2023arXiv2309.07534} & This paper \\
    \hline
  \end{tabular}
\end{table}

To formulate our results, we first introduce an analogue, at a finite place, of the ``uniform growth'' condition \eqref{eqn:uniform-growth-archimedean}:

\subsection{Uniform growth at finite places}\label{sec:cnq7tbmypw}
Let $F$ be a non-archimedean local field, with ring of integers $\mathfrak{o}$ and maximal ideal $\mathfrak{p}$.  For example, we could take $F = \mathbb{Q}_p$ and $\mathfrak{o} = \mathbb{Z}_p$, so that $\mathfrak{p} = p \mathbb{Z}_p$.  Let $\mathfrak{q} \subseteq \mathfrak{p}$ be an ideal, thus $\mathfrak{q} = \mathfrak{p}^m$ for some $m \geq 1$.  Let $\pi$ be a representation of $\GL_n(F)$. Let $K(\mfq)$ be the principal congruence subgroup of depth $\mfq$.  We say that $\pi$ is \emph{uniform at depth $\mathfrak{q}^2$} if it contains a vector such that $K(\mfq)$ acts on it by a character $\chi_\tau$ with $\det(\tau)$ a unit. (See Definition \ref{definition:we-say-that-repr-pi-mathbfgf-has-emph-depth-mathfr-uniform-depth} and \cite[Definition 8.1]{2023arXiv2309.06314}.)  This condition is analogous to the uniform growth condition \eqref{eqn:uniform-growth-archimedean}.  It is the specialization of the ``regular'' condition given in \cite[Definition 8.6]{2023arXiv2309.06314} to the case that the representation of the smaller group is unramified.  We illustrate with some examples.
\begin{itemize}
\item Let $\pi$ be a principal series representation of $\GL_n(F)$, induced by characters $\chi_1, \dotsc, \chi_n$ of $F^\times$.  Then $\pi$ is uniform at depth $\mathfrak{q}^2$ precisely when each $\chi_j$ has conductor $\mathfrak{q}^2$, that is to say, when each $\chi_j$ has trivial restriction to $1 + \mathfrak{q}^2$ but not to $1 + \mathfrak{p}^{-1} \mathfrak{q}^2$.
\item In particular, the twist of an unramified representation by a character of conductor $\mathfrak{q}^2$ is uniform at depth $\mathfrak{q}^2$.
\item Certain supercuspidal representations of depth $\mathfrak{q}^2$, as discussed in \cite[Section 8.8]{2023arXiv2309.06314}, are uniform at depth $\mathfrak{q}^2$. On the other hand  for any fixed supercuspidal representation $\pi$ of $\GL_n(F)$, the twist $\pi \otimes \chi$ by a character of conductor $\mathfrak{q}^2$ is uniform at depth $\mathfrak{q}^2$ by Lemma \ref{lem:cq0ikz837w}, provided that $m$ is large enough.
\item \cite[Proposition 8.19]{2023arXiv2309.06314} implies that uniformity at depth $\mathfrak{q}^2$ is closed under parabolic induction.  That is to say, suppose given a partition $n = m_1 + \dotsb + m_k$, corresponding to a standard parabolic subgroup of $\GL_n(F)$.  For each $j \in \{1, \dotsc, k\}$, let $\pi_j$ be a representation of $\GL_{m_j}(F)$ that is uniform at depth $\mathfrak{q}^2$.  Then the same holds for the parabolic induction of $\pi_1 \otimes \dotsb \otimes \pi_k$ to $\GL_n(F)$.
\end{itemize}
We do not determine which representations are uniform at depth $\mathfrak{q}^2$; we are content to observe that the above discussion provides a rich source of examples.

\subsection{Main results}


\begin{theorem}\label{MainTheorem}
  Let $n$ be a natural number, and let $0<c_0<1$ be fixed.  There exists $\delta>0$ and $C_1\geq 0$ with the following property.

  Let $S$ be a finite set of places of $\mathbb{Q}$, containing the infinite place.  Let $T_\infty \geq 1$.  For each finite prime $p \in S$, let $\mathfrak{q}_p$ be a proper ideal in $\mathbb{Z}_p$.
  Denote $T_p=[\mathfrak{o}_p:\mathfrak{q}_p]^2$,
  and set $T := \prod_{\mathfrak{p} \in S} T_\mathfrak{p} \geq 1$.

  Let $\pi$ be a unitary cuspidal automorphic representation of $\GL_n$ over $\mathbb{Q}$ with the following properties.
  \begin{enumerate}
    [(i)]
  \item The archimedean $L$-function parameters for $\pi$ satisfy the uniform growth condition~\eqref{eqn:uniform-growth-archimedean}.
  \item For each finite prime $p \in S$, the local component $\pi_p $ is uniform at depth $\mathfrak{q}_p^2$.
  \end{enumerate}
  Then
  \begin{equation*}
    \lvert L(\pi, \tfrac{1}{2}) \rvert \leq C_1 (T^n)^{1/4 - \delta} C(\pi^S)^{1/2}.
  \end{equation*}
\end{theorem}
Here $C(\pi^S)$ is the conductor ``away from $S$''.

In view of examples of ``uniform growth'' discussed in \S\ref{sec:cnq7tbmypw}, this theorem allows us to resolve the subconvexity problem for twists of a fixed standard $L$-function in the case that the twisting character has square conductor:
\begin{corollary}
  For each natural number $n$, there exists $\delta>0$ such that for each fixed unitary cuspidal automorphic representation $\pi$ of $\GL_n$ over the rational numbers, there exists $C_1 \geq 0$ so that for all $t\in \mathbb{R}$, $N\in \mathbb{Z}$, and Dirichlet characters $\chi$ of conductor $N^2$ which is coprime to $C(\pi)$, we have
  \begin{equation*}
    \lvert L(\pi \otimes \chi, \tfrac{1}{2}+it) \rvert \leq C_1 \left((1+|t|)N^2\right)^{n(1/4-\delta)}.
  \end{equation*}
\end{corollary}
\begin{remark}
  The coprime condition is not essential.  To remove it, one must check the local uniformity in the scenario of joint ramification at a finite place $p$ where $\pi_p$ is fixed and $C(\chi_p)$ is large enough. Considering the discussion in Section \ref{sec:cnq7tbmypw}, this is indeed true when $\pi_p$ is parabolically induced from supercuspidal representations. Though we shall not bother to check for more general cases in this paper.
\end{remark}

\subsection{Sketch of proof}\label{Sec:sketch}
The strategy to prove Theorem \ref{MainTheorem} follows that of \cite[\S5]{2021arXiv210915230N}, with the main efforts focusing on formulating and proving the analogues of the key ingredients (more precisely Theorem 3.1 and 4.1) of \cite{2021arXiv210915230N}.  The explicit constructions of test vectors at finite places follow those in \cite{2023arXiv2309.06314}, and the bilinear form control is directly taken from \cite[Theorem 9.1]{2023arXiv2309.06314}.

Many of the $p$-adic adaptations of their archimedean counterparts in \cite{2021arXiv210915230N} turn out to be more streamlined and explicit.  This simplicity makes the current treatment more accessible, though we still rely on the previous work at several crucial points.  There are, however, additional complications when estimating the local period integrals relevant for controlling the error terms, which we briefly explain here.

The first issue is that in the estimate for archimedean local period integral in \cite[Proposition 21.28]{2021arXiv210915230N}, there is a term roughly of size $T_\infty^{-\infty}$, which is negligible in the setting of \cite {2021arXiv210915230N} as $T=T_\infty\rightarrow \infty$. This term however could lead to very large errors in our hybrid setting in Theorem \ref{MainTheorem}, when $T_\infty$ is very small or even fixed  compared to $T$.

Our solution to this issue is to introduce an auxiliary parameter $R$, defined by
\begin{equation*}
  R=\max\{T_\infty, T^{\delta_0}\}
\end{equation*}
for some  small fixed $\delta_0>0$ to be optimized later.  Then one can construct archimedean test function and test vectors with $R$ replacing $T_\infty$, and get similar bound as in \cite[Proposition 21.28]{2021arXiv210915230N} in terms of $R$. The term with $R^{-\infty}$ is then negligible.

Note that in the case $T_\infty<R$, \cite[Theorem 3.1]{2021arXiv210915230N} needs to be adjusted with some bounds weakened (see Lemma \ref{Lemma:archimedeananalogueforR}), since now the stability condition (see \cite[Section 7]{2021arXiv210915230N}) no longer holds. As the result of this, one get an upper bound for the L-function with an additional coefficient of size $R^{O(1)}$. This however can be compensated with the saving from the p-adic places and taking $\delta_0$ small enough.


The second issue is that one can not naively  prove an analogue of \cite[Proposition 21.28]{2021arXiv210915230N} for every local place $v\in S$ and multiply them together. This is due to the fact that the proof of \cite[Proposition 21.28]{2021arXiv210915230N} used conditions following from global considerations (more precisely, the conditions on $c$ and $Y$). For example, the diagonal matrix $c$ as in Lemma \ref{lemma:let-f_s-=-otim-in-s-f_mathfr-be-as-abov-assume-tha} or \cite[Lemma 20.19]{2021arXiv210915230N} must have integer entries due to the nonvanishing requirement of local period integrals at the remaining unramified places.
In the current setting, we only have $S-$integral condition (meaning that the entries of $c$ are integral only at finite places $p\nmid S$). Thus \cite[Proposition 21.28]{2021arXiv210915230N} is actually not directly applicable.

To solve this issue, we adapt some of the arguments and prove directly an $S$-adic variant of \cite[Proposition 21.28]{2021arXiv210915230N} (see Lemma \ref{lemma:cnjen3r9kt}), for which being $S$-adic integral is enough. See Section \ref{sec:cnq4lj3ylk} for more details.
(This approach also solves the issue about $Y$.)

Furthermore we also need to control the denominators for $c$ at finite places $p|S$, extending the integral condition at unramified places in \cite[Lemma 21.3]{2021arXiv210915230N}. We achieve this in Section \ref{sec:cnjobc8wnb}, reducing the problem to an effective control  of nontrivial contributions to Jacquet's integral for Whittaker functions in Proposition \ref{Prop:WhittakerIntDomain}. The proof of Proposition \ref{Prop:WhittakerIntDomain} is left in the Appendix, as it has somewhat different flavor from the rest of the paper and may have independent interest.
\subsection{Organization of the paper}
\S\ref{Sec:basic} consists of the basic notations. In \S\ref{Sec:LWhit} we discuss the Whittaker function associated to the test vectors constructed in \cite{2023arXiv2309.06314}. In \S\ref{sec:some-analysis-u-backslash-g} we give explicit construction of test functions relevant for defining pseudo-Eisenstein series, and verify its properties.

In \S\ref{Sec:mainlocal} we combine the results from \cite{2023arXiv2309.06314} with those from \S\ref{Sec:LWhit} \S\ref{sec:some-analysis-u-backslash-g}  to prove Theorem \ref{theorem:main-local-noncompact}, which is the direct analogue of \cite[Theorem 3.1]{2021arXiv210915230N}.
We also show Lemma \ref{Lemma:archimedeananalogueforR}, which is an adapted version of \cite[Theorem 3.1]{2021arXiv210915230N} in the case $T_\infty$ is too small.

In \S\ref{sec:cnjen3n0cb} we prove Theorem \ref{theorem:eisenstein-growth-bound-general-level} which is the analogue of \cite[Theorem 4.1]{2021arXiv210915230N} with a couple of novelties in the proof as explained in Section \ref{Sec:sketch}.

In \S\ref{sec:cnpv1hj5sh} we finish the proof of Theorem \ref{MainTheorem}.
\subsection{Acknowledgment}
Y.H. is supported by the National Key Research and Development Program of China (No. 2021YFA1000700).  P.N.\ is a Villum Investigator based at Aarhus University, supported by a research grant (VIL54509) from VILLUM FONDEN.

\section{Basic notations}\label{Sec:basic}
\subsection{Global}\label{Sec:notationglobal}
We denote by
\begin{itemize}
\item $\mathbf{V}$ a finite free $\mathbb{Z}$-module regarded as an affine group scheme, so that for each ring $R$, we have $\mathbf{V}(R)\simeq R^n$, with $n$ the rank of $\mathbf{V}$,
\item $\mathbf{G} = \GL(\mathbf{V})$, thus $\mathbf{G}(R)=\text{Aut}_R(\mathbf{V}(R))\simeq \GL_n(R)$,
\item $\mathbf{A}$ the diagonal subgroup of $\mathbf{G}$,
\item $\mathbf{N}$ the strictly upper-triangular unipotent subgroup,
\item $\mathbf{B}=\mathbf{A}\mathbf{N}$ the upper-triangular Borel subgroup, and
\item $\mathbf{Q} = \mathbf{A} \mathbf{U}$ another minimal parabolic subgroup, $\mathbf{U}$ its unipotent radical.
\end{itemize}
According to convenience, we will take $\mathbf{Q}$ to be $\mathbf{B}$ or its opposite, the lower-triangular Borel subgroup.

When working with a Rankin--Selberg pair (\S\ref{sec:cnq7tfnatf}, \S\ref{Sec:LWhit} and \S\ref{Sec:mainlocal}), we fix a basis element $e\in \mathbf{V}$ and let $\mathbf{V}_H$ denote the complement submodule, so that
\begin{equation*}
  \mathbf{V}=\mathbf{V}_H\oplus \mathbb{Z}e.
\end{equation*}
We define group schemes $\mathbf{H}, \mathbf{N}_H, \mathbf{U}_H$, etc., by analogy.  In such cases, we denote by $n+1$ the rank of $\mathbf{V}$, so that $n$ is the rank of $\mathbf{V}_H$.

\subsection{Local}\label{Sec:Notationlocal}
Let $F$ be a non-archimedean local field.  We then denote by
\begin{equation*}
  \mathfrak{o}, \qquad \mathfrak{p}, \qquad \varpi \in \mathfrak{p} ,\qquad q = \lvert \mathfrak{o} / \mathfrak{p}  \rvert
\end{equation*}
its ring of integers, its maximal ideal, a fixed uniformizer, and the cardinality of the residue field.  We write $G=\mathbf{G}(F)$, $V=\mathbf{V}(F)$, etc.

We denote by $K=\mathbf{G}(\mathfrak{o})$ the standard maximal compact open subgroup of $G$, and define $K_H=\mathbf{H}(\mathfrak{o})$, $K_N=\mathbf{N}(\mathfrak{o})$, etc., similarly.

For a unipotent subgroup $U$ (or $N$), denote by $\delta_U$ the corresponding modular character for the group $UA$.

Let $\psi$ be an unramified unitary character of $F$ (i.e., trivial on $\mathfrak{o}$ but not on $\mathfrak{p}^{-1}$).  We use it to define an unramified nondegenerate unitary character of $N$, also denoted $\psi$, in the standard way: $\psi(n) = \psi(n_{12} + n_{23} + \dotsb)$.

Let $\mathfrak{q} \subseteq \mathfrak{p}$ be a nonzero $\mathfrak{o}$-ideal.  We set
\begin{equation}\label{eqn:cj3m0de4ie}
  K(\mathfrak{q}) := \ker(\mathbf{G}(\mathfrak{o}) \rightarrow \mathbf{G}(\mathfrak{o}/\mathfrak{q})),
  \quad
  K_H(\mathfrak{q}) := \ker(\mathbf{H}(\mathfrak{o}) \rightarrow \mathbf{H}(\mathfrak{o}/\mathfrak{q})).
\end{equation}
We choose a generator $\tilde{T}^{1/2}$ of the fractional ideal $\mathfrak{q}^{-1}$, and define
\begin{equation}\label{eqn:tild-:=-left-tild-right2-qquad-tild-rho_-:=-left-t}
  \tilde{T} := \left( \tilde{T}^{1/2} \right)^2, \qquad \tilde{T}^{\pm \rho_U^\vee} :=
  \left( \tilde{T}^{1/2} \right)^{\pm 2 \rho_U^\vee},
  \quad
  T := \lvert \tilde{T} \rvert = \lvert \mathfrak{o} / \mathfrak{q}^2 \rvert.
\end{equation}
Here $U$ is the lower triangular unipotent subgroup, $\rho_U$ is the half sum of the positive roots relative to $U$, and $\rho_U^\vee$ is the corresponding coroot.

We denote by $\psi_{\tilde{T}}$ the unitary character of $F$ given by
\begin{equation*}
  \psi_{\tilde{T}} (x) := \psi (\tilde{T} x),
\end{equation*}
which then has conductor $\mathfrak{q}^2$.

The Haar measures on respective groups will be normalized so that $K$, $K_H$, $K_A=\mathbf{A}(\mathfrak{o})$, $\mathbf{N}(\mathfrak{o})$ and $\mathbf{U}(\mathfrak{o})$ all have volume $1$.

\section{Localized Whittaker functions} \label{Sec:LWhit}

\subsection{Linear-algebraic preliminaries}
In this section, we work over an individual ring $R$, and write $V := \mathbf{V}(R)$, etc.
\begin{definition}
  By a \emph{decorated flag} $\mathcal{F}$, we mean a flag $V_1 \subset \dotsb \subset V_n$ together with a basis element $\bar{e}_j \in V_j / V_{j-1}$ for each $j$.
\end{definition}
\begin{example}
  An ordered basis $\mathcal{B} = (e_1,\dotsc,e_n)$ gives rise to a decorated flag by taking $V_j := \langle e_1,\dotsc,e_j \rangle$ and $\bar{e}_j$ to be the image of $e_j$.  Conversely, every decorated flag arises in this way.
\end{example}
We denote by $N_{\mathcal{F}}$ the stabilizer of the decorated flag $\mathcal{F}$.  This coincides with the group of linear automorphisms that are strictly upper-triangular {and unipotent} with respect to some (equivalently, any) basis $\mathcal{B}$ giving rise to $\mathcal{F}$.

\begin{definition}
  For $\tau \in M$ and a decorated flag $\mathcal{F}$, we say that $\tau$ is \emph{subcyclic with respect to $\mathcal{F}$} if for some (equivalently, any) basis $(e_1,\dotsc,e_n)$ giving rise to $\mathcal{F}$, we have
  \begin{equation*}
    \tau e _j - e _{j + 1} \in \langle e_1, \dotsc, e_j \rangle
  \end{equation*}
  for all $j \in \{1, \dotsc, n-1\}$.
\end{definition}
Note that the above condition depends only upon $\mathcal{F}$.  Concretely, it says that the matrix form of $\tau$ with respect to such a basis is given by
\begin{equation*}
  \begin{pmatrix}
    \ast & \ast & \ast \\
    1 & \ast & \ast \\
    0 & 1 & \ast \\
  \end{pmatrix}
  .
\end{equation*}
Observe that if $\tau$ satisfies this condition, then so does $\Ad(g) \tau$ for each $g \in N_{\mathcal{F}} G_\tau$.  In fact, the converse holds:
\begin{proposition}\label{proposition:let-tau-in-m-be-regul-let-g-in-g-let-mathc-be-deco}
  Let $\tau \in M$, let $g \in G$, and let $\mathcal{F}$ be a decorated flag.  Suppose that $\tau$ and $\Ad(g) \tau$ are both subcyclic with respect to $\mathcal{F}$.  Then $g \in N_{\mathcal{F}} G_\tau$.
\end{proposition}
The proof will be given after a couple of lemmas.

\begin{definition}
  Let $\mathcal{B} = (e_1,\dotsc,e_n)$ be an ordered basis of $V$.  We say that $\tau \in M$ is \emph{cyclic} with respect to $\mathcal{B}$, or that $\mathcal{B}$ is a \emph{cyclic basis} for $\tau$, if $\tau e _j = e _{j + 1}$ for $j = 1, \dotsc, n-1$, or equivalently, if $\tau$ is of the form
  \begin{equation*}
    \begin{pmatrix}
      0 & 0 & \ast \\
      1 & 0 & \ast \\
      0 & 1 & \ast \\
    \end{pmatrix}
    .
  \end{equation*}
\end{definition}
In particular, if $\tau$ is cyclic with respect to $\mathcal{B}$, then $\tau$ is subcyclic with respect to the decorated flag $\mathcal{F}$ attached to $\mathcal{B}$.


\begin{lemma}\label{lemma:if-tau-in-m-regul-then-g_tau-acts-simply-trans-set}
  If $\tau \in M$ admits a cyclic basis, then the centralizer $G_\tau$ acts simply-transitively on the set of all bases with respect to which $\tau$ is cyclic.
\end{lemma}
\begin{proof}
  It is obvious that the action is simple. We check the transitivity here.  The Cayley---Hamilton theorem says $\tau$ is annihilated by its characteristic polynomial $P(x) = x^n + \sum_{j=0}^{n-1} c_j x^j$.  In particular, if $\tau$ is cyclic with respect to $(e_1,\dotsc,e_n)$, then
  \begin{equation*}
    0 = P(\tau) e_1 = \tau e_n + \sum_{j=0}^{n-1} c_j e_{j+1}.
  \end{equation*}
  If $\tau$ is cyclic with respect to both $(e_1,\dotsc,e_n)$ and $(e_1',\dotsc,e_n')$, then the above relation holds also for the $e_j'$, which shows that the linear automorphism $e_j \mapsto e_j'$ lies in $G_\tau$.
\end{proof}

\begin{lemma}\label{lemma:let-tau-in-m-be-regul-let-mathc-be-decor-flag-with}
  Let $\tau \in M$.  Let $\mathcal{F}$ be a decorated flag with respect to which $\tau$ is subcyclic.  Let $\mathcal{B}$ be any ordered basis that gives rise to $\mathcal{F}$.  Then there is a unique $N_{\mathcal{F}}$-conjugate of $\tau$ that is cyclic with respect to $\mathcal{B}$.
\end{lemma}
\begin{proof}
  We first verify that there is a unique ordered basis $\mathcal{B} ' = (e_1', \dotsc, e_n')$ such that
  \begin{enumerate}
    [(i)]
  \item\label{enumerate:mathc--gives-rise-same-decor-flag-as-mathc-} $\mathcal{B} '$ gives rise to the same decorated flag as $\mathcal{B}$, and
  \item\label{enumerate:tau-cyclic-with-respect-mathcalb-.-} $\tau$ is cyclic with respect to $\mathcal{B} '$.
  \end{enumerate}
  The uniqueness follows from the observation that, by~\eqref{enumerate:mathc--gives-rise-same-decor-flag-as-mathc-} and~\eqref{enumerate:tau-cyclic-with-respect-mathcalb-.-} respectively, we have
  \begin{equation*}
    e_1' = e_1 \quad \text{ and } \quad e_j' = \tau e_{j-1}' \quad (j = 2, \dotsc, n).
  \end{equation*}
  As for existence, since $\tau$ is subcyclic with respect to $\mathcal{F}$, it follows from the above formulas that the set $\mathcal{B} '$ of vectors constructed inductively by the above formulas is a basis and gives rise to $\mathcal{F}$.  It is then clear from the same formulas that $\tau$ is cyclic with respect to $\mathcal{B} '$.

  Turning to the proof of the lemma, it is equivalent to show that there is a unique $N_{\mathcal{F}}$-translate of $\mathcal{B}$ with respect to which $\tau$ is cyclic.  This follows from the previous observation and the fact that $N_{\mathcal{F}}$ acts simply-transitively on the set of ordered bases giving rise to $\mathcal{F}$, which is in turn a consequence of the fact that $G$ acts simply-transitively on the set of all bases and $N_{\mathcal{F}}$ is the stabilizer of $\mathcal{F}$.
\end{proof}

\begin{proof}
  [Proof of Proposition~\ref{proposition:let-tau-in-m-be-regul-let-g-in-g-let-mathc-be-deco}]
  Suppose that the cyclic element $\tau$ and its conjugate $\Ad(g) \tau$ are subcyclic with respect to the same decorated flag $\mathcal{F}$.  Let $\mathcal{B}$ be an ordered basis for $\mathcal{F}$ with respect to which $\tau$ is cyclic.  By Lemma~\ref{lemma:let-tau-in-m-be-regul-let-mathc-be-decor-flag-with}, there is a unique $u \in N_{\mathcal{F}}$ such that $\Ad(u g) \tau$ is cyclic with respect to $\mathcal{B}$.  By Lemma~\ref{lemma:if-tau-in-m-regul-then-g_tau-acts-simply-trans-set}, there is a unique $c \in G_\tau$ such that $c u g = 1$.  It follows, as required, that $g$ lies in $N_{\mathcal{F}} G_\tau$.
\end{proof}

\subsection{Stability}\label{sec:cnsgt0z1cm}

Here we work with $\mathbf{V} = \mathbf{V}_H \oplus \mathbb{Z} e$.  We denote by $\mathbf{M}$ the endomorphism ring of $\mathbf{V}$, regarded as the affine group scheme with $\mathbf{M}(R) = \End_R(\mathbf{V} (R)) \cong M_{n+1}(R)$.  We analogousy define the ``upper left $n \times n$ block'' $\mathbf{M}_H$ of $\mathbf{M}$, thus $\mathbf{M}_H (R) = \End_R(\mathbf{V}_H(R)) \cong M_n(R)$.  Given $\tau \in \mathbf{M}(R)$, we may define its projection $\tau_H \in \mathbf{M}_H(R)$.  We recall from \cite[Lemma 4.6]{2023arXiv2309.06314} that the following conditions on $\tau$ are equivalent:
\begin{itemize}
\item The vector $e$ and the dual vector $e^* : \mathbf{V}(R) \rightarrow R$ given by $e^*|_{\mathbf{V}_H(R)} = 0$, $e^*(e) = 1$ generate $\mathbf{V}(R)$ and $\Hom(\mathbf{V}(R), R)$, respectively, as $R[\tau]$-modules.
\item The characteristic polynomials of $\tau$ and $\tau_H$ generate the unit ideal in the one variable polynomial ring over $R$.
\end{itemize}
We say in such cases that $\tau$ is \emph{stable}.

\subsection{Concentration of localized Whittaker functions}

Retaining the setting of \S\ref{sec:cnsgt0z1cm}, we now let $F$ be a non-archimedean local field, with the standard accompanying notation (\S\ref{Sec:Notationlocal}).

Let $(e_1,\dotsc,e_{n+1})$ be a basis for $\mathbf{V}(\mathfrak{o})$ such that $(e_1,\dotsc,e_n)$ is a basis for $\mathbf{V}_H(\mathfrak{o})$ and $e=e_{n+1}$.  These induce bases of $\mathbf{V}(F), \mathbf{V}(\mathfrak{o}/\mathfrak{q})$, etc., that we write in the same way.  They give rise to flags, hence to unipotent subgroups $\mathbf{N} \leq \mathbf{G}$ (resp. $\mathbf{N}_H \leq \mathbf{H}$) given by strictly upper-triangular matrices.

Recall from \S\ref{Sec:Notationlocal} that for an ideal $\mathfrak{q} \subseteq \mathfrak{p}$, we have defined a character $\psiT$ of $F$ that is trivial on $\mathfrak{q}^{2}$ but not on $\mathfrak{p}^{-1} \mathfrak{q}^2$.  It can also be viewed as a character of $N$, by setting
\begin{equation*}
  \psiT(u) := \psiT(u_{12} + u_{23} + \dotsb).
\end{equation*}

We recall from \cite[(8.1)]{2023arXiv2309.06314} that for $\tau \in \mathbf{M}(\mathfrak{o}/\mathfrak{q})$, we can define a character $\chi_\tau$ of $K(\mathfrak{q})/K(\mathfrak{q}^2)$ by the formula
\begin{equation}\label{Eq:chitau}
  \chi_\tau(1 + x) = \psiT (\trace (x \tau )).
\end{equation}

We show next that any Whittaker function on $\mathbf{G}(F)$ that transforms under $K(\mathfrak{q})$ according to a stable subcyclic parameter $\tau \in \mathbf{M}(\mathfrak{o}/\mathfrak{q})$ has restriction to $\mathbf{H}(F)$ concentrated near the identity:
\begin{proposition}\label{proposition:localized-whittaker-function-subcyclic-parameter-concentrates}
  Let $\tau \in \mathbf{M}(\mathfrak{o}/\mathfrak{q})$ be
 subcyclic with respect to the basis $e _1, \dotsc, e_{n+1}$.
  Let $W : \mathbf{G}(F) \rightarrow \mathbb{C}$ satisfy that for all $(u,y) \in \mathbf{N}(F) \times K(\mathfrak{q})$, we have
  \begin{equation}\label{eq:cnsgt121j6}
    W(u g y) = \psiT(u) \chi_{\tau}(y) W(g).
  \end{equation}
  Then for $h \in \mathbf{H}(F)$, we have $W(h) \neq 0$ only if
  \begin{equation}\label{eqn:wh-neq-0-implies-h-in-mathbfn_hf-k_hmathfrakq.-}
    h \in \mathbf{N}_H(F) K_H(\mathfrak{q})
  \end{equation}
  in which case, writing $h = u y$ with $(u,y) \in \mathbf{N}(F) \times K(\mathfrak{q})$, we have
  \begin{equation}\label{eqn:wh-=-psiu-chi_tauy-w1.-}
    W(h) =\psiT(u) \chi_\tau(y) W(1).
  \end{equation}
\end{proposition}
\begin{proof}
  If $W(h) \neq 0$, then for all $(u,y) \in \mathbf{N}(F) \times K(\mathfrak{q})$, we must have
  \begin{equation*}
     u h y = h \implies \psiT(u) \chi_\tau(y) = 1,
   \end{equation*}
  or equivalently,
  \begin{equation}\label{eqn:u-in-mathbfnf-cap-h-mathc-h-1-impl-psiu-=-chi_t-1-}
    u \in \mathbf{N}(F) \cap h K(\mathfrak{q}) h^{-1} \implies \psiT(u) = \chi_\tau(h^{-1} u h).
  \end{equation}

  Consider first the case $h \in \mathbf{H}(\mathfrak{o})$.  Then $h K(\mathfrak{q}) h^{-1} = K(\mathfrak{q})$, so any $u$ as in the left hand side of~\eqref{eqn:u-in-mathbfnf-cap-h-mathc-h-1-impl-psiu-=-chi_t-1-} lies in $\mathbf{N}(F) \cap K(\mathfrak{q})$.  For such $u$, we have $\psiT(u) = \chi_\theta(u)$, where $\theta \in M(\mathfrak{o}/\mathfrak{q})$ is the standard lower-triangular nilpotent Jordan block, given, e.g., for $n+1 = 3$ by
  \begin{equation}\label{eqn:theta-:=-beginpmatrix-0--0--0-}
    \theta :=
    \begin{pmatrix}
      0 & 0 & 0 \\
      1 & 0 & 0 \\
      0 & 1 & 0 \\
    \end{pmatrix}
    .
  \end{equation}
  The right hand side of~\eqref{eqn:u-in-mathbfnf-cap-h-mathc-h-1-impl-psiu-=-chi_t-1-} then implies that $\Ad(h) \tau - \theta$ pairs trivially against strictly upper-triangular matrices in $K(\mathfrak{q})$, or equivalently, is upper-triangular when mod $\mathfrak{q}$, i.e.,
  \begin{equation*}
    \Ad(h) \tau \equiv
    \begin{pmatrix}
      \ast & \ast & \ast \\
      1 & \ast & \ast \\
      0 & 1 & \ast \\
    \end{pmatrix}
    \mod \mathfrak{q}
    .
  \end{equation*}
  In other words, $\Ad(h) \tau \in \mathbf{M}(\mathfrak{o}/\mathfrak{q})$ is subcyclic with respect to the given basis.

  Since $\tau$ is subcyclic with respect to the same basis, we deduce from Proposition~\ref{proposition:let-tau-in-m-be-regul-let-g-in-g-let-mathc-be-deco} that the image $\bar{h} \in \mathbf{H}(\mathfrak{o}/\mathfrak{q})$ lies in the intersection
  \begin{equation*}
    \mathbf{H}(\mathfrak{o}/\mathfrak{q}) \cap \mathbf{N}(\mathfrak{o}/\mathfrak{q}) \mathbf{G}_\tau(\mathfrak{o}/\mathfrak{q}).
  \end{equation*}
  We may thus find an element $v \in \mathbf{N}(\mathfrak{o}/\mathfrak{q})$ such that $v^{-1} \bar{h}$ centralizes $\tau$.  On the other hand, $v^{-1} \bar{h}$ fixes the $\tau$-cyclic vector $e^*$ (as in \cite[Definition 3.1]{2023arXiv2309.06314}), so by the proof of \cite[Lemma 4.7]{2023arXiv2309.06314}, we deduce that $v^{-1} \bar{h} = 1$.  It follows that $h \in \mathbf{N}_H(\mathfrak{o}) K_H(\mathfrak{q})$, as required.

  It remains to address the case where $h \in \mathbf{H}(F) - \mathbf{H}(\mathfrak{o})$.  By the Iwasawa decomposition, we have $h \in \mathbf{N}_H(F) \mathbf{A}_H(F) K_H$.  By the transformation property of $W$ on the left under $\mathbf{N}(F)$, we may reduce to the case $h \in \mathbf{A}_H(F) K_H$.  We may then write
  \begin{equation*}
    h = a k, \quad k \in K_H, \quad a = \diag(\varpi^{\ell_1}, \dotsc, \varpi ^{\ell_{n+1}}), \quad \ell_j \in \mathbb{Z}, \quad \ell_{n+1} = 0.
  \end{equation*}
  Assuming $W(h) \neq 0$ and using the transformation property \eqref{eq:cnsgt121j6}, we aim to produce a contradiction to~\eqref{eqn:u-in-mathbfnf-cap-h-mathc-h-1-impl-psiu-=-chi_t-1-}.  Doing so will involve the elementary matrices $E_{i j}$ with $1$ in the $(i,j)-$entry and $0$ elsewhere.

  Suppose for the moment that
  \begin{equation}\label{eqn:ell_k--ell_k+1-}
    \ell_k < \ell_{k+1}
  \end{equation}
  for some $k \in \{1, \dotsc, n\}$.  Let $t$ be an element of $\mathfrak{q}^2 - \mathfrak{p} \mathfrak{q}^2$, to be determined later.  Set
  \begin{equation*}
    v := 1 + t E_{k,k+1} \in \mathbf{N}(F) \cap K(\mathfrak{q}^2),
  \end{equation*}
  so that
  \begin{equation*}
    u := a v a^{-1} = 1 + \varpi^{\ell_k - \ell_{k+1}} t E_{k,k+1}.
  \end{equation*}
  By~\eqref{eqn:ell_k--ell_k+1-} and our construction of $\psiT$, we see that $t$ may be chosen so that $\psiT (u) \neq 1$.  On the other hand, we have
  \begin{equation*}
    h^{-1} u h = k^{-1} v k \in K(\mathfrak{q}^2),
  \end{equation*}
  so $\chi_\tau(h^{-1} u h) = 1$.  We obtain the required contradiction to~\eqref{eqn:u-in-mathbfnf-cap-h-mathc-h-1-impl-psiu-=-chi_t-1-}.

  It remains to consider the ``anti-dominant'' case in which
  \begin{equation*}
    \ell_1 \geq \ell_2 \geq \dotsb \geq \ell_n \geq \ell_{n+1} = 0.
  \end{equation*}
  If each of these inequalities is in fact an equality, then $h \in \mathbf{H}(\mathfrak{o})$, contrary to assumption.  Otherwise, we may find $m \in \{1, \dotsc, n\}$ so that
  \begin{equation*}
    \ell_1 \geq \dotsb \geq \ell_m > 0, \qquad \ell_{m+1} = \dotsb = \ell_{n+1} = 0.
  \end{equation*}
  We now argue as in the proof of \cite[Lemma 9.3]{2023arXiv2309.06314}, as follows.  We use the partition
  \begin{equation*}
     n + 1 = m + (n-m + 1)
  \end{equation*}
  to describe $\mathbf{M}(\mathfrak{o}/\mathfrak{q})$ in terms of $2 \times 2$ block matrices.
  We claim that the lower-left block of $\Ad(k) \tau$ is nonzero modulo $\mathfrak{p}$: otherwise, $\mathbf{V}_H(\mathfrak{o}/\mathfrak{p})$ would contain a nontrivial $\Ad(k)\tau$-invariant submodule, and thus a nontrivial $\tau$-invariant submodule as $k\in K_H$; On the other hand, $\tau$ being subcyclic implies that there is no $\tau$-invariant submodule of $\mathbf{V}_H(\mathfrak{o}/\mathfrak{p})$, giving the required contradiction.


  We may thus find $1\leq i\leq m$ and $m+1 \leq j\leq n+1$ so that
  \begin{equation}\label{eqn:tracee_k-n+1-k-cdot-tau-in-mathfrako-mathfrakp.-}
    \trace(E_{ij} (\Ad(k)\tau)) \in \mathfrak{o} - \mathfrak{p}.
  \end{equation}
  We then put
  \begin{equation*}
    u := 1 + t \varpi^{\ell_i} E_{ij}
  \end{equation*}
  with $t \in \mathfrak{p}^{-1} \mathfrak{q}^2 - \mathfrak{q}^2$ to be determined.  Since $\ell_i > 0=\ell_j$, we have $u \in K(\mathfrak{q}^2)$, and so $\psiT(u) = 1$.  On the other hand, we have $a ^{-1} u a = 1 + t E_{ij}\in K(\mathfrak{q})$, and so
  \begin{equation*}
    \chi_\tau(h^{-1} u h) = \psiT (t \trace(E_{ij} (\Ad(k)\tau))).
  \end{equation*}
  By~\eqref{eqn:tracee_k-n+1-k-cdot-tau-in-mathfrako-mathfrakp.-}, we may choose $t$ so that the right hand side is $\neq 1$, giving the required contradiction to~\eqref{eqn:u-in-mathbfnf-cap-h-mathc-h-1-impl-psiu-=-chi_t-1-}.
\end{proof}

\begin{remark}
  Proposition~\ref{proposition:localized-whittaker-function-subcyclic-parameter-concentrates} may be understood as a non-archimedean counterpart of some results of~\cite[Part 3]{2021arXiv210915230N}.  More precisely, the cited reference shows (in an archimedean setting) that Whittaker functions lying in an irreducible representation that satisfy a support condition like~\eqref{eqn:wh-neq-0-implies-h-in-mathbfn_hf-k_hmathfrakq.-} and oscillate on their support in a specific way must be localized at a specific parameter $\tau$.  Both results have the effect of producing localized Whittaker functions that concentrate in the Kirillov model.  Here the invariance property is assumed and the localization property is deduced, whereas in \cite[Part 3]{2021arXiv210915230N}, the reverse implication is used.  Vectors having such invariance properties may be constructed as in \cite[Section 8]{2023arXiv2309.06314}.
\end{remark}

\subsection{Multiplicity one for localized vectors in generic representations}\label{Sec:WhittakerKirillov}
We retain the setup of the previous subsection.  Let
\begin{equation*}
  \mathcal{W}(\psiT) := C^\infty(\mathbf{N}(F) \backslash \mathbf{G}(F), \psiT)
\end{equation*}
denote the ``Whittaker space'' consisting of smooth functions $W : \mathbf{G}(F) \rightarrow \mathbb{C}$ satisfying $W(u g) = \psiT(u) W(g)$ for all $u \in \mathbf{N}(F)$.  We recall that a representation $\pi$ of $G$ is \emph{generic} if there is a nontrivial $\mathbf{G}(F)$-equivariant map $\pi \rightarrow \mathcal{W}(\psiT)$.  This notion does not depend upon the choice of $\psiT$.  If $\pi$ is generic and irreducible, then the space of such maps is one-dimensional (see \cite{MR348047}, \cite[\S1]{MR546599}), hence the image of any such map depends only upon $\pi$ and $\psiT$.  That image, denoted $\mathcal{W}(\pi,\psiT)$, is called the \emph{Whittaker model}.  Thus for $\pi$ generic and irreducible, we may identify $\pi$ with $\mathcal{W}(\pi,\psiT) \subseteq \mathcal{W}(\psiT)$.  We recall a basic result concerning the Kirillov model (see \cite[\S6.5]{MR748505}, \cite[\S10.2]{MR1999922}): for $\pi$ generic and irreducible, the restriction map
\begin{equation*}
  \mathcal{W}(\pi,\psiT) \rightarrow \{ \text{functions } \mathbf{H}(F) \rightarrow \mathbb{C}  \}
\end{equation*}
is injective.

\begin{proposition}
  Let $\pi$ be an irreducible generic representation of $\mathbf{G}(F)$.  Let $\tau \in \mathbf{M}(\mathfrak{o}/\mathfrak{q})$ be
    cyclic.  Then the space of vectors in $\pi$ that transform under $K(\mathfrak{q})$ via $\chi_\tau$ is at most one-dimensional.
\end{proposition}
\begin{proof}
  We may assume that $\pi = \mathcal{W}(\pi,\psiT)$.  We may assume also that $\tau$ is cyclic with respect to the basis $e_1,\dotsc,e_n$, since conjugating $\tau$ does not change the dimension of the space in question.  Let $W \in \pi$ transform under $K(\mathfrak{q})$ via $\chi_\tau$.  Then $W$ satisfies the hypotheses of Proposition \ref{proposition:localized-whittaker-function-subcyclic-parameter-concentrates}, whose conclusion shows that the restriction of $W$ to $\mathbf{H}(F)$ lies in the one-dimensional space of functions satisfying \eqref{eqn:wh-=-psiu-chi_tauy-w1.-}.  By the theory of the Kirillov model, $W$ is determined by this restriction, so the space of such $W$ is at most one-dimensional.
\end{proof}

\begin{remark}
  The space in question is nonzero precisely when $\tau$ is a ``regular parameter for $\pi$ at depth $\mathfrak{q}^2$'' in the sense of \cite[Definition 8.1]{2023arXiv2309.06314}.
\end{remark}

\section{Some analysis on $U \backslash G$}\label{sec:some-analysis-u-backslash-g}
In this section we take $\mathbf{U}$ to be the \emph{lower}-triangular unipotent subgroup, so that it is opposite to $\mathbf{N}$, the upper-triangular unipotent subgroup.  We study certain test functions relevant for defining pseudo-Eisenstein series.

\subsection{Preliminaries}\label{sec:cnjen3o6uy}

Writing $X(\mathbf{A})$ for the group of rational characters $\mathbf{A} \rightarrow \mathbf{G} \mathbf{L} _1$, we set
\begin{equation*}
  \mathfrak{a}^* := X(\mathbf{A}) \otimes_{\mathbb{Z}} \mathbb{R} \hookrightarrow \mathfrak{a}_{\mathbb{C}}^* := X(\mathbf{A}) \otimes_{\mathbb{Z}} \mathbb{C},
\end{equation*}
thus $\mathfrak{a}^* \cong \mathbb{R}^n$ and $\mathfrak{a}_{\mathbb{C}}^* \cong \mathbb{C}^n$.


Recall for $F$ a local field, we abbreviate $G := \mathbf{G}(F)$, etc.

\subsubsection{Schwartz spaces}\label{Sec:Schwartz}

We denote by $\mathcal{S}(U \backslash G)$ the space of Schwartz functions $f : U \backslash G \rightarrow \mathbb{C}$, and by $\mathcal{S}^e(U \backslash G)$ the subspace of $\mathcal{S}(U \backslash G)$ consisting of functions that are left-invariant under the maximal compact subgroup of $A$.  In the archimedean case, ``Schwartz'' means ``rapidly-decaying together with all derivatives'', while in the non-archimedean case, it means ``locally constant and compactly-supported'' (see~\cite[\S2.6]{2021arXiv210915230N} for details).

\subsubsection{Mellin components}\label{Sec:localMellin}
Each $s \in \mathfrak{a}_{\mathbb{C}}^*$ defines a continuous complex-valued character $|.|^s : a \mapsto \lvert a \rvert^s$ of $A$.  Elements $f \in \mathcal{S}^e(U \backslash G)$ may be described in terms of their Mellin components
\begin{equation*}
  f[s] : U \backslash G \rightarrow \mathbb{C},
\end{equation*}
defined for $s \in \mathfrak{a}_{\mathbb{C}}^*$ by
\begin{equation*}
  f[s](g) := \int_{a \in A} \delta_U^{1/2} (a) \lvert a \rvert^s f (a^{-1} g) \, d a.
\end{equation*}
This defines an element of the representation obtained by normalized induction from $|.|^s$.

\subsubsection{Normalized intertwining operators}\label{sec:normalized-intertwining-operators}
Let $\psi$ be any given nontrivial unitary character of $F$.  For $w$ in the Weyl group $W$ of $(\mathbf{G},\mathbf{A})$, we denote by
\begin{equation*}
  \mathcal{F}_{w,\psi} : L^2(U \backslash G) \rightarrow L^2(U \backslash G)
\end{equation*}
the normalized intertwining operators, as defined in~\cite[\S2.13]{2021arXiv210915230N}.  We denote by ${\mathcal{S}^e(U \backslash G)}^W$ the subspace fixed by each of the operators $\mathcal{F}_{w,\psi}$.  In general, this subspace depends upon $\psi$, but we will always take $\psi$ to be
\begin{itemize}
\item unramified (that is, trivial on $\mathfrak{o}$ but not on $\mathfrak{p}^{-1}$) when $F$ is non-archimedean, and
\item to be the standard choice, or possibly its inverse, when $F = \mathbb{R}$.
\end{itemize}
With these restrictions, the subspace is independent of $\psi$.

\subsection{Qualitative analysis}
Here we write $n$ for the rank of $\mathbf{V}$ and choose a basis $e_1,\dotsc,e_n$ for $\mathbf{V}$.  We record some analogues of the results of \cite[\S9, \S10]{2021arXiv210915230N}, established there over archimedean local fields; the proofs simplify in the present non-archimedean case.

Let $\beta \in C_c^\infty(A)$.  As in \cite[\S9.2]{2021arXiv210915230N}, we define a function
\begin{equation*}
  \mathcal{J} [\psiT, \beta ] : U \backslash G \rightarrow \mathbb{C}
\end{equation*}
to be supported on the ``open cell'' consisting of cosets represented by products $a n$, with $(a,n) \in A \times N$, and given there by
\begin{equation}\label{Eq:Jpsibeta}
  \mathcal{J} [\psiT, \beta ] (a n) = \delta_U^{1/2} (a) \beta (a) \psiT (n).
\end{equation}
\begin{example}\label{example:cnsgt7kf1o}
  Suppose that $n = 2$ and $\beta$ is the characteristic function of $K_A$.  We can identify $U \backslash G$ with $F^\times \times (F^2 - \{0\})$ via the map that assigns to a matrix $g = \left(
    \begin{smallmatrix}
      a&b\\
      c&d \\
    \end{smallmatrix}
  \right)$ the pair $(\det (g), (a,b))$.  Using the factorization
  \begin{equation}\label{eq:cnsgt7lxbp}
    g =
    \begin{pmatrix}
      1      & 0 \\
      c/a & 1 \\
    \end{pmatrix}
    \begin{pmatrix}
      a      & 0 \\
      0 & \det(g)/a \\
    \end{pmatrix}
    \begin{pmatrix}
      1      & b/a \\
      0 & 1 \\
    \end{pmatrix},
  \end{equation}
  we see that the function $\mathcal{J}[\psi_{\tilde{T}}, \beta]$ is given in these coordinates by
  \begin{equation*}
    U g \mapsto 1_{a \in \mathfrak{o}^\times} 1_{\det (g) \in \mathfrak{o}^\times} \psi_{\tilde{T}}(b/a).
  \end{equation*}
\end{example}
For $\gamma \in C_c^\infty(G)$, we then define $\mathcal{J} [\psiT, \beta, \gamma ] : U \backslash G \rightarrow \mathbb{C}$ by convolving on the right:
\begin{equation}\label{eqn:mathc-psi-beta-gamma-h-:=-int_g-in-g-gamma-g-mathc}
  \mathcal{J} [\psiT, \beta, \gamma ](h) := \int_{g \in G} \gamma (g) \mathcal{J} [\psiT, \beta ] (h g) \, d g.
\end{equation}
\begin{lemma}\label{lemma:compact-support-of-smoothened-jacquet-distributions}
  We have $\mathcal{J} [\psiT, \beta, \gamma ] \in C_c^\infty(U \backslash G)$.
\end{lemma}
\begin{proof}
  It is clear that $\mathcal{J}[\psiT,\beta,\gamma]$ is smooth, but less obvious that it is compactly-supported.  The latter follows from
  the analytic continuation of Jacquet integrals, as explained in the proof of \cite[Prop 9.4]{2021arXiv210915230N}.  The cited reference treats the archimedean case under some additional assumptions.  Those assumptions are used to verify that the Jacquet integral is of ``moderate growth in vertical strips''.  In the non-archimedean case, such arguments are not necessary, because the Mellin transform of $\beta$ is supported on finitely many cosets of the group of $K_A$-invariant characters of $A$.  The proof thus simplifies and gives what is stated.
\end{proof}
\begin{remark}
  Lemma~\ref{lemma:compact-support-of-smoothened-jacquet-distributions} will be employed only to streamline certain Mellin-theoretic arguments.  Incidentally, a minor adjustment to those arguments would reprove the conclusion of Lemma \ref{lemma:compact-support-of-smoothened-jacquet-distributions} for the $\beta$ and $\gamma$ ultimately of interest to us (see Proposition \ref{proposition:formulas-for-fak-fan}).  We have thus included Lemma~\ref{lemma:compact-support-of-smoothened-jacquet-distributions} primarily to simplify our presentation.
\end{remark}


\begin{lemma}\label{lemma:w-in-w-we-have-begin-mathc-psi-1-mathc-beta-gamma-}
  For $w \in W$, we have
  \begin{equation*}
    \mathcal{F}_{w,\psi^{-1}} \mathcal{J}[\psi,\beta,\gamma] = \mathcal{J}[{\psi,{}^w \beta,\gamma}],
  \end{equation*}
  where ${}^w \beta(a) := \beta(w^{-1} a)$.
\end{lemma}
\begin{proof}
  See \cite[Lemma 9.11]{2021arXiv210915230N}.
\end{proof}

\subsection{Quantitative setting and notation}
Recall that $\psiT : F \rightarrow \mathbb{C}^\times$ is a unitary character that is trivial on $\mathfrak{q}^2$ but not on $\mathfrak{p}^{-1} \mathfrak{q}^2$, and is also used to denote a non-degenerate character of $N$.

Let $\tau \in \mathbf{M}(\mathfrak{o}/\mathfrak{q})$.  Recall from \eqref{Eq:chitau} that, using $\psiT$ and $\tau$, we have defined a character $\chi_\tau$ of $K(\mathfrak{q}) / K(\mathfrak{q}^2)$.  We define the corresponding Hecke algebra idempotent
\begin{equation*}
  e_\tau := {\vol (K (\mathfrak{q} ))} ^{-1} 1 _{K(\mathfrak{q})} \chi _\tau^{-1} \in C_c^\infty(K(\mathfrak{q})) \subseteq C_c^\infty(\mathbf{G}(F)),
\end{equation*}
which projects in each representation of $G$ onto vectors transforming under $K(\mathfrak{q})$ via $\chi_\tau$.

We denote by $\theta \in \mathbf{M}(\mathfrak{o}/\mathfrak{q})$ the standard lower-triangular nilpotent Jordan block, as in \eqref{eqn:theta-:=-beginpmatrix-0--0--0-}.  Then
\begin{equation*}
  \chi_\theta(g) = \psiT(n) \quad \text{ for } g = a u n \in K_A(\mathfrak{q}) K_U(\mathfrak{q}) K_N(\mathfrak{q}).
\end{equation*}

\subsection{Construction of $f$}\label{sec:construction-f}
For the remainder of \S\ref{sec:some-analysis-u-backslash-g}, we set
\begin{equation}\label{eqn:f-:=-mathcaljpsi-beta-gamma.-}
  f := \mathcal{J}[\psiT,1_{K_A},e_\theta].
\end{equation}
Here $1_{K_A} \in C_c^\infty(A)$ denotes the characteristic function of the maximal compact subgroup $K_A$, while $e_\theta$ is the idempotent defined above.  By Lemma \ref{lemma:compact-support-of-smoothened-jacquet-distributions}, we have $f\in \mathcal{S}^e(U \backslash G)$.

\subsection{Invariance under normalized intertwining operators}
\begin{lemma}\label{lemma:f-mathcalf_w-psi-1-invariant-each-w-in-w.-}
  $f$ is $\mathcal{F}_{w,\psiT^{-1}}$-invariant for each $w \in W$.
\end{lemma}
\begin{proof}
  This follows from Lemma \ref{lemma:w-in-w-we-have-begin-mathc-psi-1-mathc-beta-gamma-} and the fact that $1_{K_A}$ is $W$-invariant.
\end{proof}

\subsection{Explicit formulas}\label{sec:cnjgx9803p}

\begin{proposition}\label{proposition:formulas-for-fak-fan}
  For $(a,k) \in A \times K$, we have
  \begin{equation}\label{eqn:fa-k--=-1_k_aa-1_kmathfrakqk-chi_thetak.-}
    f(a k ) = 1_{K_A}(a) 1_{K_A K_U K(\mathfrak{q})}(k) \chi_{\theta}(k'),
  \end{equation}
  where $k=a'u'k'$ for $a'\in K_A, u'\in K_U, k'\in K(\mfq)$ in the support.  For $(a,n) \in A \times N$, we have
  \begin{equation}\label{eqn:fa-n-=-1_k_aa-1_k_nmathfrakqn-psin.-}
    f(a n) = 1_{K_A}(a) 1_{K_N(\mathfrak{q})}(n) \psiT(n).
  \end{equation}
\end{proposition}
\begin{example}
  When $n = 2$, with coordinates $U g \leftrightarrow (\det(g), (a,b))$ as in Example \ref{example:cnsgt7kf1o}, the function $f$ is given by
  \begin{equation*}
    U g \mapsto 1_{a \in \mathfrak{o}^\times} 1_{\det(g) \in \mathfrak{o}^\times} 1_{\mathfrak{q}}(b) \psi_{\tilde{T}}(b/a).
  \end{equation*}
\end{example}

The proof is contained in the following lemmas.

\begin{lemma}\label{lemma:f-left-invar-under-k_a-u-transf-right-under-kmathf-equivariance-properties-f}
  $f$ is left-invariant under $K_A U$ and transforms on the right under $K(\mathfrak{q})$ via $\chi_\theta$.  Moreover,
  \begin{equation*}
    \chi_{\theta} |_{K_N(\mathfrak{q})} = \psiT|_{K_N(\mathfrak{q})}.
  \end{equation*}
\end{lemma}
\begin{proof}
  Each assertion is clear by construction.
\end{proof}

By Lemma \ref{lemma:compact-support-of-smoothened-jacquet-distributions}, we see that for each $g \in G$, the function $A \ni a \mapsto f(a g)$ lies in $C_c^\infty(A)$.  These functions are thus determined by their Mellin transforms, which are supported on the subgroup of unramified (i.e., $K_A$-invariant) characters.  For such a character $\eta$, the Mellin component
\begin{equation}\label{eqn:f_etag-:=-int_a-in-a-delta_u12-eta--1-a-f-a-g--}
  f[\eta](g) :=
  \int_{a \in A} (\delta_U^{1/2} \eta )^{-1} (a)
  f (a g ) \, d a
\end{equation}
lies in the normalized induced representation $I(\eta) := \Ind_{Q}^G(\eta)$, where we recall that $Q = A U$.  More precisely, it lies in the subspace $I(\eta)_{\theta}$ of that representation on which $K(\mathfrak{q})$ acts via $\chi_\theta$.

\begin{lemma}\label{Lem:localmultipicityone-Ieta}
  Let $\eta$ be an unramified character of $A$.  Then $I(\eta)_\theta$ is one-dimensional, with each element supported on $U A K(\mathfrak{q})$.
\end{lemma}
\begin{proof}
  By the Iwasawa decomposition $G = Q K$, restriction to $K$ defines an isomorphism from $I(\eta)_\theta$ to the subspace of $\Ind_{K_Q}^{K} \eta$ on which $K(\mfq)$ acts by $\chi_\theta$.  By \cite[Proposition 8.19(iii)]{2023arXiv2309.06314}, the latter subspace is one-dimensional.  This establishes everything but the support property.  To that end, we verify readily that for $(b,k) \in Q \times K(\mathfrak{q})$, the quantity $\delta_U^{1/2} \eta(b) \chi_\theta(k)$ depends only upon the product $b k$, and that the resulting function $G \rightarrow \mathbb{C}$ supported on $Q K(\mathfrak{q})$ defines a nonzero element of $I(\eta)_\theta$.
\end{proof}

\begin{lemma}\label{lemma:fa-k-=-0-unless-a-k-in-k_a-times-kmathfrakq.-}
  $f(a k) = 0$ unless $(a,k) \in K_A \times K_A K_U K(\mathfrak{q})$.
\end{lemma}

\begin{proof}
  By Lemma \ref{Lem:localmultipicityone-Ieta}, we see that each Mellin component $f[\eta]$ is supported on the same double coset $UAK(\mfq)$.  By Mellin inversion, $f$ is also supported on $UAK(\mfq)$.  Thus $f(ak)\neq 0$ implies that $k\in a^{-1}UAK(\mfq)\cap K=K_AK_UK(\mfq)$.

  By the support condition that we just established in the $k$ variable together with the equivariance property under $K(\mathfrak{q})$ noted in Lemma \ref{lemma:f-left-invar-under-k_a-u-transf-right-under-kmathf-equivariance-properties-f}, we may reduce to the case $k=1$.  It follows then from the definitions \eqref{eqn:mathc-psi-beta-gamma-h-:=-int_g-in-g-gamma-g-mathc} and \eqref{eqn:f-:=-mathcaljpsi-beta-gamma.-} that if $f(a) \neq 0$, then there exists $g \in K(\mathfrak{q})$ with $a g \in U K_A N$.  By the Bruhat decomposition, we may write $g = u b n$ with $(u, b, n) \in K_U(\mathfrak{q}) K_A(\mathfrak{q}) K_N(\mathfrak{q})$.  Then
  \begin{equation*}
    a g = (a u a^{-1}) a b n \in U K_A N.
  \end{equation*}
  By the uniqueness of the Bruhat decomposition, it follows that $a b \in K_A$, hence $a \in K_A$, as required.
\end{proof}

\begin{lemma}\label{lemma:let-n-in-n-have-iwas-decomp-n-=-b-k-with-b-in-u}
  Let $n \in N$ have the Iwasawa decomposition $n = b k$, with $b \in A U$ and $k \in K$.  Suppose that $k \in K(\mathfrak{q})$.  Then $n \in K_N(\mathfrak{q})$.
\end{lemma}
\begin{proof}
  We have $k = b^{-1} n$.  By the Bruhat decomposition, this identity determines $b$ and $n$ from $k$.  On the other hand, $k$ admits an Iwahori factorization relative to $K(\mathfrak{q}) = K_Q(\mathfrak{q}) K_N(\mathfrak{q})$.  It follows that $n \in K_N(\mathfrak{q})$.
\end{proof}
\begin{lemma}\label{Lem:fansupport}
  $f(a n) = 0$ unless $(a,n) \in K_A \times K_N(\mathfrak{q})$.
\end{lemma}
\begin{proof}
  This follows from Lemmas \ref{lemma:fa-k-=-0-unless-a-k-in-k_a-times-kmathfrakq.-} and \ref{lemma:let-n-in-n-have-iwas-decomp-n-=-b-k-with-b-in-u}.  Indeed we may write $n=bk$ for $b\in AU$ and $k\in K$.  If $f(an)\neq 0$, we have by Lemma \ref{lemma:fa-k-=-0-unless-a-k-in-k_a-times-kmathfrakq.-} $k\in K_AK_UK(\mfq)$.  This implies that there exists $b'\in K_AK_U$ such that $(b')^{-1}k=(b')^{-1}b^{-1}n\in K(\mfq)$. Then Lemma \ref{lemma:let-n-in-n-have-iwas-decomp-n-=-b-k-with-b-in-u} implies that $n\in K_N(\mfq)$. The rest is clear.
\end{proof}

\begin{lemma}\label{Lem:f1}
  $f(1) = 1$.
\end{lemma}
\begin{proof}
  Recall that by definition,
  \begin{equation*}
    f(1) = \int_{g \in G} \gamma(g) \mathcal{J} [\psiT, \beta ](g) \, d g.
  \end{equation*}
  By construction, $\gamma(g)$ vanishes unless $g \in K(\mathfrak{q})$, in which case $g$ admits the Iwahori factorization $g = u a n$, with
  \begin{equation}\label{eqn:iwahori-components-uan}
    (u,a,n) \in K_U(\mathfrak{q}) \times K_A(\mathfrak{q}) \times K_N(\mathfrak{q}).
  \end{equation}
  By the construction of $\theta$, we then have
  \begin{equation*}
    \gamma(g) = {\vol (K (\mathfrak{q} ))} ^{-1} \psiT^{-1}(n).
  \end{equation*}
  On the other hand,
  \begin{equation*}
    \mathcal{J} [\psiT, \beta ](g) = 1_{K_A}(a) \psiT(n) = \psiT(n);
  \end{equation*}
  here we have used the definition \eqref{Eq:Jpsibeta} and that $\delta_U \equiv 1$ on $K_A$.  Thus for $g \in K(\mathfrak{q})$, we have
  \begin{equation*}
    \gamma(g) \mathcal{J}[\psi_{\bar{T}}, \beta](g) = \vol(K(\mathfrak{q}))^{-1}.
  \end{equation*}
  Integrating over $g$ gives the required identity.
\end{proof}
Proposition \ref{proposition:formulas-for-fak-fan} now follows from Lemma \ref{lemma:f-left-invar-under-k_a-u-transf-right-under-kmathf-equivariance-properties-f}, \ref{lemma:fa-k-=-0-unless-a-k-in-k_a-times-kmathfrakq.-}, \ref{Lem:fansupport}, \ref{Lem:f1},

\subsection{Extension}

\begin{lemma}\label{lemma:let-j_th-leq-k-denote-inverse-image-centr-mathbfg}
  Let $J_\theta \leq K$ denote the inverse image of the centralizer $\mathbf{G}_\theta(\mathfrak{o}/\mathfrak{q})$.  There is a unique extension of $\chi_\theta$ to a character $\tilde{\chi}_\theta$ of $J_\theta$ such that $f$ transforms on the right under $J_\theta$ via $\tilde{\chi}_\theta$.
\end{lemma}
\begin{proof}
  This follows from combining the argument for existence as in \cite[Lemma 8.12]{2023arXiv2309.06314} and the uniqueness of such elements in Mellin components in Lemma \ref{Lem:localmultipicityone-Ieta}.

  Indeed, a specialization of \cite[Lemma 8.12]{2023arXiv2309.06314} reads as follows (note that $\theta \in \mathbf{M}(\mathfrak{o} / \mathfrak{q})$ is cyclic).  Let $\pi$ be a representation of $\mathbf{G}(\mathfrak{o})$ that contains a nonzero vector that transforms under $K(\mathfrak{q})$ via the character $\chi_\theta$.  Then there is an extension of $\chi_\theta$ to a character $\tilde{\chi}_\theta$ of $J_\theta$ and a nonzero vector $v \in \pi$ that transforms under $J _\theta$ via $\tilde{\chi}_\theta$.

  We can decompose $f$ into its Mellin components $f [\eta] \in I(\eta)$ as in \eqref{eqn:f_etag-:=-int_a-in-a-delta_u12-eta--1-a-f-a-g--}.  By Lemma \ref{lemma:f-left-invar-under-k_a-u-transf-right-under-kmathf-equivariance-properties-f}, each component transforms on the right under $K(\mathfrak{q})$ via $\chi_\theta$.  By Lemma \ref{Lem:f1}, each component satisfies $f [\eta] (1) = 1$, hence is nonzero.  By \cite[Lemma 8.12]{2023arXiv2309.06314}, there is for each $\eta$ an extension $\tilde{\chi}_{\theta,\eta}$ such that $I(\eta)$ contains a vector that transforms under $J_\theta$ via $\tilde{\chi}_{\theta, \eta}$.  By the one-dimensionality of $I(\eta)_\theta$ noted in Lemma \ref{Lem:localmultipicityone-Ieta}, we see that the vector in question must be $f[\eta]$ itself.  Since $I(\eta)|_{K}$ is independent of the unramified character $\eta$, we see that $\tilde{\chi}_{\theta, \eta}$ is likewise independent -- call it $\tilde{\chi}_\theta$.  By Mellin inversion, we deduce that $f$ likewise transforms on the right via $\tilde{\chi}_\theta$.  Uniqueness is clear, using again that $f \neq 0$.
\end{proof}

\subsection{$L^2$-estimates}
\begin{lemma}
  We have $\int_{U \backslash G} \lvert f \rvert^2 \asymp T^{- \dim(N)/2}$.
\end{lemma}
\begin{proof}
  This is immediate from \eqref{eqn:fa-n-=-1_k_aa-1_k_nmathfrakqn-psin.-} and the formula for integration in Bruhat coordinates.
\end{proof}


%

%
%

\section{Main local result in non-compact case}\label{Sec:mainlocal}
\subsection{Local Zeta integral}
%
For $W \in \mathcal{W}(\pi,\psi)$ and $f \in \mathcal{I}(\eta)$, we define
\begin{equation*}
  Z(W,f) = \int_{N_H \backslash H} W \cdot W[f, \psi^{-1}],
\end{equation*}
which appears as the local component of the integral representation of special value of Rankin--Selberg L-function on $\GL_{n+1}\times \GL_n$ (see \cite{MR620708},\cite{MR701565}).  Such integrals typically requires interpretation via meromorphic continuation, but when $W|_{N_H \backslash H}$ is compactly-supported (as it is for us), the above integral converges for all $s$ and defines an entire function of $s$.

\label{sec:cqzbt5us6y}\subsection{p-adic Formulation}\label{sec:cqzbt5u05j}
We recall \cite[Definition 8.1]{2023arXiv2309.06314}:
\begin{definition}\label{Def:regular}
  We say that $\pi$ is \emph{regular at depth $\mfq^2$} if there is some $\tau \in \mathbf{M}(\mathfrak{o}/\mfq)$ that is cyclic (i.e., admits a cyclic vector) and such that $\pi$ contains a nonzero vector $v$ that transforms under $K (\mfq )$ via the character $\chi_\tau$ (see \eqref{Eq:chitau}), i.e.,
  \begin{equation}\label{eqn:g-v-=-chi_taug-v-quad-text-all--g-in-kmathfrakq.-}
    g v = \chi_\tau(g) v \quad \text{ for all } g \in K(\mfq).
  \end{equation}
  In that case, we refer to $\tau$ as a \emph{regular parameter for} $\pi$ \emph{at depth} $\mfq^2$, to its characteristic polynomial $P \in (\mathfrak{o}/\mfq)[X]$ as a \emph{polynomial for} $\pi$ \emph{at depth} $\mfq^2 $.
\end{definition}
Here we introduce slightly stronger condition:
\begin{definition}\label{definition:we-say-that-repr-pi-mathbfgf-has-emph-depth-mathfr-uniform-depth}
  We say that $\pi$ is \emph{uniform at depth} $\mathfrak{q}^2 $ if it satisfies the condition of regularity at depth $\mfq^2$ with $\det(\tau)$ (or equivalently, $P(0)$) being a unit.
\end{definition}
\begin{remark}
  $\pi$ being uniform at depth $\mathfrak{q}^2$ is equivalent to the pair $(\pi,\sigma)$ being stable at depth $\mathfrak{q}^2$ as in \cite[Definition 8.6] {2023arXiv2309.06314} for any unramified representation $\sigma$ of $H$.
\end{remark}
\begin{example}
  Let $\chi$ be a character of $\mathfrak{o}^\times$ with conductor $\mfq^2$, i.e., with trivial restriction to $\mfq^2$ but not to $\mathfrak{p}^{-1} \mfq^2$.  Then $\chi$ is regular and uniform at depth $\mfq^2$.

  The supercuspidal representations satisfying the assumptions indicated in \cite[Section 8.8]{2023arXiv2309.06314} are regular and uniform at depth $\mfq^2$.

  Being regular and uniform at depth $\mfq^2$ is closed under parabolic induction (see \cite[Proposition 8.19]{2023arXiv2309.06314}).
\end{example}

\begin{remark}
  Let $\pi$ be a principal series representation induced from characters $\chi_1,\dotsc,\chi_n$, where the conductors of the $\chi_i$ are at most $\mfq^2$ but not all equal to $\mfq$.  One can show then that $\pi$ is regular, but not uniform, at depth $\mfq^2$.
\end{remark}

\begin{lemma}\label{lem:cq0ikz837w}
  Let $\pi$ be a fixed irreducible smooth supercuspidal representation of $G(\mathbb Q_p)$. Let $\chi$ be a character of $ \mathbb Q_p^\times$ with conductor $\mfq^2$ where $\mfq=\mathfrak{p}^m$ for $m$ sufficiently large. Then $\pi\times \chi$ is uniform at depth $\mfq^2$.
\end{lemma}
\begin{proof}
  Recall first from \cite[Lemma 2.14]{assing2024localanalysiskuznetsovformula} that when $m$ is large enough, there exist a vector $v'\in \pi$ and an element $\tau'\in \mathbf{M}(\mathfrak{o}/\mathfrak{p})$ which is regular and nilpotent, such that $K_G(\mathfrak{p}^{2m-1})$ acts on $v'$ by $\chi_{\tau'}$.  That is, for all $1+x\in K_G(\mathfrak{p}^{2m-1})$,
  \begin{equation}\label{Eq:chitau}
    \pi(1+x)v'=\psi_{\tilde{T}}(\trace(x\tau'))v'.
  \end{equation}
  Note that $\tau'$ being regular and nilpotent is equivalent to that $\tau'$ is conjugated to a single Jordan block with eigenvalues $0$ by $K_G$.


  Note that $K_G(\mathfrak{p}^m)/K_G(\mathfrak{p}^{2m})$ is abelian.  Using a similar argument as in \cite[Lemma 8.12]{2023arXiv2309.06314} we further conclude that there exists $\tau\in \mathbf{M}(\mathfrak{o}/\mfq)$ and a vector $v\in \pi$ such that $K_G(\mathfrak{p}^{m})$ acts on $v$ by $\chi_{\tau}$, extending \eqref{Eq:chitau}.  In particular $\tau\equiv \tau' \mod \mathfrak{p}$, and by \cite[Example 3.5]{2023arXiv2309.06314} $\tau$ is a regular parameter for $\pi$ at depth $\mfq^2$.

  For $\chi$ with conductor $\mfq^2$, let $\alpha_\chi\in \mathbb Z_p^\times$ be the element satisfying
  \[
    \chi(1+x)=\psi_{\tilde{T}}(\alpha_\chi x).
  \]
  Consider now the natural map for, for example, the representations in Whittaker models
  \begin{align*}
    \iota :\pi&\rightarrow \pi\times \chi\\
    v &\mapsto v\chi(\det(g)).
  \end{align*}
  Then $\iota(v)$ is a vector in $\pi\times \chi$ that transforms under $K_G(\mathfrak{p}^{m})$ via $\chi_{\tau+\alpha_\chi}$, with $\tau+\alpha_\chi$ still a regular parameter for $\pi\times \chi$ at depth $\mfq^2$.   Furthermore $\tau+\alpha_\chi\mod\mathfrak{p}$ is invertible by being the sum of an invertible central element with a nilpotent element. Thus $\pi\times \chi$ is uniform at depth $\mfq^2$.
\end{proof}

\begin{theorem}\label{theorem:main-local-noncompact}
  Let $\pi$ be a generic irreducible unitary representation of $G$ that is uniform at depth $\mathfrak{q}^2$.  Let $\psi$ be an unramified nondegenerate unitary character of $N$.  There exists
  \begin{itemize}
  \item an idempotent $\omega \in C_c^\infty(K)$,
  \item a Whittaker function $W \in \mathcal{W}(\pi,\psi)$,
  \item a compact open subgroup $J_H$ of $K_H$ containing the principal congruence subgroup $K_H(\mathfrak{q})$,
  \item a character $\chi_H$ of $J_H$ with trivial restriction to $K_H(\mathfrak{q}^2)$, and
  \item a function $f \in {C_c^\infty(U_H \backslash H)}$
  \end{itemize}
  with the following properties.

  Define
  \begin{equation*}
    \omega ^\sharp (g) := \int _{z \in \mathbf{Z}(F)} \pi |_{Z}(z) \omega(z g) \, d z,
  \end{equation*}
  where $\pi|_{Z}$ denotes the central character of $\pi$.

  \begin{enumerate}
    [(i)]
  \item\label{enumerate:cqzbt42t3t} $W$ and $f$ satisfy the $L^2$-normalizations
    \begin{equation}\label{eqn:int_n_h-backslash-h-lvert-w-rvert2-=-1-l2-normalization-W}
      \int_{N_H \backslash H} \lvert W \rvert^2 = 1,
    \end{equation}
    \begin{equation}\label{eqn:int_u_h-backslash-h-lvert-f-rvert2-=-1.-l2-normalization-f}
      \int_{U_H \backslash H} \lvert f \rvert^2  = 1.
    \end{equation}
  \item\label{enumerate:piomega-w-=-w.-} $\pi(\omega) W = W$.
  \item\label{enumerate:we-have-int-_h-lvert-omega-sharp-rvert-ll-t-n2.-} We have $\int _{H} \lvert \omega ^\sharp \rvert \ll T ^{n/2}$.
  \item\label{enumerate:f-invar-norm-intertw-oper-mathc-psi-w-in-w_h-defin} $f$ is invariant by the normalized intertwining operators $\mathcal{F}_{w,\psi}$ ($w \in W_H$) defined in \cite[\S2.13.2]{2021arXiv210915230N}.
  \item\label{enumerate:f-left-invar-under-k_a-right-invar-under-k_hm-supp} $f$ is left-invariant under $K_{A_H}$, right-invariant under $K_H(\mathfrak{q}^2)$, and supported on $U_H K_{A_H} \tilde{T}^{-\rho_{U_H}^\vee} K_H(\mathfrak{q})$.
  \item\label{enumerate:we-have-fg-h-=-chi_hh-fg-all-g-h-in-h-times-j_h.-} We have $f(g h) = \chi_H(h) f(g)$ for all $(g,h) \in H \times J_H$.
  \item\label{enumerate:each-s-in-mathfr-mathbbc-we-have-begin-zw-fs-=-c-tzeta-integrals} For each $s \in \mathfrak{a}_{H,\mathbb{C}}^*$, we have
    \begin{equation*}
      Z(W, f[s]) = c T^{\rank(G) \trace(s) / 2 - \rank(H)^2 / 4}=c T^{(n+1)\trace(s) / 2 - n^2 / 4},
    \end{equation*}
    where $c > 0$ depends only upon measure normalizations.
  \item\label{enumerate:let-psi_1-psi_2-:-k_h-right-mathbbc-be-funct-satis} Let $\Psi_1, \Psi_2 : K_H \rightarrow \mathbb{C}$ be functions satisfying
    \begin{equation*}
      \lvert \Psi_j(g z) \rvert = \left\lvert \Psi_j(g) \right\rvert \text{ for all } (g,z) \in K_H \times  J_H.
    \end{equation*}
    Let
    \begin{equation*}
      \gamma \in G - H Z
    \end{equation*}
    Then
    \begin{equation*}
      \int _{x, y \in K_H} \left\lvert \Psi_1(x) \Psi_2(y) \omega ^\sharp (x ^{-1} \gamma y) \right\rvert
      \ll
      \Delta
      T^{n/2}
      \lVert \Psi_1 \rVert_{L^2}
      \lVert \Psi_2 \rVert_{L^2},
    \end{equation*}
    where
    \begin{equation*}
      \Delta := \frac{1}{1 + T^{1/2} d_H(\gamma)}
      + \frac{d_H(\gamma)^\infty }{T^{1/4} }.
    \end{equation*}
  \end{enumerate}
\end{theorem}

\begin{remark}
  We have the following rough analogies between the conditions stated above and those in
  \cite[Theorem 3.1]{2021arXiv210915230N}:
  \begin{itemize}
  \item Conditions \eqref{enumerate:cqzbt42t3t} and \eqref{enumerate:f-invar-norm-intertw-oper-mathc-psi-w-in-w_h-defin} are analogous to unnumbered hypotheses in \emph{loc. cit.}
  \item Condition \eqref{enumerate:piomega-w-=-w.-} is like (v) in \emph{loc. cit.}, but gives an exact reproducing property rather than an approximate one.
  \item Condition \eqref{enumerate:we-have-int-_h-lvert-omega-sharp-rvert-ll-t-n2.-} is similar to (vi) in \emph{loc. cit.}.
  \item Condition \eqref{enumerate:f-left-invar-under-k_a-right-invar-under-k_hm-supp} is roughly analogous to (iii) in \emph{loc. cit.}.  The latter bounds certain seminorms of $f$ which, as explained \cite[Lemma 20.2]{2021arXiv210915230N}, says that $f$ has invariance and support properties that may be understood as approximate analogues of those noted here.
  \item Condition \eqref{enumerate:we-have-fg-h-=-chi_hh-fg-all-g-h-in-h-times-j_h.-} is roughly analogous to the combination of (i) and (ii) in \emph{loc. cit.}.  Both conditions relate $f$ to its convolution with respect to some character of a compact open subgroup.  Here, the relation is exact with respect to $J_H$, while in \emph{loc. cit.}, it is approximate (and with respect to just the central directions in $J_H$, which suffice for the study of archimedean and depth aspects).
  \item Condition \eqref{enumerate:each-s-in-mathfr-mathbbc-we-have-begin-zw-fs-=-c-tzeta-integrals}, an exact identity for local zeta integrals, is roughly analogous to the upper and lower bounds established in (iv) of \emph{loc. cit.}.
  \item  Condition \eqref{enumerate:let-psi_1-psi_2-:-k_h-right-mathbbc-be-funct-satis} is roughly analogous to (vii) in \emph{loc. cit.} (but here we postulate invariance with respect to $J_H$, rather than just its central directions).
  \end{itemize}
\end{remark}

\subsection{Setup}

Recall that $\pi$ is assumed to be uniform at depth $\mathfrak{q}^2$.  By definition, there exists a cyclic element $\tau \in \mathbf{M}(\mathfrak{o}/\mathfrak{q})$ with unit determinant such that $\pi$ admits a nonzero vector that transform under $K(\mathfrak{q})$ via $\chi_\tau$ as in \eqref{Eq:chitau}.  It is clear that any conjugate of $\tau$ has the same property concerning $\pi$.


\begin{lemma}\label{lemma:cqzbt6bk3w}
  Let $\tau \in \mathbf{M}(\mathfrak{o} /\mathfrak{q})$, as above, be a cyclic element with unit determinant.  Then, possibly after conjugating $\tau$ by an element of $\mathbf{G}(\mathfrak{o})$, the following conditions hold.
  \begin{enumerate}[(i)]
  \item $\tau_H = \theta$
  \item $\tau$ is stable.
  \end{enumerate}
\end{lemma}
\begin{proof}
  Let $e_1,\dotsc,e_n$ be a basis for $\mathbf{V}_H$, so that $e_1,\dotsc,e_{n+1} := e$ is a basis for $\mathbf{V}$.  By \cite[Lemma 3.6]{2023arXiv2309.06314}, $\tau$ admits a cyclic basis.  By conjugating $\tau$, we may assume that $e_1,\dotsc,e_{n+1}$ is a cyclic basis for $\tau$, e.g., for $n = 2$,
  \begin{equation*}
    \tau =
    \begin{pmatrix}
      0 & 0 & \ast \\
      1 & 0 & \ast \\
      0 & 1 & \ast \\
    \end{pmatrix}
    .
  \end{equation*}
  In particular, the upper-left $n\times n$ block $\tau_H$ of $\tau$ coincides with the nilpotent Jordan block $\theta \in \mathbf{M}_H(\mathfrak{o})$ characterized by $\theta e_i = e_{i+1}$ for $i < n$ and $\theta e_n = 0$.  The element $\tau$ is then stable by the equivalent description in \cite[Lemma 4.6]{2023arXiv2309.06314}.
\end{proof}

We henceforth suppose that $\tau$ satisfies the conclusions of Lemma \ref{lemma:cqzbt6bk3w}.

\subsection{Construction of $W$}

We identify $\pi$ with its Whittaker model $\mathcal{W}(\pi,\psi)$.  By hypothesis on $\pi$, there is a nonzero vector $W \in \pi$ that transforms on the right under $K(\mathfrak{q})$ via $\chi_\tau$.  We now describe the support and invariance properties of the restriction of $W$ to $H$.  To that end, define $a_T \in A_H \leq H$ by
\begin{equation*}
  a_T := \diag(\tilde{T}^{\rank(H)}, \tilde{T}^{\rank(H) - 1}, \dotsc, \tilde{T}, 1).
\end{equation*}
\begin{lemma}\label{Lem:supportofW}
  For $W$ as above:
  \begin{enumerate}[(i)]
  \item For $h \in H$, we have $W(a_T h) \neq 0$ only if $h \in N_H K_H(\mathfrak{q})$.
  \item For $(n,h) \in N_H \times K_H(\mathfrak{q})$, we have
    \begin{equation}\label{eqn:wa_t-n-h-=-psi_tildetn-chi_thetah-wa_t.-}
      W(a_T n h) = \psi_{\tilde{T}}(n) \chi_\theta(h) W(a_T).
    \end{equation}
  \end{enumerate}
\end{lemma}
\begin{proof}
  The element $a_T$ conjugates $\psi|_{N_H}$ to $\psi_{\tilde{T}}|_{N_H}$, so the left translate $W (a_T \bullet)$ lies in the $\psi_{\tilde{T}}$-Whittaker model $\mathcal{W}(\pi,\psi_{\tilde{T}})$ and transforms on the right under $K(\mathfrak{q})$ via $\chi_\tau$.  By Proposition \ref{proposition:localized-whittaker-function-subcyclic-parameter-concentrates}, we see that $W(a_T \bullet)$ satisfies the stated support condition.  The transformation property \eqref{eqn:wa_t-n-h-=-psi_tildetn-chi_thetah-wa_t.-} is clear by construction.
\end{proof}
We may thus normalize $W$ so that
\begin{equation*}
  W(a_T) > 0
\end{equation*}
and $W$ satisfies the $L^2$-normalization property \eqref{eqn:int_n_h-backslash-h-lvert-w-rvert2-=-1-l2-normalization-W}.

Recall that, by the theory of the Kirillov model (see Section \ref{Sec:WhittakerKirillov}), $W$ is determined by its restriction to $H$.  By the Bruhat decomposition for $H$, the restriction of $W$ to $H$ is determined by the further restriction of $W$ to $Q_H$, the lower-triangular subgroup of $H$.
\begin{lemma}\label{lemma:restr-w-q_h-posit-mult-char-funct-a_t-k_q_h}
  The restriction of $W$ to $Q_H$ is a positive multiple of the characteristic function of $a_T K_{Q_H}(\mathfrak{q})$.
\end{lemma}
\begin{proof}
  Since $Q_H$ consists of matrices of form
  \begin{equation*}
    \begin{pmatrix}
      \ast & 0 & 0 \\
      \ast & \ast & 0 \\
      0 & 0 & 1 \\
    \end{pmatrix}
    ,
  \end{equation*}
  we see that $\chi_\theta$ has trivial restriction to $K_{Q_H}(\mathfrak{q})$, and $N_H K_H(\mfq)\cap Q_H=K_{Q_H}(\mfq)$ by Bruhat decomposition.  The conclusion then follows from Lemma \ref{Lem:supportofW}.
\end{proof}

\subsection{Construction of $\omega$}

Recall from Lemma \ref{lemma:let-j_th-leq-k-denote-inverse-image-centr-mathbfg} that, with $J_\tau \leq K$ the preimage of $\mathbf{G}_\tau(\mathfrak{o}/\mathfrak{q})$, there is a unique extension $\tilde{\chi}_\tau$ of $\chi_\tau$ to a character of $J_\tau$ such that $W$ transforms on the right under $J_\tau$ via $\tilde{\chi}_\tau$.

We define
\begin{equation*}
  \omega := {\vol(J_\tau)}^{-1} \tilde{\chi}_\tau^{-1} \in C_c^\infty(K).
\end{equation*}
This is an idempotent with $\pi(\omega) W = W$, as required by Theorem \ref{theorem:main-local-noncompact}, part \eqref{enumerate:piomega-w-=-w.-}.

The $L^1(H)$-estimate asserted in Theorem \ref{theorem:main-local-noncompact}, part \eqref{enumerate:we-have-int-_h-lvert-omega-sharp-rvert-ll-t-n2.-} is verified exactly as in the ``compact'' case (see \cite[Section 9.4]{2023arXiv2309.06314}, with $Q$/$J_G$ there corresponding to $T^{1/2}$/$J_\tau$ here).

\subsection{Construction of \texorpdfstring{$f$}{f}}\label{Sec:constructfp}
We first construct an element $f_0 \in C_c^\infty(U_H \backslash H)$ by
\begin{equation*}
  \overline{f_0} = \mathcal{J} [\psi_{\tilde{T}}, 1_{K_A}, e_\theta],
\end{equation*}
as in \S\ref{sec:construction-f}.  We then define $f \in C_c^\infty(U_H \backslash H)$ to be the normalized left translate
\begin{equation*}
  f(g) := c_1 f_0 \left(\tilde{T}^{\rho_{U_H}^\vee} g\right),
\end{equation*}
where $c_1 > 0$ is chosen to achieve the normalization \eqref{eqn:int_u_h-backslash-h-lvert-f-rvert2-=-1.-l2-normalization-f}.

\begin{lemma}
  $f$ is invariant under the normalized intertwining operators $\mathcal{F}_{w,\psi}$ ($w \in W_H$) defined in \cite[\S2.13.2]{2021arXiv210915230N}.
\end{lemma}
\begin{proof}
  This follows from Lemma \ref{lemma:f-mathcalf_w-psi-1-invariant-each-w-in-w.-}, taking into account that we have simultaneously passed from $\psi$ to $\psi_{\tilde{T}}$ and translated on the left by $\tilde{T}^{\rho_{U_H}^\vee }$.  A proof for the archimedean analogue is recorded in \cite[Lemma 10.9]{2021arXiv210915230N} and the same argument applies here.

\end{proof}

Thus part \eqref{enumerate:f-invar-norm-intertw-oper-mathc-psi-w-in-w_h-defin} of Theorem \ref{theorem:main-local-noncompact} holds.

The invariance and support properties of $f$ asserted in Theorem \ref{theorem:main-local-noncompact}, part \eqref{enumerate:f-left-invar-under-k_a-right-invar-under-k_hm-supp} are a consequence of the following lemma:
\begin{lemma}\label{Lem:fu-n-=-c_1-1_k_a-left-tild-right-1_k_n_hm-psi}
  For $f$ as above:
  \begin{enumerate}[(i)]
  \item For $(u,a,k)\in U_H \times A_H \times K_H$, we have
    \begin{equation}\label{eqn:fu-n-=-c_1-1_k_a-left-tild-right-1_k_n_hm-psi}
      f(u a k) = c_1 1_{K_A}\left( \tilde{T}^{\rho_{U_H}^\vee} a \right)1_{K_{A_H} K_{U_H} K_H(\mathfrak{q})}(k) \chi_{\theta}^{-1}(k')
    \end{equation}
    where $k=a'u'k'$ for $a'\in K_{A_H}, u'\in K_{U_H}, k'\in K_H(\mfq)$ in the support.
  \item For $(u,a,n) \in U_H \times A_H \times N_H$, we have
    \begin{equation*}
      f(u a n) = c_1 1_{K_A} \left( \tilde{T}^{\rho_{U_H}^\vee} a \right) 1_{K_{N_H}(\mathfrak{q})}(n) \psi_{\tilde{T}}^{-1}(n)
    \end{equation*}
    for some normalizing factor $c_1 > 0$.
  \end{enumerate}
\end{lemma}
\begin{proof}
  This follows by translation of Proposition \ref{proposition:formulas-for-fak-fan}.
\end{proof}

Note that $\chi_\theta^{-1}=\chi_{-\theta}$.  We take
\begin{equation*}
  J_H := J_{-\theta} = J_{\theta}, \ \ \chi_H=\tilde{\chi}_{-\theta}.
\end{equation*}
By Lemma \ref{lemma:let-j_th-leq-k-denote-inverse-image-centr-mathbfg}, we see that $f_0$, and hence also $f$, transforms on the right under $J_H$ via the character $\chi_H$.  Thus part \eqref{enumerate:we-have-fg-h-=-chi_hh-fg-all-g-h-in-h-times-j_h.-} of Theorem \ref{theorem:main-local-noncompact} holds.

We note also that
\begin{equation}\label{eqn:textf-right-invariant-under-k_q_hmathfrakq}
  \text{$f$ is right-invariant under $K_{Q_H}(\mathfrak{q})$}.
\end{equation}
This follows from the explicit formula \eqref{eqn:fu-n-=-c_1-1_k_a-left-tild-right-1_k_n_hm-psi} and the fact that $\tilde{\chi}_{-\theta}$ has trivial restriction to $K_{Q_H}(\mfq)$, as noted earlier in the proof of Lemma \ref{lemma:restr-w-q_h-posit-mult-char-funct-a_t-k_q_h}.

\subsection{Zeta integrals}
We turn to the evaluation of the zeta integral:
\begin{equation*}
  Z(W,f[s]) = \int_{N_H \backslash H} W \cdot W[f[s], \psi^{-1}].
\end{equation*}
We split this evaluation into a series of lemmas.
\begin{lemma}
  We have
  \begin{equation}\label{Eq:explicitZeta}
    Z(W,f[s]) = W[f[s], \psi^{-1}](a_T)\sqrt{\operatorname{vol}(N_H\backslash N_H a_T   K_H(\mfq))}.
  \end{equation}
\end{lemma}
\begin{proof}
  We first substitute the formula for $W |_H$ given in Lemma \ref{Lem:supportofW}, which says that the restriction of $W$ to $H$ is supported on $N_H a_T K_H(\mathfrak{q})$ and transforms on the right under $K_H(\mathfrak{q})$ by $\chi_\theta$.  We have seen (from the proof of assertion \eqref{enumerate:we-have-fg-h-=-chi_hh-fg-all-g-h-in-h-times-j_h.-} of Theorem \ref{theorem:main-local-noncompact}) that $f$ transforms on the right under $K_H(\mathfrak{q})$ by $\chi_\theta^{-1}$.  It follows that, writing $c := W(a_t) > 0$, we have
  \begin{equation*}
    Z(W, f[s]) = W[f[s], \psi^{-1}](a_T) c \vol(X), \quad
    X := N_H \backslash N_H a_T K_H(\mathfrak{q}).
  \end{equation*}
  On the other hand, by our normalization of $W$, we have $1 = \lVert W \rVert^2 = c^2 \vol(X)$.  The claimed identity follows.
\end{proof}
Recall that $(G,H)=(\GL_{n+1},\GL_n)$ and $T=|\mathfrak{o}/\mfq^2|$.  As elements of $G$,
\begin{equation*}
  \tilde{T}^{\rho_{U_H}^\vee}=\diag(\tilde{T}^{-(n-1)/2}, \tilde{T}^{-(n-3)/2},\dots, \tilde{T}^{(n-1)/2},1),
\end{equation*}
\begin{equation*}
  \tilde{T}^{\rho_{U_H}^\vee} a_T=\diag(\tilde{T}^{(n+1)/2},\tilde{T}^{(n+1)/2},\dots,\tilde{T}^{(n+1)/2},1).
\end{equation*}

Recall from Lemma \ref{Lem:fu-n-=-c_1-1_k_a-left-tild-right-1_k_n_hm-psi} that for some normalizing factor $c_1$, we have
\begin{equation}\label{eq:cqzbuaqaq7}
  f(u a n) = c_1 1_{K_A} \left( \tilde{T}^{\rho_{U_H}^\vee} a \right) 1_{K_{N_H}(\mathfrak{q})}(n) \psi_{\tilde{T}}^{-1}(n),
\end{equation}
\begin{lemma}
  We have
  \begin{equation}\label{Eq:Whittakervalue}
    W[f[s], \psi^{-1}](a_T) = c_1T^{(n+1)\trace(s)/2}\operatorname{vol}(a_TK_{N_H}(\mfq)a_T^{-1}, \, d n).
  \end{equation}
\end{lemma}
\begin{proof}
  We expand the definitions and compute: we have
  \begin{equation*}
    f[s](g) :=
    \int_{a \in A} (\delta_{U_H}^{1/2} |\cdot|^s )^{-1} (a)
    f (a g ) \, d a,
  \end{equation*}
  hence for $(a, n) \in A_H \times N_H$,
  \begin{align*}
    f[s](a n)&=c_1\int\limits_{b\in A} (\delta_{U_H}^{1/2} |\cdot|^s )^{-1} (b)
               1_{K_A} \left( \tilde{T}^{\rho_{U_H}^\vee} a b \right) 1_{K_{N_H}(\mathfrak{q})}(n) \psi_{\tilde{T}}^{-1}(n) \, d b\\
             &=c_1(\delta_{U_H}^{1/2} |\cdot|^s )(\tilde{T}^{\rho_{U_H}^\vee} a)1_{K_{N_H}(\mathfrak{q})}(n) \psi_{\tilde{T}}^{-1}(n),
  \end{align*}
  thus
  \begin{align}
    W[f[s], \psi^{-1}](a_T)&=\int\limits_{N_H}f[s](na_T)\psi(n)dn=\int\limits_{N_H}f[s](a_T \underbrace
                          {
                          a_T^{-1}na_T
                          }_{
                          \in N_H
                          })\psi(n)dn\\
                        &=c_1(\delta_{U_H}^{1/2} |\cdot|^s )(\tilde{T}^{\rho_{U_H}^\vee} a_T)\int\limits_{N_H}1_{K_{N_H}(\mathfrak{q})}(a_T^{-1}na_T) \psi_{\tilde{T}}^{-1}(a_T^{-1}na_T)\psi(n)dn\notag\\
                        &=c_1T^{(n+1)\trace(s)/2}\operatorname{vol}(a_TK_{N_H}(\mfq)a_T^{-1}, dn).\notag
  \end{align}
  Here in the last equality we used that $a_T$ conjugates $\psi$ and $\psi_{\tilde{T}}$.
\end{proof}

We collect the remaining ingredients in the following lemma:
\begin{lemma}
  \begin{enumerate}
  \item $ \operatorname{vol}(N_H\backslash N_H a_T K_H(\mfq))\asymp \delta_{N_H}^{-1}(a_T)T^{-n(n+1)/4}$
  \item $c_1\asymp \delta_{N_H}^{-1/2}(\tilde{T}^{-\rho_{U_H}^\vee})T^{n(n-1)/8}$
  \item $\operatorname{vol}(a_TK_{N_H}(\mfq)a_T^{-1}, dn)\asymp \delta_{N_H}(a_T) T^{-n(n-1)/4}$
  \end{enumerate}
\end{lemma}
\begin{proof}
  Note that in general for integration over the open cell $UA N$ in Bruhat decomposition, the corresponding Haar measure can be written as
  \begin{equation*}
    dg=c_0\delta_{N}(a)\, d u \, d a \, d n
  \end{equation*}
  for some absolutely bounded constant $c_0$.
  Similarly for $NAU$, the measure is $c_0\delta_{U}(a)\, d n \, d a \, d u$.

  (1) We have
  \begin{align*}
    &\operatorname{vol}(N_H\backslash N_H a_T   K_H(\mfq))=\operatorname{vol}(N_H\backslash N_H a_T K_{A_H}(\mfq)  K_{U_H}(\mfq))\\
    \asymp &\delta_{U_H}(a_T)T^{-n(n+1)/4}=\delta_{N_H}^{-1}(a_T)T^{-n(n+1)/4}.
  \end{align*}

  (2) By normalization and Lemma \ref{Lem:fu-n-=-c_1-1_k_a-left-tild-right-1_k_n_hm-psi}, we have
  \begin{equation*}
    1=\int\limits_{U_H\backslash H}|f|^2=|c_1|^2\operatorname{vol}(U_H\backslash U_H\tilde{T}^{-\rho_{U_H}^\vee}K_AK_{N_H}(\mfq))\asymp|c_1|^2\delta_{N_H}(\tilde{T}^{-\rho_{U_H}^\vee})T^{-n(n-1)/4}.
  \end{equation*}
  The expression for $c_1$ follows.

  (3) follows also from direct computations.
  \begin{equation*}
    \operatorname{vol}(a_TK_{N_H}(\mfq)a_T^{-1},dn)\asymp \delta_{N_H}(a_T) T^{-n(n-1)/4}
  \end{equation*}

\end{proof}
Plugging these computations and~\eqref{Eq:Whittakervalue} into~\eqref{Eq:explicitZeta}, we get
\begin{align*}
  Z(W,f[s]) &=c_1T^{(n+1)\trace(s)/2}\text{Vol}(a_TK_{N_H}(\mfq)a_T^{-1}, dn)\text{Vol}^{1/2}(N_H\backslash N_H a_T   K_H(\mfq)).\\
            &\asymp T^{(n+1)\trace(s)/2} \delta_{N_H}^{-1/2}(\tilde{T}^{-\rho_{U_H}^\vee})T^{n(n-1)/8}\delta_{N_H}(a_T) T^{-n(n-1)/4}\delta_{N_H}^{-1/2}(a_T)T^{-n(n+1)/8}\\
            &=T^{(n+1)\trace(s)/2}\delta_{N_H}^{1/2}( \tilde{T}^{\rho_{U_H}^\vee}a_T)T^{-n^2/4}\\
            &=T^{(n+1)\trace(s)/2}T^{-n^2/4}.
\end{align*}
This confirms part~\eqref{enumerate:each-s-in-mathfr-mathbbc-we-have-begin-zw-fs-=-c-tzeta-integrals} of Theorem~\ref{theorem:main-local-noncompact}.

\subsection{Bilinear form estimates}
Part~\eqref{enumerate:let-psi_1-psi_2-:-k_h-right-mathbbc-be-funct-satis} of Theorem~\ref{theorem:main-local-noncompact} can be proved exactly as in the compact case in~\cite[Theorem 9.1(iii)]{2023arXiv2309.06314}.

\subsection{Archimedean case and adaption} \label{Sec:ArchimedeanTest}

We also need the analogue of Theorem \ref{theorem:main-local-noncompact} at the Archimedean place.  This is contained in \cite[Theorem 3.1]{2021arXiv210915230N} when the archimedean component $T_\infty$ is large enough.  (Recall from Theorem \ref{theorem:cnpy9gjptx} that $T_\infty$ is a parameter that controls the size of our archimedean parameters.) However when $T_\infty$ is very small compared to $T$, the same choice of test functions would lead to the problem that the last term of \cite[(21.24)]{2021arXiv210915230N} would be non-negligible and mess up the rest arguments.

For this reason, we introduce some auxiliary parameters for the archimedean place.  We fix a small parameter $\delta_0>0$, to be chosen later, and set
\begin{equation}
  \label{Eq:auxiliaryR}
  R := \max\{T_\infty, T^{\delta_0}\}.
\end{equation}
If $T_\infty =R$, we pick test vectors as in \cite[Theorem 3.1]{2021arXiv210915230N} directly; if $T_\infty < R$ on the other hand, we use the following lemma.
\begin{lemma}\label{Lemma:archimedeananalogueforR}
  Let $\pi$ be a generic irreducible unitary representation of $G$ over $\mathbb{R}$, each of whose archimedean $L$-function parameters are at most $R$. There exists
  \begin{itemize}
  \item $f,f_*\in \mathcal{S}^e(U_H\backslash H)^{W_H}$, a fixed element $\phi_{Z_H}\in C_c^\infty(Z_H)$,
  \item a smooth vector $W\in \mathcal{W}(\pi,\psi)$ of norm $\leq 1$,
  \item $\omega_0\in C_c^\infty(G)$, supported in each fixed neighborhood of the identity element,
  \end{itemize}
  with the following further properties. Define
  \begin{equation*}
    \omega(g)=(\omega_0*\omega_0^*)(g),\ \omega^\sharp(g)=\int\limits_{z\in Z}\pi|_Z(z)\omega(zg)dz,
  \end{equation*}
  \begin{enumerate}
    [(i)]
  \item\label{Item1} $f_*(g)=\int\limits_{z\in Z_H}f(gz^{-1})\phi_{Z_H}(z)dz.$
  \item\label{enumerate:cnpv1gulzv} Each of the following assertions remains valid upon replacing $f$ with $f_*$.
  \item\label{enumerate:cnpv1gumz8} For each fixed compact subset $\mathcal{D}$ of $\mathfrak{a}^*$ and fixed $l\in \mathbb{Z}_{\geq 0}$, we have
    \begin{equation*}
      \nu_{\mathcal{D},l,T}(f)\ll R^{o(1)}.
    \end{equation*}
  \item\label{enumerate:cnpv1gunur} The local zeta integral $Z(W,f[s])$ is entire in $s$, with lower bound
    \begin{equation*}
      Z(W,f[0])\gg R^{-\dim(H)/4}.
    \end{equation*}
    For each $s\in \mathfrak{a}_{\mathbb{C}}^*$ with $\mathfrak{R}(s)\ll 1$, we have the upper bound
    \begin{equation*}
      Z(W,f[s])\ll R^{O(1)}(1+|s|)^{-\infty},
    \end{equation*}
    where the implied exponent is independent of $\delta_0$.
  \item\label{enumerate:cnpv1guowu} For each fixed $u\in \mathfrak{U}(G)$, we have
    \begin{equation*}
      ||\pi(u)(\pi(\omega_0)W-W)||\ll R^{-\infty}.
    \end{equation*}
  \item\label{enumerate:cnpv1guqe8} We have $\int\limits_{H}|\omega^\sharp|\ll R^{\dim(N)+\epsilon}$.
  \item\label{enumerate:cnpv1gus83} For each fixed compact subset $\Omega_H$ of $H$, there is a fixed compact subset $\Omega'_H$ of $H$ with the following property. Let $\Psi_1,\Psi_2:H\rightarrow \mathbb{C}$ be continuous functions. For $j=1,2$ define the convolutions
    \begin{equation*}
      (\Psi_j*\phi_{Z_H})(x):=\int\limits_{z\in Z_H}\Psi_j(xz^{-1})\phi_{Z_H}(z)dz.
    \end{equation*}
    Let $\gamma\in \overline{G}-H$. Then the integral
    \begin{equation*}
      I:=\int\limits_{x,y\in \Omega_H}|(\Psi_1*\phi_{Z_H})(x)(\Psi_2*\phi_{Z_H})(y)\omega^\sharp(x^{-1}\gamma y)dxdy
    \end{equation*}
    satisfies the estimate
    \begin{equation*}
      I\ll R^{O(1)}||\Psi_1||_{L^2(\Omega'_H)}||\Psi_2||_{L^2(\Omega'_H)}.
    \end{equation*}
    Here the exponent is some fixed quantity that does not depend upon $\delta_0$.
  \end{enumerate}
\end{lemma}
\begin{proof}
  We follow closely the proof in \cite[Section 11.4]{2021arXiv210915230N}, with the parameter $T$ there (corresponding to $T_\infty$ in the current paper) replaced by $R$.  We take for $f$ the element $c_2\mathcal{J}_R[\psi^{-1},\beta_0,\gamma]$ as in \cite[Definition 10.4]{2021arXiv210915230N}.  By \cite[Prop 9.4, Lemma 10.9]{2021arXiv210915230N}, it is a member of $\mathcal{S}^e(U_H\backslash H)^{W_H}$.  We construct the element $W$ as in \cite[Section 11.4]{2021arXiv210915230N} and the test function $\omega_0$ as in \cite[Theorem 8.7]{2021arXiv210915230N}, again with the parameter $R$ playing the role of $T$.

  The proofs of assertions \eqref{Item1}--\eqref{enumerate:cnpv1guowu} are then exactly as in \cite{2021arXiv210915230N} (sometines in sharper forms).  Indeed:
  \begin{itemize}
  \item (i) is automatic from definition.
  \item (ii) is true, as all arguments manifestly apply to both $f$ and $f_*$.
  \item (iii) can be proven similarly as in \cite[Section 11.4]{2021arXiv210915230N}.
  \item For (iv), one can obtain the lower bound like in \cite[Section 11.4]{2021arXiv210915230N}.  The upper bound follows again from the fact that $W|_{Q_H}$ is a smooth bump function near the identity, as discussed in \cite[Definition 8.4]{2021arXiv210915230N}.
  \item (v) follows from \cite[Theorem 8.7(iv)]{2021arXiv210915230N}
  \end{itemize}
  It remains to treat \eqref{enumerate:cnpv1guqe8} and \eqref{enumerate:cnpv1gus83}.  The first of these, \eqref{enumerate:cnpv1guqe8} follows from combining (i) and (ii) of \cite[Theorem 8.7]{2021arXiv210915230N}. The point here is that when $T_\infty\lll R$, we no longer have stability condition (see \cite[Section 7]{2021arXiv210915230N}, \cite[Definition 8.6] {2023arXiv2309.06314}) at the archimedean place.

  Turning to  \eqref{enumerate:cnpv1gus83}, we can argue as in \cite[Section 11.4]{2021arXiv210915230N}, using trivial bounds by $R^{O(1)}$ whenever $T$ was involved there, without using the additional power saving from volume bound (since again stability is not available).
\end{proof}

\section{Growth bounds for Eisenstein series}\label{sec:cnjen3n0cb}
In this section, we generalize~\cite[Theorem 4.1]{2021arXiv210915230N}.  This takes a bit of preparations.


\subsection{Some classes of Schwartz functions}\label{sec:cnm700p7hg}
Suppose for the moment that $F = \mathbb{R}$ (a similar discussion applies when $F = \mathbb{C}$, replacing $T$ by $T^{1/2}$ where appropriate).  Let $\mathcal{D}$ be a compact subset of $\mathfrak{a}^*$, let $\ell \in \mathbb{Z}_{\geq 0}$, and let $T \geq 1$.  In~\cite[\S2.9]{2021arXiv210915230N}, we defined a seminorm $\nu_{\mathcal{D},\ell,T}$ on $\mathcal{S}^e(U \backslash G)$ by the following formula:
\begin{equation*}
  \nu_{\mathcal{D},\ell,T}(f) := \sup_{
    \substack{
      s \in \mathfrak{a}_{\mathbb{C} }^\ast :  \\
      \Re(s) \in \mathcal{D}
    }
  }
  \sum _{
    \substack{
      m \leq \ell \\
      \dotsb
    }
  } {\left( 1 + \lvert s \rvert \right)}^{\ell }
  \frac{\lVert R (x_1 \dotsb x_m ) f [s] \rVert}{
    T^{\left\langle \rho_U^\vee , \Re(s) \right\rangle + m }
  }.
\end{equation*}
Here the sum is taken over integers elements $x_1,\dotsc,x_m$ of a fixed basis for the Lie algebra of $G$.


As in \cite[\S20.1.1]{2021arXiv210915230N}, we denote by $\mathfrak{E}(N \backslash G, T)$ the class of all $f \in \mathcal{S}^e(N \backslash G)$ such that for all fixed $\mathcal{D}$ and $\ell$ as above, we have $\nu_{\mathcal{D}, \ell, T}(f) \ll T^{o(1)}$.  The informal idea (made precise in \cite[Lemma 20.2]{2021arXiv210915230N}) is that $f$ belongs to this class precisely when $f$ concentrates near $T^{-\rho_U^\vee}$, has $L^2$-norm $O(1)$, is uniformly smooth under left translation by $A$ (in the sense that its norms barely increase after taking derivatives on the left under $A$), and has ``frequency'' $O(T)$ under right translation by $G$ (e.g., in the sense that its norms increase by a factor of at most roughly $T^m$ when we take $m$ derivatives on the right).

Suppose now that $F$ is non-archimedean, and let $(\mathfrak{o}, \mathfrak{p}, q, \varpi)$ denote the associated data.  We will define (simpler) non-archimedean analogues of the classes defined above in the archimedean case.
\begin{definition}\label{Defn:Esubspace}
  Let $\mathfrak{q} \subseteq \mathfrak{p} $ be an $\mathfrak{o}$-ideal.  We choose a generator $\tilde{T}^{1/2}$ of $\mathfrak{q}^{-1}$ and define $\tilde{T}$, $\tilde{T}^{\pm \rho_U^\vee}$ and $T$ as in \eqref{eqn:tild-:=-left-tild-right2-qquad-tild-rho_-:=-left-t}.  We denote then by
  \begin{equation*}
    \mathfrak{E}(U \backslash G, \mathfrak{q}^2)
  \end{equation*}
  the class of functions
  \begin{equation*}
    f : U \backslash G \rightarrow \mathbb{C}
  \end{equation*}
  with the following properties.
  \begin{enumerate}
    [(i)]
  \item\label{enumerate:cnh5ndgerr} $f$ is left-invariant under $K_A$ and right-invariant under $K(\mathfrak{q}^2)$.
  \item\label{enumerate:cnh5ndggry} $f$ is supported on
    \begin{equation*}
      U K_A \tilde{T}^{-\rho_U^\vee} K.
    \end{equation*}
  \item\label{enumerate:cnh5ndghxh} The $L^2$-norm of $f$, call it $\lVert f \rVert$, is $O(1)$.
  \end{enumerate}
  We denote by $\mathfrak{E}(U \backslash G, \mathfrak{q}^2)^W$ the subclass consisting of elements fixed by $\mathcal{F}_{w,\psi}$ for each $w \in W$ (\S\ref{sec:normalized-intertwining-operators}).
\end{definition}

\begin{lemma}\label{Lem:duality}
  The classes $\mathfrak{E}(U \backslash G, \mathfrak{q}^2)$ and $\mathfrak{E}(U \backslash G, \mathfrak{q}^2)^W$ are both preserved under the duality map
  \begin{equation*}
    f \mapsto [\tilde{f}:g \mapsto f (w_G g^{- \transpose})].
  \end{equation*}
  Here $w_G$ is the standard representative for the longest Weyl element.
\end{lemma}
\begin{proof}
  This is an analogue of \cite[Lemma 20.5]{2021arXiv210915230N}.  It may be proved in the same way, but the proof simplifies a bit in the present non-archimedean case.  The main point is that $w_G (T^{- \rho_U^\vee })^{- \transpose} = T^{- \rho_U^\vee } w_G$, which shows that the duality map preserves the concentration property.

  As in the proof of \cite[Lemma 20.5]{2021arXiv210915230N}, the $W-$invariance follows from that
  \begin{equation*}
    \tilde{f}[s](g)=f[-w_G s](w_Gg^{-T}),
  \end{equation*}
  \begin{equation*}
    W[\tilde{f}[s],\psi](g)=W[f[-w_Gs],\psi^{-1}](w_Gg^{-T})
  \end{equation*}
  by direct computations, \cite[(2.27)]{2021arXiv210915230N} and equivalent description of $W-$invariance in \cite[Lemma 2.6]{2021arXiv210915230N}.

\end{proof}

\subsection{Global preliminaries}\label{sec:cnpv1czec0}
\subsubsection{Basic notation}
Let $\mathbb{A}$ denote the adele ring of $\mathbb{Q}$ and $\psi : \mathbb{A} / \mathbb{Q} \rightarrow \mathbb{C}$ the standard nontrivial unitary character, whose restriction to $\mathbb{R}$ is given by $x \mapsto e^{2 \pi i x}$. For a finite set $S$ of places, let $\mathbb{A}_S$ denote the product of local fields inside $S$, and $\mathbb{A}^S$ the restricted product of local fields outside $S$.

For a general minimal parabolic subgroup $Q$ with its unipotent subgroup $U$, we set $[\mathbf{G}]_Q := \mathbf{U}(\mathbb{A}) \mathbf{A}(\mathbb{Q}) \backslash \mathbf{G}(\mathbb{A})$ and $[\mathbf{G}] := \mathbf{G}(\mathbb{Q}) \backslash \mathbf{G}(\mathbb{A})$.

\subsubsection{Schwartz spaces}
We define the Schwartz space $\mathcal{S}([\mathbf{G}]_Q)$ and its subspace $\mathcal{S}^e([\mathbf{G}]_Q)$ as in \cite[\S2.6]{2021arXiv210915230N}; the latter consist of Schwartz functions on $[\mathbf{G}]_Q$ that are left-invariant under the maximal compact subgroup of $\mathbf{A}(\mathbb{A})$ (or equivalently, that of $\mathbf{A}(\mathbb{A}) / \mathbf{A}(\mathbb{Q})$).  Both spaces admit a Mellin theory as in the local case in \S\ref{Sec:localMellin}, which for $\mathcal{S}^e([\mathbf{G}]_Q)$ is indexed by $s \in \mathfrak{a}_{\mathbb{C}}^*$. See \cite[\S2.8]{2021arXiv210915230N}.


\subsubsection{Completed pseudo Eisenstein series}\label{Sec:completePseudoEisenstein}
Given $\Phi \in \mathcal{S}([\mathbf{G}]_Q)$, we denote by $\Eis[\Phi] : [\mathbf{G}] \rightarrow \mathbb{C}$ the associated Eisenstein series
\begin{equation}\label{eqn:eisph-:=-sum_g-in-mathbfqm-backsl-mathbfgm-phiga-eisenstein-series-incomplete}
  \Eis[\Phi](g) := \sum_{\gamma \in \mathbf{Q}(\mathbb{Q}) \backslash \mathbf{G}(\mathbb{Q})} \Phi(\gamma g).
\end{equation}
This sum converges absolutely, and defines a smooth function of uniform moderate growth. See \cite[\S2.15]{2021arXiv210915230N} for further references.

In \cite[\S2.16]{2021arXiv210915230N}, we defined certain ``completed pseudo Eisenstein series'' on $[\mathbf{G}]$ attached to elements $f_\infty \in \mathcal{S}^e(\mathbf{U}(\mathbb{R}) \backslash \mathbf{G}(\mathbb{R}))$.  As noted in the cited reference, this definition is the specialization to $F = \mathbb{Q}$ and $S = \{\infty\}$ of a more general recipe valid over any number field and for any finite set of places $S$ containing the archimdean places.  In particular, we can apply it to $F = \mathbb{Q}$ and a general finite set of places $S$.  In that case, it reads as follows.  For each $s \in \mathfrak{a}_{\mathbb{C}}^*$, denote by $\mathcal{I}(s)$ the corresponding principal series representation, obtained by normalized induction from $\mathbf{Q}(\mathbb{A})$ of the character $|.|^s$ of $\mathbf{A}(\mathbb{A})$.  It factors as the restricted tensor product of local principal series representations $\mathcal{I}_{\mathfrak{p}}(s)$.  Writing $R_U^+$ for the set of $U$-positive roots of $\mathbf{A}$, we set
\begin{equation*}
  \zeta^{(S)}(U,s) := \prod_{\alpha \in R_U^+} \zeta^{(S)}(1 + \alpha^\vee(s)),
\end{equation*}
where $\zeta^{(S)}(z) = \prod_{p \notin S} (1 - p^{-z})^{-1}$ denotes the partial Riemann zeta function.
We define the Weyl-invariant polynomial
\begin{equation*}
  \mathcal{P} _G (s) := \prod _{\alpha \in R_U^+} \alpha ^\vee (s)^2 (\alpha^\vee(s)^2 - 1).
\end{equation*}
The product $\mathcal{P}_G(s) \zeta^{(S)}(U,s)$ is then entire and of polynomial growth in vertical strips, hence acts as a multiplier on, e.g., the Schwartz space $\mathcal{S}^e(\mathbf{U}(\mathbb{R}) \backslash \mathbf{G}(\mathbb{R}))$.

Fix now, for each place $\mathfrak{p} \in S$, an element $f_\mathfrak{p} \in \mathcal{S}^e(\mathbf{U}(\mathbb{Q}_\mathfrak{p}) \backslash \mathbf{G}(\mathbb{Q}_\mathfrak{p}))$, with associated Mellin components $f_\mathfrak{p}[s]$.  Write $f_S$ for the tensor product of these elements.  We associate to $f_S$ an element
\begin{equation*}
  \Phi[f_S] \in \mathcal{S} ^{e} ({[G] }_{Q})
\end{equation*}
characterized by requiring that its Mellin components $\Phi[f_S][s]$ are given by
\begin{equation*}
  \Phi[f_S][s] = \mathcal{P}_G(s) \zeta^{(S)}(U,s) \otimes _{\mathfrak{p} } f_\mathfrak{p}[s],
\end{equation*}
where for $\mathfrak{p} \in S$, the element $f_\mathfrak{p}[s]$ is a Mellin component of one of our chosen $f_\mathfrak{p}$, while for $p \notin S$, we take $f_p[s] \in \mathcal{I}_p(s)$ to be the spherical vector, normalized to take the value $1$ at the identity element.  The ``completed'' pseudo Eisenstein series $\Eis[\Phi[f_S]]$ may then be defined as in \eqref{eqn:eisph-:=-sum_g-in-mathbfqm-backsl-mathbfgm-phiga-eisenstein-series-incomplete}.  We refer to \cite[\S2.16]{2021arXiv210915230N} for a detailed discussion of these series.  We note that, due to our introduction of the factor $\mathcal{P}_G(s)$, the series $\Eis[\Phi[f_S]]$ actually defines a Schwartz function on $[\mathbf{G}]$, and in particular, has rapid decay.

\subsubsection{Miscellaneous notation}
For $Y \in \mathbf{A}(\mathbb{A})$ and $\Phi \in \mathcal{S}([\mathbf{G}]_Q)$, we define
\begin{equation*}
  L(Y) \Phi(g) := \delta_U^{-1/2}(Y) \Phi(Y g).
\end{equation*}
This definition applies in particular to $Y \in \mathbf{A}(\mathbb{R}) \hookrightarrow \mathbf{A}(\mathbb{A})$.

Let $\Omega \subseteq \mathbf{G}(\mathbb{R})$ be a fixed compact subset.  Set $K_p := \mathbf{G}(\mathbb{Z}_p)$.  Let $\mathbf{A}(\mathbb{R})^0$ denote the connected component of the identity element of $\mathbf{A}(\mathbb{R})$, and $\mathbf{A}(\mathbb{R})^0_{\geq 1}$ the subset of dominant elements
of $\mathbf{A}(\mathbb{R})^0$.

For each $t \in \mathbf{A}(\mathbb{R})^0_{\geq 1}$, we define a seminorm $\lVert . \rVert_{t,\Omega}$ on $L^2([\mathbf{G}])$ by
\begin{equation*}
  \lVert \varphi  \rVert_{t, \Omega }^2 := \int_{g \in \Omega \times \prod_p K_p} \lvert \varphi (t g)  \rvert^2 \, d g.
\end{equation*}

\subsection{Main statement}\label{sec:cnjobat5rx}
For the rest of this section, we shall specify $\mathbf{Q}$ to be $\mathbf{B}$.  We are now prepared to state our generalization of \cite[Theorem 4.1]{2021arXiv210915230N}, of which the cited result is simply the case $S = \{\infty \}$.

\begin{theorem}\label{theorem:eisenstein-growth-bound-general-level}
  Assume that $n = \rank(\mathbf{V})$ is fixed.  Let $\Omega$ be a fixed compact subset of $\mathbf{G}(\mathbb{R})$.  Let $S$ be a finite set of places of $\mathbb{Q}$, containing the infinite place.  Let $T_\infty \geq 1$.  For each finite prime $p \in S$, let $\mathfrak{q}_p \subseteq p \mathbb{Z}_p$ be a nonzero $\mathbb{Z}_p$-ideal, and write $T_p := [\mathbb{Z}_p:\mathfrak{q}_p]^2$.  Set $T := \prod_\mathfrak{p} T_\mathfrak{p}$.
  Let $Y \in A(\mathbb{R})^0$ satisfy that $T^{-\kappa} \leq |Y_j| \leq T^{\kappa}$ for  $\kappa =1/2$, all $v\in S$ and $j \in \{1, \dotsc, n\}$, with trivial components outside $S$.  Let $t \in {\mathbf{A}(\mathbb{R})}_{\geq 1}^0$.  For each $\mathfrak{p} \in S$, let $f_\mathfrak{p} \in {\mathfrak{E}(N(\mathbb{Q}_\mathfrak{p}) \backslash G(\mathbb{Q}_\mathfrak{p}), T_\mathfrak{p})}^W$
  be given, and write $f_S$ for their tensor product.  Then the pseudo Eisenstein series
  \begin{equation*}
    \Psi := \Eis[\Phi[L(Y) f_S]] \in \mathcal{S} ([G])
  \end{equation*}
  as defined in \cite[\S2.16.1]{2021arXiv210915230N}  satisfies the following local $L^2$-norm bound:
  \begin{equation*}
    \frac{\lVert \Psi  \rVert_{t,\Omega}^2}{\delta_N(t)} \ll T ^{o(1)}
    \min \left( |{\det(Y) }|^{-1} t _1 ^{- n }, |\det (Y)| t _n ^n  \right).
  \end{equation*}
\end{theorem}
The proof is a generalization of that of the cited result, with a few nuances.  A couple arguments need to be translated from the archimedean to the non-archimedean setting, but the translation always introduces simplification rather than complication.


%
%

The general strategy, explained in~\cite[\S1.3.6--1.3.8 and \S18]{2021arXiv210915230N}, remains the same: we define certain mirabolic Eisenstein series $\theta$
that majorize the region $t \Omega$, and estimate the integral $\int \lvert \Psi \rvert ^2 \theta$ via unfolding.



Together with duality, we shall reduce the proof Theorem \ref{theorem:eisenstein-growth-bound-general-level} to that of Proposition \ref{proposition:proposition-20-14} below, which requires a little more preparations.

\subsection{Reduction to bounds for Rankin---Selberg integrals}\label{sec:cnh5ncyv28}


Recall that by choosing a basis for $\mathbf{V}$, we can identify $\mathbf{G} = \mathbf{G} \mathbf{L} _n$.  Let $\mathbf{V}^*$ be the scheme of row vectors for $\mathbf{G}$.  We denote by $\mathbf{V}^*_{\prim}$ the subfunctor given by taking $\mathbf{V}^*_{\prim}(R)$ to consist of all elements of $\mathbf{V}^*(R)$ whose entries generate the unit ideal of $R$.  For example, for a field $F$, we have $\mathbf{V}^*_{\prim}(F) = \mathbf{V}^*(F) - \{0\}$, while for a local ring $\mathfrak{o}$ with maximal ideal $\mathfrak{p}$, we have $\mathbf{V}^*_{\prim}(\mathfrak{o}) = \mathbf{V}^*(\mathfrak{o}) - \mathfrak{p} \mathbf{V}^*(\mathfrak{o})$.  The space $C_c^\infty(\mathbf{V}_{\prim}^*(\mathbb{A}))$ is spanned by pure tensors $\phi = \otimes_\mathfrak{p} \phi_\mathfrak{p}$, where $\phi_\mathfrak{p} \in C_c^\infty(\mathbf{V}_{\prim}^*(\mathbb{Q}_\mathfrak{p}))$ such that for almost all primes $p$, the function $\phi_p$ is the characteristic function of $\mathbf{V}^*_{\prim}(\mathbb{Z}_p)$.

Let $\mathbf{P}$ be either a standard (upper-triangular) maximal proper parabolic subgroup of $\mathbf{G}$, or $\mathbf{G}$ itself.  Having chosen $\mathbf{P}$, we introduce the following auxiliary notation, as in \cite[\S19.8 and \S20.5]{2021arXiv210915230N}.
\begin{itemize}
\item If $\mathbf{P}$ is a proper parabolic subgroup, then we write its standard (block-diagonal) Levi subgroup as $\mathbf{M} ' \times \mathbf{M} ''$, where $\mathbf{M} '$ (resp. $\mathbf{M} ''$) is the upper-left (resp. lower-right) block.
\item If $\mathbf{P} = \mathbf{G}$, then we set $\mathbf{M} ' := \{1\}$ and $\mathbf{M} '' := \mathbf{G}$.
\item We denote by $n'$ (resp. $n''$) the rank of $\mathbf{M} '$ (resp. $\mathbf{M} ''$), so that $n = n' + n''$.
\item We denote by $\mathbf{U} _P$ the unipotent radical of $\mathbf{P}$.
\item We denote by $\mathbf{B}_{M'}$ and $\mathbf{B}_{M''}$ the standard (upper-triangular) Borel subgroups of $\mathbf{M} '$ and $\mathbf{M} ''$.
\item We denote by $\mathbf{A} '$ and $\mathbf{A} ''$ the respective diagonal subgroups of $\mathbf{M} '$ and $\mathbf{M} ''$.
\item For $\Phi \in \mathcal{S}([\mathbf{G}]_B)$, we define $W^P [\Phi ] \in C^\infty(\mathbf{G}(\mathbb{A}))$ by the absolutely convergent integral formula
  \begin{equation*}
    W^P [\Phi ] (g) := \int_{N'' (\mathbb{A} )}
    \Phi (w_{M''}^{-1} u g ) \psi^{-1} (u) \, d u,
  \end{equation*}
  where $w_{M''} = w_{M''}^{-1} \in \mathbf{M}''$ denotes the standard long Weyl element.
\item Let $S$ be a finite set of places of $\mathbb{Q}$, containing the infinite place.  Let $f_S = \otimes_{\mathfrak{p} \in S} f_\mathfrak{p}$, where $f_\mathfrak{p} \in \mathfrak{E}(\mathbf{N}(\mathbb{Q}_\mathfrak{p}) \backslash \mathbf{G}(\mathbb{Q}_\mathfrak{p}), T_\mathfrak{p})$.
  We may then define $\Phi[f_S] \in \mathcal{S}([\mathbf{G}]_B)$ as in \S\ref{Sec:completePseudoEisenstein}.  Let $\phi \in C_c^\infty(\mathbf{V}^*_{\prim}(\mathbb{A}))$.  We then define the integral
  \begin{equation}
    \mathcal{Q}_P(\phi, f_S)
    :=
    \int _{g \in (M' N'' {U_P} \backslash G)(\mathbb{A}) } \phi(e_n g)
    \int _{h \in [M']_{B_{M'}}} \\
    \left\lvert
      W^P[\Phi[f_S]](h g)
    \right\rvert^2
    \, \frac{d h}{\delta_{U_P}(h)}
    \, d g.
  \end{equation}
\end{itemize}

We have the following analogue of \cite[Proposition 20.14]{2021arXiv210915230N}.
\begin{proposition}\label{proposition:proposition-20-14}
  Suppose given
  \begin{itemize}
  \item a finite set $S$ of places of $\mathbb{Q}$,
  \item for each $\mathfrak{p} \in S$, a quantity $T_\mathfrak{p} \geq 1$ such that for finite primes $p$, the quantity $T_p$ is an even power of $p$,
  \item elements $f_\mathfrak{p} \in \mathfrak{E}(N(\mathbb{Q}_\mathfrak{p}) \backslash G(\mathbb{Q}_\mathfrak{p}), T_\mathfrak{p})^W$,
  \item an element $r > 0$ with $r =T^{O(1)}$ (playing the role of $t_n$).


  \item { an element $Y \in A(\mathbb{R})^0$ with $T^{-\kappa} \leq |Y_j| \leq T^{\kappa}$ for $\kappa = 1/2$ and all $j \in \{1, \dotsc, n\}$.}
  \end{itemize}
  Let $\phi_\infty \in C_c(\mathbf{V}_{\prim}^*(\mathbb{R}))$ be a nonnegative function that is $K_\infty$-invariant and supported on elements $v$ with $\lVert v \rVert \asymp r$.  Let $\phi_p \in C_c(\mathbf{V}^*_{\prim}(\mathbb{Q}_p))$ denote the characteristic function of $\mathbf{V}^*_{\prim}(\mathbb{Z}_p)$.  Set $\phi := \otimes_{\mathfrak{p}} \phi_\mathfrak{p}$.  Then
  \begin{equation*}
    \mathcal{Q}_P (\phi, L(Y) f_S ) \ll r^n\det (Y)  T^{o(1)}.
  \end{equation*}
\end{proposition}

The proof of Proposition \ref{proposition:proposition-20-14} will be given later.  For now, let us explain why Proposition \ref{proposition:proposition-20-14} implies Theorem \ref{theorem:eisenstein-growth-bound-general-level}.  This implification may be verified exactly as in \cite[\S20]{2021arXiv210915230N}.  The basic idea is as follows.  For $\phi$ as in the statement of Proposition \ref{proposition:proposition-20-14}, $g \in \Omega \times \prod_p K_p$,
define the mirabolic Eisenstein series $\theta : [\mathbf{G}] \rightarrow \mathbb{C}$ by
\begin{equation*}
  \theta(g) := \sum_{v \in \mathbf{V}^*_{\prim}(F)} \phi(v g)
  =
  \sum_{
    \substack{
      v = (v_1,\dotsc,v_n) \in \mathbb{Z}^n  \\
      \gcd (v_1,\dotsc,v_n) = 1
    }
  } \phi_\infty(v g).
\end{equation*}
Let $t$ and $\Omega$ be as in the statement of Theorem \ref{theorem:eisenstein-growth-bound-general-level}.  Write $e_1^*, \dotsc, e_n^*$ for the standard basis of $\mathbf{V}^*$.  Then
\begin{equation*}
  \theta(t g) \geq \phi_\infty(e_n^* t g)
  = \phi_\infty(t_n e_n^* g).
\end{equation*}
By choosing $r:= t_n$, and choosing $\phi_\infty(r\cdot ) $ 
in a suitable manner depending only upon $\Omega$, we may arrange that the right hand side of the above is $\geq 1$ for any $g\in \Omega \times \prod_p K_p$ (see \cite[Lemma 20.7]{2021arXiv210915230N} for details).  It follows readily that for each $\Psi \in \mathcal{S}([G])$,
\begin{equation*}
  \lVert \Psi  \rVert_{t,\Omega}^2 \ll \delta_N(t) \int_{[\mathbf{G}]} \theta \lvert \Psi  \rvert^2.
\end{equation*}
(This is explained in \cite[Corollary 20.8]{2021arXiv210915230N} under the additional hypothesis that $\Psi$ be right $\mathbf{G}(\mathbb{Z}_p)$-invariant, but that hypothesis is used only to simplify the subsequent presentation, and is not used in the proof.)  By specializing to $\Psi = \Eis [\Phi[L(Y) f_S]]$, we see that the proof of Theorem \ref{theorem:eisenstein-growth-bound-general-level} reduces to obtaining suitable estimates for the integrals
\begin{equation*}
  \int_{[\mathbf{G} ] } \theta
  \left\lvert \Eis [\Phi[L(Y) f_S]]  \right\rvert^2.
\end{equation*}
The latter integrals may be unfolded.  The unfolding is detailed in \cite[\S19]{2021arXiv210915230N}.  It introduces a bounded sum (over parabolic subgroups $\mathbf{P}$ and Weyl group elements) of integrals as in Proposition \ref{proposition:proposition-20-14}.  The discussion of \cite[\S20.5]{2021arXiv210915230N} generalizes in a straightforward way from $\{\infty \}$ to general $S$, giving the required reduction.

We may and shall assume that $t_n= T^{\O(1)}$, as otherwise the required estimates follow (in stronger form) from standard Sobolev-type bounds as in \cite[\S2]{michel-2009}.  Proposition \ref{proposition:proposition-20-14} then implies the bound
\begin{equation*}
  \frac{\lVert \Psi  \rVert_{t,\Omega}^2}{\delta_N(t)} \ll T ^{o(1)}
  \det (Y)t _n ^n.
\end{equation*}
To obtain the other bound in Theorem \ref{theorem:eisenstein-growth-bound-general-level}, we apply the duality as in \cite[\S20.4]{2021arXiv210915230N}. The observations of \cite[\S20.1.2]{2021arXiv210915230N} concerning duality remain valid by Lemma \ref{Lem:duality}.


\subsection{Strong approximation and further unfolding}\label{sec:cnh5ncxyno}
Here we record an analogue of \cite[Lemma 20.19]{2021arXiv210915230N}.  To formulate that analogue, we define a linear map
\begin{equation*}
  \Theta^P : \left( \otimes_{\mathfrak{p}  \in S} \mathcal{S}^e\left(N(\mathbb{Q}_\mathfrak{p}) \backslash G(\mathbb{Q}_\mathfrak{p})\right) \right)
  \rightarrow C^\infty(N(\mathbb{A}) \backslash G(\mathbb{A}))
\end{equation*}
\begin{equation*}
  f_S \mapsto \Theta^P[f_S]
\end{equation*}
which upon taking $S = \{\infty \}$ specializes to the definition in \cite[\S20.8]{2021arXiv210915230N}.  We assume that $f_S = \otimes_{\mathfrak{p} \in S} f_{\mathfrak{p} }$.  Then $\Theta^P[f_S] = \prod_{\mathfrak{p}} \Theta^P[f_S]_\mathfrak{p}$, with local factors defined as follows.
\begin{itemize}
\item As in \cite[\S20.8]{2021arXiv210915230N}, we define $\Theta^P [f_S]_\infty$ by requiring that its Mellin components $\Theta^P[f]_\infty [s] \in \mathcal{I}_\infty(s)$ are given by
  \begin{equation*}
    \Theta^P [f ]_\infty[s] = \mathcal{P}_G(s) \frac{\zeta^{(S)} (N, s) }{ \zeta^{(S)} (N'', s)} f _\infty [s].
  \end{equation*}
\item For $p \in S - \{\infty\}$, we take $\Theta^P [f]_p := f_p$.
\item For $p \notin S$, we take $\Theta^P [f]_p := \Theta^P_p$ , where $\Theta^P_p$ is the ``$\mathbf{M}''$-basic vector", defined in \cite[\S20.8]{2021arXiv210915230N}, whose Mellin components $\Theta_p^P[s]$ are spherical and take the value $\zeta_{\mathfrak{p}}(N'', s)$ at the identity.
\end{itemize}
$\Theta^P[f_S]$ is related to $\Phi[f_S]$ above similarly as in \cite[Lemma 20.18]{2021arXiv210915230N}.

\begin{lemma}\label{lemma:let-f_s-=-otim-in-s-f_mathfr-be-as-abov-assume-tha}
  Let $f_S = \otimes_{\mathfrak{p} \in S} f_\mathfrak{p}$ be as above.  Assume that for each finite prime $p \in S$ and $g \in \mathbf{G}(\mathbb{Q}_p)$, the set
  \begin{equation*}
    \left\{ a \in \mathbf{A}(\mathbb{Q}_p) : f_\mathfrak{p}(a g) \neq 0 \right\}
  \end{equation*}
  is contained in a single $\mathbf{A}(\mathbb{Z}_p)$-coset.  Let $\phi \in C_c^\infty(\mathbf{V}^*_{\prim}(\mathbb{A}))$ be nonnegative.  For $(c,g) \in \mathbf{A}(\mathbb{Q}) \times \mathbf{G}(\mathbb{A})$, introduce the notation
  \begin{equation}\label{eq:cnjgi2biuc}
    W(c,g) := \int_{u \in N''(\mathbb{A})} \Theta^P [f] (c^{-1} w_{M''}^{-1} u g ) \psi^{-1} (u) \, d u.
  \end{equation}
  Then
  \begin{equation*}
    \mathcal{Q}_P (\phi, f) = 2^{n '}
    \int_{
      \substack{
        g \in \mathbf{N} \backslash \mathbf{G}(\mathbb{A}_S)  \\
        a'' \in \mathbf{A} ''(\mathbb{A}^S)
      }
    }
    \phi (e_n  ga'')
    \left\lvert
      \sum_{c \in \mathbf{A} ''(\mathbb{Q})}
      W(c, ga'' )
    \right\rvert^2
    \, d g
    \, \frac{d a ''}{\delta_N(a'')}.
  \end{equation*}
\end{lemma}
\begin{proof}
  This is an analogue of \cite[Lemma 20.19]{2021arXiv210915230N}.\footnote{The referenced lemma was formulated as an inequality, but the same proof gives equality.}  The new feature here is that the set of $S$-units in $\mathbb{Z}$ is infinite when $S$ is strictly larger than $\{\infty\}$, which we compensate for by the support condition that we have assumed concerning $f_{\mathfrak{p}}$.

  Indeed, following the first step of the proof in cited lemma, we get
  \begin{align}\label{Eq:QpIwasawa}
    & \mathcal{Q}_P (\phi, f)=\int\limits_{a''\in \mathbf{A}''(\mathbb{A})}\phi(e_na'')\int\limits_{
      \substack{
      k\in K\\
    k'\in K'\\
    a'\in [\mathbf{A}']
    }}\left|\sum\limits_{c\in \mathbf{A}(\mathbb{Q})}W(c,a'k'a''k)\right|^2 \, \frac{da'}{\delta_N(a')}\, d k' \, d k \frac{da''}{\delta_N(a'')}
  \end{align}
  Note that the integral over $a'$ is well-defined, i.e., independent of the choice of coset representative.  We may factor $c=c'c''$, with $(c', c'')\in \mathbf{A}'(\mathbb{Q}) \times \mathbf{A}''(\mathbb{Q})$.  Using that $\mathbf{M} '$ and $\mathbf{M} ''$ commute, we see that $W(c' c'', a' k' a'' k) = W(c'', (c')^{-1}a' k' a'' k)$.  Inverting $c'$, we see that the sum over $c$ and integral over $a'$ in the above may be written
  \begin{equation}\label{eq:cnjgi2hw8g}
    \int_{
      a'\in [\mathbf{A}']
    }
    \left|
      \sum_{c'\in \mathbf{A}'(\mathbb{Q})}
      \sum_{c''\in \mathbf{A}''(\mathbb{Q})}W(c'',c' a' k'a''k)\right|^2 \, \frac{da'}{\delta_N(a')}.
  \end{equation}

  We temporarily set $g := a' k' a'' k$.  We claim that the function $\mathbf{A}'(\mathbb{Q}) \ni c' \mapsto W(c'', c' g)$ is supported on a single $\mathbf{A}'(\mathbb{Z})$-coset that is independent of $c''$, and takes a constant value there.  To see this, we factor $W = \otimes W_{\mathfrak{p}}$ as a product over the places of integrals like \eqref{eq:cnjgi2biuc}, and reduce to showing that for each finite prime $p$, the function $\mathbf{A} '(\mathbb{Q}_p) \ni c' \mapsto \Theta^P [f]_p ((c'')^{-1} w_{M''}^{-1} u c' g)$ is supported on a single $\mathbf{A}'(\mathbb{Z}_p)$-coset independent of $c''$ and take a constant value there.  This feature holds by \cite[Lemma 20.17]{2021arXiv210915230N} when $p\notin S$ and by Lemma \ref{Lem:fu-n-=-c_1-1_k_a-left-tild-right-1_k_n_hm-psi} when $p \in S$, hence the claim.

  By the claim just established and the identity $\lvert \mathbf{A} '(\mathbb{Z}) \rvert = 2^{n'}$, we may simplify \eqref{eq:cnjgi2hw8g} to
  \begin{equation}\label{eq:cnjgi2hw8g2}
    2^{n '}    \int_{
      a'\in [\mathbf{A}']
    }
    \sum_{c'\in \mathbf{A}'(\mathbb{Q})}
    \left|
      \sum_{c''\in \mathbf{A}''(\mathbb{Q})}W(c'',c' a' k'a''k)\right|^2 \, \frac{da'}{\delta_N(a')}.
  \end{equation}
  The $a'$-integral and $c'$-sum now unfold to an integral over $\mathbf{A} '(\mathbb{A})$.  We insert this unfolded integral into \eqref{Eq:QpIwasawa}, giving
  \begin{align*}
    &\mathcal{Q}_P (\phi, f)\\
    = &2^{n '} \int\limits_{a''\in \mathbf{A}''(\mathbb{A})}\phi(e_na'')\int\limits_{
      \substack{
        k\in K\\
        k'\in K'\\
        a'\in \mathbf{A}'(\mathbb{A})
      }}\left|\sum\limits_{c\in \mathbf{A}''(\mathbb{Q})}W(c,a'k'a''k)\right|^2 \, \frac{da'}{\delta_N(a')}\, d k' \, d k \frac{da''}{\delta_N(a'')}
  \end{align*}
  We now write each adelic integral as a pair of integrals, one taken inside $S$, the other outside, and conclude using the Iwasawa decomposition, \cite[Lemma 20.17]{2021arXiv210915230N} and the $K_p$-invariance of $f_p$ for $p \notin S$.
\end{proof}

\subsection{Local estimates relevant for Eisenstein growth bound}\label{sec:cnh5ncyz7d}
Here we derive the local estimates at a non-archimedean place relevant for the proof of Theorem \ref{theorem:eisenstein-growth-bound-general-level}.  These are all straightforward adaptations (with significant simplifications) of the local archimedean estimates recorded in parts of \cite[\S21]{2021arXiv210915230N}.


We start with the case $P=G$. Let $f \in \mathcal{S}(N \backslash G)$, $c \in A$ and $g \in G$.  As in \cite[\S21.1]{2021arXiv210915230N}, we define
\begin{equation}\label{Eq:Wfcg}
  W(f, c, g) = \delta _N ^{1/2} (c) \int _{u \in N} f (c ^{-1} w _G u g) \psi ^{-1} (u) \, d u.
\end{equation}
Note that this integral is not exactly matching the local integral coming from Lemma \ref{lemma:let-f_s-=-otim-in-s-f_mathfr-be-as-abov-assume-tha}. But as $c$ is rational, we are free to multiply with $\delta_N^{1/2}(c)$ for every local place using the product law.
The integral converges absolutely under a mild growth assumption on $f$ (see \cite[Lemma 21.1]{2021arXiv210915230N}); in particular, it converges when $f \in \mathcal{S}(N \backslash G)$.

\begin{lemma}
  Suppose that $f \in \mathcal{S}(N \backslash G)$ is left $K_A$-invariant.  Then
  \begin{equation*}
    \int_{h \in N_H \backslash H}
    \int_{z \in Z}
    \lvert W (f, c, h z g) \rvert^2 \, d h \, d z
    \ll \lVert f \rVert^2.
  \end{equation*}
\end{lemma}
\begin{proof}
  This is an analogue of \cite[Lemma 21.4]{2021arXiv210915230N}, with the left $K_A$-invariance substituting for the Sobolev norm denoted $\mathcal{S}_{d_1,0}(f)$ in that reference, and can be verified in the same way.  The proof consists of expanding $f$ via Mellin inversion under left translation by $A$, using here the left $K_A$-invariance assumption to truncate the Mellin expansion to the group of unramified characters of $A$, and then appealing to the relationship established in \cite[Appendix A]{MR2930996} between $L^2$-norms in the Kirillov and induced models.  The main difference in the $p$-adic case is that the unramified characters form a compact group, so there are no convergence issues (addressed using differential operators in the archimedean case).  The remaining arguments are the same.
\end{proof}

\begin{lemma}
  \label{lemma:let-d--0.-let-f-in-mathc-backsl-g-be-such-that-fg-}
  Let $D > 0$.  Let $f \in \mathcal{S}(N \backslash G)$ be left $K_A$-invariant and $f(g) \neq 0$ only if $\lvert \det g \rvert = D$.  Let $c \in A$.  Then the integral
  \begin{equation*}
    I(c) := \int_{g \in N \backslash G}
    \lVert e _n g  \rVert^{- n } \lvert W (f, c, g )  \rvert^2 \, d g
  \end{equation*}
  satisfies
  \begin{equation*}
    I (c ) \ll \frac{\lVert f \rVert^2}{D \lvert \det c \rvert}.
  \end{equation*}
\end{lemma}
\begin{proof}
  This is an analogue of \cite[Lemma 21.5]{2021arXiv210915230N}, and can be deduced from the previous lemma in the same way.  The proof consists of a comparison of measures.
\end{proof}

Denote by $\Delta_N$ the set of $N-$positive simple roots. Recall that $\mathfrak{E}(N\backslash G,\mathfrak{q}^2)$ consists of $K(\mathfrak{q}^2)-$invariant functions, which motivates the condition in the following result:
\begin{lemma}\label{Cor:Wsupport}
  Suppose that $f$ is $K(\mathfrak{q}^2)-$invariant. Then $W(f,c,ak)$ is nonvanishing only if $a^\alpha\in \mathfrak{q}^{-2}$ for all $\alpha\in \Delta_N$. In other words it is necessary that $a_i/a_{i+1}\in \mathfrak{q}^{-2}$.
\end{lemma}
\begin{proof}
  Similar to the proof of Proposition \ref{proposition:localized-whittaker-function-subcyclic-parameter-concentrates}, with $\chi_\tau$ replaced by the trivial character, one can compare the left $(N,\psi)$-equivariance of $W(f,c,\bullet)$ against its right $K(\mathfrak{q}^2)$-invariance. The support requirement follows easily.
\end{proof}

As in \cite[\S21.4]{2021arXiv210915230N}, we define for $\phi \in C_c^\infty(F^n - \{0\})$ the Rankin---Selberg integral
\begin{equation*}
    \mathcal{Q}(\phi, f, c) := \int_{g \in N \backslash G} \phi (e_n g) \lvert W (f, c, g ) \rvert^2 \, d g.
\end{equation*}
At non-archimedean places we in fact require this definition only when $\phi$ is
the characteristic function of the subset of primitive integral elements,
so we henceforth suppose that $\phi$ is of this form.

\begin{lemma}
  \label{lemma:let-d--0.-supp-that-f-in-c_cinftyn-backsl-g-left-k}
  Let $D > 0$.  Suppose that $f \in \mathcal{S}(N \backslash G)$ is left $K_A$-invariant and supported on elements $g$ with $\lvert \det g \rvert = D$.  Then
  \begin{equation*}
      \mathcal{Q}(\phi, f, c)  \ll \frac{\lVert f \rVert^2}{D \lvert \det c \rvert}.
  \end{equation*}
\end{lemma}
\begin{proof}
  This is an analogue of \cite[Lemma 21.15]{2021arXiv210915230N}, and may be deduced from Lemma \ref{lemma:let-d--0.-let-f-in-mathc-backsl-g-be-such-that-fg-} in the same way.
\end{proof}


\begin{lemma}\label{Lem:SupportQ}
  Let $f \in \mathcal{S}(N \backslash G)$ be such that $f$ is $K(\mathfrak{q}^2)-$invariant and $f(g) \neq 0$ only if $\lvert \det g \rvert = D$. Then for $\phi$ as above and $T=[\mathfrak{o}:\mathfrak{q}^2]$, we have the following:
  \begin{enumerate}
        \item Let $b,c\in A$, then the integral
        $$
  I(f,b,c):=\int_{g \in N \backslash G} \phi (e_n g)  W (f, b, g )\overline{W(f,c,g)}dg
        $$
        is nonvanishing only if $|\det(b)|=|\det(c)|$.
        \item  $  \mathcal{Q}(\phi,f,c) $ is nonvanishing only if
        \begin{equation*}
                D|\det(c)|\leq  T^{n(n-1)/2}.
        \end{equation*}
  \end{enumerate}

\end{lemma}
Note that       (1) is the analogue of  the 2nd nonvanishing condition in \cite[Lemma 21.3]{2021arXiv210915230N}.
\begin{proof}
Recall that one can further expand $W(f,b,g)$ and $W(f,c,g)$ by \eqref{Eq:Wfcg}. From the condition  that $f(g)\neq 0$ only if $|\det(g)|=D$, we get that for fixed $g$, $W(f,b,g)$ and $W(f,c,g)$ are both nonvanishing only if $$|\det(b)|=|\det(c)|=\frac{|\det(g)|}{D}.$$ This proves (1).
  For (2), as $\phi$ is $K-$invariant, one can use the Iwasawa decomposition and factor the integral $  \mathcal{Q}(\phi, f,c)$ as
  \begin{equation*}
   \int\limits_{a\in A}\phi (e_n a_n)
    \int\limits_{k\in K} \lvert W(f,c,ak) \rvert^2\ d k \frac{da}{\delta_N(a)}.
  \end{equation*}
  By the support condition for $\phi$, we get $|a_n|=1$ for nonzero contributions to the integral. As $f$ is $K(\mathfrak{q}^2)-$invariant, we have by Lemma \ref{Cor:Wsupport} that the nonzero contribution comes from
  \begin{equation*}
    |\det(a)|\leq T^{n(n-1)/2}.
  \end{equation*}
  (2) follows then by the condition on the support of $f$ and  \eqref{Eq:Wfcg}.
\end{proof}

Consider now the general case where $P \leq G$ could be $G$ itself, or a standard maximal parabolic subgroup.  We denote by $W_P(f,c,g)$ and $\mathcal{Q}_P(\phi,f,c)$ the (partial) Whittaker transform and (partial pseudo) local Rankin---Selberg integral, defined as in \cite[\S21.6]{2021arXiv210915230N}
\footnote{The current version of \cite{2021arXiv210915230N} has a typo that it uses $\delta_{N''}$ instead of $\delta_N$ in the definition of $W_P$. This typo does not propagate in \cite{2021arXiv210915230N} for archimedean computations as it is relevant only for \cite[Lemma 21.19]{2021arXiv210915230N}, which actually requires using $\delta_N$ to be true. On the other hand the bounds at unramified places would be weaker than used in \cite{2021arXiv210915230N}, but still enough to derive the desired bounds. Compare Lemma
  \ref{Lem:globalperiodboundreduction} with \cite[Lemma 22.1]{2021arXiv210915230N}.
}
by
\begin{equation*}
  W_P(f,c,g) := \delta_{N}^{1/2}(c) \int _{u \in N''} f(c^{-1} w^{-1}_{M''} u g) \psi ^{-1} (u) \,d u
\end{equation*}
and
\begin{equation*}
  \mathcal{Q}_P(\phi,f,c) := \int _{g \in N \backslash G}
  \phi(e_n g) |W_P(f,c,g)|^2 \, d g.
\end{equation*}

Denote
\begin{equation}\label{Eq:DP}
        D_P=\prod\limits_{n'<i\leq n}T^{i-\frac{n+1}{2}}.
\end{equation}

\begin{proposition}[Main bound for $\mathcal{Q}_P$]\label{Prop:mainboundforQp}
  Let $f \in \mathfrak{E}(N \backslash G, \mathfrak{q}^2)$ (as in Definition \ref{Defn:Esubspace}) with $\lVert f \rVert \ll 1$, and let $c \in A''$.
Then $\mathcal{Q}_P(\phi,f,c)$ is nonvanishing only if
\begin{equation*}
D_P|\det(c)|\leq T^{n''(n''-1)/2},
\end{equation*}
 in which case we have
  \begin{equation*}
    \mathcal{Q}_P (\phi, L(Y)f, c) \ll
     \frac{1}{D_P|\det(c)|}
  \end{equation*}
\end{proposition}
\begin{proof}
  The upper bound of this result is an analogue of \cite[Lemma 21.32]{2021arXiv210915230N} with $r=1$, $Y=1$. 

  Recall from Definition \ref{Defn:Esubspace} \eqref{enumerate:cnh5ndggry}
  that for $(a,k) \in A \times K$, we have $f(a k) \neq 0$ only if
  \begin{equation*}
    \lvert
    a_1 \rvert = T^{\frac{1 - n}{2}},
    \quad
    \lvert
    a_2 \rvert = T^{\frac{3 - n}{2}},
    \quad
    \dotsc,
    \quad
    \lvert
    a_n \rvert = T^{\frac{n - 1}{2}}.
  \end{equation*}
  (Note that we are taking $U=N$ in Definition \ref{Defn:Esubspace} here; One can swap between upper right and lower left unipotent subgroups by using a left translate by the longest Weyl element.)
  This implies that as a function in $a''\in M''$, if $f(a'a''k)\neq 0$, then
  \begin{equation*}
    D=|\det(a'')|=|\prod\limits_{i=n'+1}^n a_i|=
 D_P.
  \end{equation*}
%

  By applying Lemma \ref{Lem:SupportQ}(2) to the formula in \cite[Lemma 21.19]{2021arXiv210915230N}
  with $G$ replaced by a general Levi block $M''$, the integral $\mathcal{Q}_P$ would be nonvanishing only when
\begin{equation*}
        D_P|\det(c)|
\leq T^{n''(n''-1)/2}.
\end{equation*}

  Furthermore by \cite[Lemma 21.20]{2021arXiv210915230N} we have
\begin{equation*}
\int\limits_{a'\in A'}\int\limits_{k\in K}\lVert |\det|^{n'/2}R(a'k)f\rVert^2_{L^2(N''\backslash M'')}\frac{da'}{\delta_N(a')}dk=\|f\|^2_{L^2(N\backslash G)}\ll 1,
\end{equation*}
and thus by Lemma \ref{lemma:let-d--0.-supp-that-f-in-c_cinftyn-backsl-g-left-k}
\begin{equation*}
  \mathcal{Q}_P (\phi, f, c) \ll
  \frac{1}{D|\det(c)|}=\frac{1}{D_P|\det(c)|}.
\end{equation*}
\end{proof}
\begin{corollary}\label{Cor:QPbound}
  When $ \mathcal{Q}_P (\phi, f, c)$ is nonvanishing, we have
  \begin{equation*}
    \mathcal{Q}_P (\phi, f, c)\ll T^{-n''(n''-1)/2}\left(\frac{T^{n''(n''-1)/2}}{D_P|\det(c)|}\right)^l
  \end{equation*}
  for any $l\geq 1$.
\end{corollary}
This is an analogue of \cite[Lemma 21.33]{2021arXiv210915230N}, and follows immediately from the upper bound and the nonvanishing condition of the last proposition.

\subsection{Control for denominators}\label{sec:cnjobc8wnb}
The goal of this subsection is to control the denominators of entries of $c\in A''(\mathbb{Q})$ at $p\in S-\{\infty\}$ when $\mathcal{Q}_P(\phi,f,c)$ is non-vanishing, providing the extension of the 1st nonvanishing condition
of \cite[Lemma 21.3]{2021arXiv210915230N}.

We consider for simplicity the case $M''=G$. The general case is similar using \cite[Lemma 21.19]{2021arXiv210915230N}.  Recall from Lemma \ref{Cor:Wsupport} that if $g=ak$ and $f$ is $K(\mathfrak{q}^2)$-invariant, then $W(f,c,ak)\neq 0 $ only if
\begin{equation}\label{Eq:acondition1}
  a^\alpha\in \mathfrak{q}^{-2}, \forall \alpha\in \Delta_N.
\end{equation}
Also recall that by the support of $\phi$, $  \mathcal{Q}(\phi,f,c)\neq 0$ requires
\begin{equation}
  \label{Eq:acondition2}
  |a_n|=1
\end{equation}
for the integral in $g=ak$.  We shall thus assume from now on that $a$ satisfies both \eqref{Eq:acondition1} and \eqref{Eq:acondition2}.
\begin{proposition}\label{Prop:denominatorcontrol}
  There exists fixed $d$ (depending only on the rank of group) such that for $a$ specified as above, $f\in \mathfrak{E}(N\backslash G, \mathfrak{q}^2)$
  and $T=[\mathfrak{o}:\mathfrak{q}^2]$,
  \begin{equation*}
    W(f,c,ak)\neq 0 \Rightarrow |c_i|\leq T^d,\forall i.
  \end{equation*}
\end{proposition}
This result can be reduced to the following result:

\begin{proposition}\label{Prop:WhittakerIntDomain}
  There exists fixed $d$ and domain $\mathcal{N}\subset N$ with the property
  \begin{equation*}
    n=(n_{ij})\in \mathcal{N}\Rightarrow |n_{ij}|<T^d,
  \end{equation*}
  such that for $a$ as above and the Mellin component $f[s]\in \mathcal{I}(s)$ of $f$ which is also $K(\mathfrak{q}^2)-$invariant, we have
  \begin{equation*}
    \int\limits_{u\in N}f[s](w_Guak)\psi^{-1}(u)du= \int\limits_{u\in \mathcal{N}}f[s](w_Guak)\psi^{-1}(u)du.
  \end{equation*}
\end{proposition}
\begin{remark}
  Note that this property would imply the analytic continuation of the Jacquet integral at p-adic places, but not the other way around. The point here is the more effective control of $\mathcal{N}$, compared with, for example, \cite[Corollary 2.3]{MR581582}. It may also be possible to quantify the argument there, though we shall go for a direct  proof in the Appendix.

\end{remark}
\begin{proof}
 [Proof of Proposition \ref{Prop:denominatorcontrol}]

  By the definition of $\mathfrak{E}(N\backslash G, \mathfrak{q}^2)$, $f$ can be written as a linear combination of elements $f_j$ supported on $ N K_A \tilde{T}^{-\rho_N^\vee} k_j K(\mfq^2)$ for $k_j\in B(\mathfrak{o})\backslash K/K(\mfq^2)$, and each $f_j$ is uniquely determined by its value on $ \tilde{T}^{-\rho_N^\vee}k_j$ by its left and right invariance properties.  It then suffices to prove the proposition when $f$ is some $f_j$, which we shall assume from now on.

  Recall that the Mellin component of $f$ is given by
  \begin{equation*}
    f[s](g) =
    \int_{a \in A} (\delta_{N}^{1/2} |\cdot|^s )^{-1} (a)
    f (a g ) \, d a.
  \end{equation*}
  Then for $f=f_j$ and $k\in K$, it is clear that
  \begin{equation}\label{Eq:fsk}
    f[s](k)=
\begin{cases}
      (\delta_{N}^{1/2} |\cdot|^s )( \tilde{T}^{\rho_N^\vee})f( \tilde{T}^{-\rho_N^\vee}k_j) , &\text{\ if }k\in B(\mathfrak{o})k_j K(\mfq^2);\\
      0 , &\text{\ otherwise}.
    \end{cases}
.
  \end{equation}


  By Mellin inversion, we have as in the proof of \cite[Lemma 21.2]{2021arXiv210915230N} that
  \begin{equation*}
    W(f,c,g)=\int\limits_{(\sigma)}|c|^{-s}W_s^0(g)ds,
  \end{equation*}
  where
  \begin{equation*}
    W_s^0(g)=\int\limits_{u\in N}f[s](w_Gug)\psi^{-1}(u)du.
  \end{equation*}
  By Proposition \ref{Prop:WhittakerIntDomain}, we can rewrite
  \begin{equation}\label{Eq:Wcrelation}
    W(f,c,ak)=\int\limits_{u\in \mathcal{N}}\left(\int\limits_{(\sigma)}|c|^{-s}f[s](w_Guak) ds\right) \psi^{-1}(u)du
  \end{equation}

  For fixed $a,k,u\in \mathcal{N}$, one can apply the Iwasawa decomposition
  \begin{equation*}
    w_G u a= a'n'k'.
  \end{equation*}
  By \eqref{Eq:fsk} the inner integral in \eqref{Eq:Wcrelation} is either vanishing, or equals
  \begin{equation*}
    \int\limits_{(\sigma)}|c|^{-s}(\delta_{N}^{1/2} |\cdot|^s )( a'\tilde{T}^{\rho_N^\vee})f( \tilde{T}^{-\rho_N^\vee}k_j) ds
  \end{equation*}
  when $k'k \in B(\mathfrak{o})k_j K(\mfq^2)$, and is non-vanishing only if $|(c^{-1}a'\tilde{T}^{\rho_N^\vee})_i|=1$ for each $i$.  The conclusion follows then from the following lemma.
\end{proof}
\begin{lemma}
  Suppose that $u=(u_{ij})$ with $|u_{ij}|\leq T^{d_1}$, $|a_n|=1$ and $|a^{\alpha}|\leq T^{d_2}$ for any $\alpha\in \Delta_N$. Then there exists $d$ depending only on $d_i$ and the rank of the group such that if we apply the Iwasawa decomposition
  \begin{equation*}
    w_G ua=a'n'k'
  \end{equation*}
  then for each $i$,
  \begin{equation*}
    |a_i'|\leq T^{d}.
  \end{equation*}
\end{lemma}
\begin{proof}
  Let $M_l(g)$ be the maximum of the norm of $l\times l$ minors of the bottom $l$ rows in $g$. Then by the relation $ w_G ua=a'n'k'$, we have
  \begin{equation*}
    M_l(w_Gua)=M_l(a'n').
  \end{equation*}
  Since $a'n'$ is upper triangular, the lemma would follow if $M_{l+1}(a'n')/M_l(a'n')\leq T^{d}$ for some $d$ and any $l$.

  On the other hand, by the strong triangle inequality and the Laplace expansion of determinant, we have
  \begin{equation*}
    M_{l+1}(w_Gua)\leq M_{l}(w_Gua)\max\{|(w_Gua)|_{i,j}\},
  \end{equation*}
  with $\max\{|(w_Gua)|_{i,j}\}\leq T^{d}$ for some fixed constant $d$ depending only on $d_1,d_2$ and the rank of the group, as $|u_{i,j}|\leq T^{d_1}$ and $|a_i|\leq T^{nd_2}$. Thus the lemma holds.
\end{proof}

\subsection{Reduction to local bounds}\label{sec:cnh5ncy2ak}
Consider now the global situation. Let $S$ be as in Proposition \ref{proposition:proposition-20-14}, and $T_f=\prod\limits_{p \in S \text{\ finite}} T_p$.  For any positive integer $d$, let
\begin{equation*}
\mathbb{Z}[1/T_f^d]
\end{equation*}
 denote the set of reduced rational numbers whose denominators divide $T_f^d$.

Let $d\in \mathbb{Z}_{\geq 0}$ now be as in Proposition \ref{Prop:denominatorcontrol}.
For any rational number $m$, let $\Sigma^P(m)$ denote the set of $c\in A''(\mathbb{Q})$ with $\det(c)=m$ and $\prod_{\mathfrak{p} \in S} \mathcal{Q}_P (\phi_\mathfrak{p}, L(Y)f_\mathfrak{p}, c)$ nonvanishing. Here we consider  $Y\in A(\mathbb{A})$ with trivial components at finite places.
\begin{lemma}\label{Lem:globalperiodboundreduction}
  With notation and assumptions as in Lemma \ref{lemma:let-f_s-=-otim-in-s-f_mathfr-be-as-abov-assume-tha} and Proposition \ref{Prop:mainboundforQp}, we have
  \begin{equation*}
    \mathcal{Q}_P (\phi, L(Y)f_S) \ll
    \sum_{0 \neq m \in \mathbb{Z}[1/T_f^{nd} ]}
        \tau (mT_f^{nd})^{O(1)}
        \max_{c \in \Sigma^P (m)}\{\mathcal{Q}_{P,S}(\phi, L(Y)f_S,c)\},
  \end{equation*}
  where $\tau(x)$ is the number of divisors of an integer $x$,
  and
\begin{equation*}
\mathcal{Q}_{P,S}(\phi, L(Y)f_S,c):=
  \prod_{v \in S}  \mathcal{Q}_P (\phi_v, L(Y)f_v, c).
\end{equation*}
\end{lemma}
\begin{proof}
  This is an analogue of \cite[Lemma 22.1]{2021arXiv210915230N},
  and may be deduced in the same way from Lemma \ref{lemma:let-f_s-=-otim-in-s-f_mathfr-be-as-abov-assume-tha}, \ref{Lem:SupportQ}, Proposition \ref{Prop:denominatorcontrol} and \cite[Lemma 21.3, 21.19]{2021arXiv210915230N}.  In particular 
   we only need to consider those $m\in \mathbb{Z}[1/T_f^{nd}]$ and $c\in \Sigma^P(m)$ with denominators bounded by $T_f^d$, in which case
\begin{equation*}
|\Sigma^P (m)|\ll \tau(mT_f^{nd})^{O(1)}.
\end{equation*}

\end{proof}

\subsection{Combining the local bounds}\label{sec:cnq4lj3ylk}

Here we combine the archimedean local estimates in \cite[Section 21]{2021arXiv210915230N} with the p-adic estimates from Section \ref{sec:cnh5ncyz7d}.

Recall the auxiliary parameter $\RR$  from Section \ref{Sec:ArchimedeanTest}.
Denote $D_P=\prod\limits_{v\in S} D_{P,v}$ where $D_{P,p}$ is defined as in \eqref{Eq:DP} for finite place $p\in S$, and $D_{P,\infty}$ is defined similarly as
\begin{equation*}
\prod\limits_{n'<i\leq n}R^{i-\frac{n+1}{2}}.
\end{equation*}

For each $p\in S$ let $f_p$ be a fixed element of  $\mathfrak{E}(N \backslash G, \mathfrak{q}^2)$ (as translated from, for example, Section \ref{Sec:constructfp}).
Let $f_\infty$ be as in Section \ref{Sec:ArchimedeanTest}. Recall from \cite[Lemma 21.26]{2021arXiv210915230N} a decomposition
\begin{equation*}
f_\infty=\sum\limits_{\mu}f^\mu+ f^{\clubsuit}.
\end{equation*}
Then $f_S=\bigotimes\limits_{v\in S} f_v$ has a similar decomposition
\begin{equation}\label{Eq:decomposefS}
f_S=\sum\limits_{\mu}f_S^\mu+ f_S^{\clubsuit}
\end{equation}
with the local components $(f_S^\mu)_p=(f_S^\clubsuit)_p=f_p$.  We shall not directly use \cite[Proposition 21.28]{2021arXiv210915230N}, but the ingredients \cite[Lemma 21.29-21.34]{2021arXiv210915230N} leading to it.
 Note first that by a change of variable we have the following:
\begin{lemma}
  Over a local field, we have
  \begin{equation*}
    \mathcal{Q}_{P}(\phi,L(Y)f,c)=
    \mathcal{Q}_P\left(\phi,f,\frac{c}{Y''}\right)
  \end{equation*}
\end{lemma}
\begin{proof}
  By definition,
  \begin{equation*}
    \mathcal{Q}_P(\phi,L(Y) f,c) = \int _{g \in N \backslash G}
    \phi(e_n g) \left| W_P(L(Y) f,c,g) \right|^2 \, d g
  \end{equation*}
  with
  \begin{equation*}
    W_P(f,c,g)  =
    \delta_{N}^{1/2}(c) \int _{u \in N''} f(c^{-1} w^{-1}_{M''} u g) \psi ^{-1} (u) \,d u.
  \end{equation*}
  By definition of $L(Y)$,
  \begin{equation*}
    W_P(L(Y) f,c,g)  =
    \delta_{N}^{1/2}(c) \delta_N^{-1/2}(Y) \int _{u \in N''} f(Y c^{-1} w^{-1}_{M''} u g) \psi ^{-1} (u) \,d u.
  \end{equation*}
  Write $Y=Y'Y''$ and substitute $g \mapsto(Y')^{-1} g$.  This introduces the factor $\delta_N(Y')$ for the integral in $g$.  We get
  \begin{align*}
    &\mathcal{Q}_P(\phi,L(Y) f,c)\\
    =
    &\int _{g \in N \backslash G}
      \phi(e_n g)	\delta_{N}(c)\delta_N^{-1}(Y'Y'') \left|\int _{u \in N''} f(Y'Y'' c^{-1} w^{-1}_{M''} u g) \psi ^{-1} (u) \,d u \right|^2 \, d g\\
    =&\int _{g \in N \backslash G}
       \phi(e_n g)	\delta_{N}(c)\delta_N^{-1}(Y'') \left|\int _{u \in N''} f(Y'' c^{-1} w^{-1}_{M''} u g) \psi ^{-1} (u) \,d u \right|^2 \, d g\\
    =&
    \mathcal{Q}_P(\phi,f,\frac{c}{Y''}).
  \end{align*}
\end{proof}
This allows us to extend the local bounds in \cite{2021arXiv210915230N} for $\mathcal{Q}_P$ to larger range for $Y$, whenever the bound $|\det(c)|_\infty\geq 1$ was used in a mild way.  In particular we have the following:
\begin{lemma}\label{Lem:3bounds}
  \begin{enumerate}
  \item\label{enumerate:cnq4los175}
    \begin{equation*}
      \mathcal{Q}_{P,S}(\phi, L(Y)f_{S}^\clubsuit,c)\ll \frac{T^{-\infty}}{|\det(c)|_\infty^{9}
        |\det(c)|_S
      };
    \end{equation*}
  \item\label{enumerate:cnq4los3pi}
    \begin{equation*}
      \mathcal{Q}_{P,S}(\phi, L(Y)f_{S}^\mu,c)\ll \frac{T^{O(1)}}{|\det(c)|_\infty^{9}
        |\det(c)|_S
      };
    \end{equation*}
  \item\label{enumerate:cnq4los4kq}
    \begin{equation*}
      \mathcal{Q}_{P,S}(\phi, L(Y)f_{S}^\mu,c)\ll \frac{r^{n''}T^{o(1)}\det(Y'')}{D_P|\det(c)|_S}.
    \end{equation*}
  \end{enumerate}

\end{lemma}
\begin{proof}
  For (1)/(2)/(3) we combine Proposition \ref{Prop:mainboundforQp}  with Lemma 21.29/21.30/21.32 of \cite{2021arXiv210915230N}, noting that $R^{-\infty}=T^{-\infty}$, $R^{O(1)}=T^{O(1)}$. \cite[Lemma 21.30]{2021arXiv210915230N} used that $|\det(c)|_\infty\geq 1$, but in a very mild way. Indeed in our case we have $|\det(c)|_\infty\gg T^{-O(1)}$, which can be absorbed in $T^{O(1)}$ in the final bound.

\end{proof}
When
\begin{equation}\label{Eq:simplerrange}
        \text{
        $P=G$, or $r^{n''}\det(Y'')\ll D_P r^n \det(Y) T^{o(1)}$,
        }
\end{equation}
 (3) above actually suffices for our purpose, giving
\begin{equation}\label{Eq:QPboundin1range}
  \mathcal{Q}_{P,S}(\phi, L(Y)f_{S}^\mu,c)\ll \frac{r^{n}T^{o(1)}\det(Y)}{|\det(c)|_S}.
\end{equation}
This is the analogue of \cite[Proposition 21.31]{2021arXiv210915230N}.
We need however additional preparations for the remaining range.
\begin{lemma}\label{Lem:comparenn''}
  When
  \begin{equation}\label{Eq:complementaryrange}
  P\neq G, n'\geq 1, r^{n''}\det(Y'')\ggg D_P r^n \det(Y)T^{o(1)},
  \end{equation}
  and $\delta_0$ small enough, we have
  \begin{equation*}
    \frac{r^{n''}(RT_f)^{n''(n''-1)/2}\det(Y'')}{D_P}\ll T^{-\epsilon}
  \end{equation*}
  for some fixed $\epsilon>0$.
\end{lemma}
\begin{proof}
  The proof is essentially the same as for \cite[Lemma 21.34]{2021arXiv210915230N}, especially when $R=T_\infty$. In general we use the fact that $\delta_0$ can be chosen to be any fixed small enough number, so the effect of $R$ can be controlled.
\end{proof}

\begin{lemma}\label{Lem:firstaltComplementaryrange}
  Under the assumption of Lemma \ref{Lem:comparenn''}, we have
  \begin{equation*}
    \mathcal{Q}_{P,S}(\phi, L(Y)f_{S}^\mu,c)\ll
     \frac{T^{-\infty}}{|\det(c)|_\infty^{9}
        |\det(c)|_S
    }
  \end{equation*}
\end{lemma}
\begin{proof}
  Note that $|\det(c)|_S\geq 1$, thus from the lemma above we have
  \begin{equation*}
    \frac{r^{n''}R^{n''(n''-1)/2}\det(Y'')}{D_{P,\infty}|\det(c)|_\infty}\prod\limits_{p\in S} \frac{T_p^{n''(n''-1)/2}}{D_{P,p}|\det(c)|_p}\ll T^{-\epsilon}
  \end{equation*}
  from Lemma    \ref{Lem:comparenn''}. By the nonvanishing condition in Proposition \ref{Prop:mainboundforQp}, either $\mathcal{Q}_{P,S}$ is zero so the lemma is automatic, or we have
  \begin{equation*}
    \frac{T_p^{n''(n''-1)/2}}{D_{P,p}|\det(c)|_p}\geq 1
  \end{equation*}
  for all $p\in S$, in which case we have
  \begin{equation*}
    \frac{r^{n''}R^{n''(n''-1)/2}\det(Y'')}{D_{P,\infty}|\det(c)|_\infty}\ll T^{-\epsilon}.
  \end{equation*}
  We then apply \cite[Lemma 21.33]{2021arXiv210915230N} for the archimedean component. Note that the proof there used $|\det(c)|_\infty\geq 1$, but in a mild way. In our case the possible denominator for $c$ is bounded by $T^d$ for some fixed $d$, which can be absorbed in $T^{C_1}$ of \cite[Lemma 21.33]{2021arXiv210915230N}. Combining \cite[Lemma 21.33]{2021arXiv210915230N} with  $l=l_0$ and Corollary \ref{Cor:QPbound} with $l=l_0-9$ for $l_0$ large enough, we get
  \begin{align*}
        \mathcal{Q}_{P,S}&\ll T^{O(1)}\left(\frac{r^{n''}R^{n''(n''-1)/2}\det(Y'')}{D_{P,\infty}|\det(c)|_\infty}\right)^{l_0}\prod\limits_{p\in S} \left(\frac{T_p^{n''(n''-1)/2}}{D_{P,p}|\det(c)|_p}\right)^{l_0-9}\\
        &=T^{O(1)}\left(\frac{r^{n''}(RT_f)^{n''(n''-1)/2}\det(Y'')}{D_{P}} \right)^{l_0}\frac{1}{|\det(c)|_S^{l_0-9}|\det(c)|_\infty^9}\\
        &\ll T^{O(1)-l_0\epsilon}\frac{1}{|\det(c)|_S|\det(c)|_\infty^9}.
\end{align*}
Here we have used that $|\det(c)|_S=\prod\limits_{\mathfrak{p}\in S}|\det(c)|_\mathfrak{p}\geq 1$ and $l_0$ is large enough. Taking $l_0\rightarrow \infty$, we conclude the proof.
\end{proof}

Let $c \in \Sigma^P(m)$.  Combining the discussions so far we get the following
\begin{lemma}\label{lemma:cnjen3r9kt}
  Suppose that $r=T^{O(1)}$, $T^{-1/2}\leq |Y_j|\leq T^{1/2}$. We have
  \begin{equation*}
 \mathcal{Q}_{P,S}(\phi, L(Y)f,c)
    \ll
    \min
    \left(
      \frac{r^n T^{o(1)} \det (Y) }{\lvert m \rvert_{S}},
      \frac{T^{O(1)}}{\lvert m \rvert _S\lvert m \rvert _\infty^9}
    \right)
    +
    \frac{T^{- \infty }}{ \lvert m \rvert _S\lvert m \rvert _\infty^9}.
  \end{equation*}
\end{lemma}
\begin{proof}
The bound follows from combining the decomposition \eqref{Eq:decomposefS}, Lemma \ref{Lem:3bounds}(1)(2) together with \eqref{Eq:QPboundin1range} in case of the range \eqref{Eq:simplerrange}, or   Lemma \ref{Lem:firstaltComplementaryrange} in case of the complementary range \eqref{Eq:complementaryrange}.
\end{proof}

\subsection{Application of local bounds}\label{sec:cnh5ncy4e8}

\begin{proof}
[Proof of Proposition \ref{proposition:proposition-20-14}]
  Recall our goal is to prove for $r=T^{O(1)}$
  \begin{equation*}
    \mathcal{Q}_P (\phi, L(Y) f_S ) \ll |\det (Y)|_S \ r^n T^{o(1)}.
  \end{equation*}
  We start with the bound in Lemma \ref{Lem:globalperiodboundreduction}:
\begin{equation*}
  \mathcal{Q}_P (\phi, L(Y)f_S) \ll \sum_{0 \neq m \in \mathbb{Z}[1/T_f^{nd} ]} {\tau (mT_f^{nd})^{O(1)}} \max_{c \in \Sigma^P (m)} \prod_{\mathfrak{p} \in S} \mathcal{Q}_P (\phi_\mathfrak{p}, L(Y_\mathfrak{p})f_\mathfrak{p}, c).
  \end{equation*}
By
Lemma \ref{lemma:cnjen3r9kt} with $T^{-\infty}$ term  negligible,  we have
  \begin{align*}
    \max_{c \in \Sigma^P (m)}
    \prod_{\mathfrak{p} \in S} \mathcal{Q}_P (\phi_\mathfrak{p}, L(Y_\mathfrak{p})f_\mathfrak{p}, c) &\ll  \min                 \left(                                                    \frac{r^n T^{o(1)} |\det (Y)| }{\lvert m \rvert_{S}},              \frac{T^{O(1)}}{\lvert m \rvert^{9}_{\infty}\lvert m \rvert_{S}}                               \right)
                                                                                    \\
                                                                                  &\ll \frac{|\det(Y)|}{|m|_S}T^{o(1)}\min\left(r^n, \frac{T^{C}}{\lvert m \rvert^{9}_{\infty}}\right)
  \end{align*}
The result follows now from the following lemma.

\end{proof}
\begin{lemma}\label{lemma:cnjen3r5jb}
  Fix $d,C \geq 0$.  Let $S$ be a finite set of places of $\mathbb{Q}$, containing $\infty$.  Let $T_f=\prod\limits_{p\in S}T_p \ll T$ be an integer and $\rn\in\mathbb{R}_{>0}$ such that $\rn=T^{O(1)}$.  Then
  \begin{equation*}
    \sum_{0 \neq m \in \mathbb{Z} [1/T_f^d]} \frac{\tau(mT_f^d)^C}{\lvert m \rvert_S}
    \min \left( \rn,  \frac{T^{C}}{\lvert m \rvert _\infty^{9}}\right)
    \ll A T^{o(1)}.
  \end{equation*}
\end{lemma}

\begin{proof}
  We start with a change of variable $mT_f^d\rightarrow m$, and note that $|m|_S$ is not changed by the product law.
  After replacing $C$ by a new fixed constant, we reduce to verifying that
  \begin{equation*}
    J:=\sum_{0 \neq m \in \mathbb{Z} } \frac{\tau(m)^C}{\lvert m \rvert_S}
    \min \left( \rn,  \frac{T^{C}}{\lvert m \rvert _\infty^{9}}\right)
    \ll \rn T^{o(1)}
  \end{equation*}
  for fixed and large enough $C$.  Write $m =\pm a b$, where $a$ is a product of primes in $S$ and $b$ is a positive integer coprime to $S$.  Then $|m|_S=|b|_S=b$.  The contribution when either $a\geq T^C$ or $b\geq T^C$ would be negligible due to the fast decay of $\frac{T^C}{|m|_\infty^9}$. It thus suffices to reduce the range of sum to $a, b< T^C$,
  where we have $\tau(m)^C=T^{o(1)}$. Then we have
  \begin{align*}
    J&\ll T^{o(1)}\sum\limits_{a,b< T^C}\frac{1}{b}\min\left( A, \frac{T^C}{a^9b^9}\right)\\
     &\leq T^{o(1)}\sum\limits_{a< T^C}\min\left( A, \frac{T^C}{a^9}\right)\sum\limits_{b<T^C}\frac{1}{b}\\
     &\ll T^{o(1)}\sum\limits_{a< T^C}\min\left( A, \frac{T^C}{a^9}\right)\\
     &\leq AT^{o(1)}\sum_{
       \substack{
       a \leq T^C :  \\
    p \mid a \implies p \in S
    }
    }1.
  \end{align*}
  It remains to prove that
  \begin{equation}\label{eqn:sum_-substack-leq-tc-:-}
    \sum_{
      \substack{
        a \leq T^C :  \\
        p \mid a \implies p \in S
      }
    }
    1
    \ll T^{o(1)}.
  \end{equation}

  To see this, we apply Rankin's trick, as follows.  For any fixed small enough $\eps > 0$, we have
  \begin{equation*}
    \sum_{
      \substack{
        a \leq T^C :  \\
        p \mid a \implies p \in S
      }
    }
    1
    \leq
    T^{C \eps}
    \sum_{
      \substack{
        p \mid a \implies p \in S
      }
    }
    a^{-\eps},
  \end{equation*}
  \begin{equation*}
    \sum_{
      \substack{
        p \mid a \implies p \in S
      }
    }
    a^{-\eps}
    \leq
    \prod_{p \in S}
    \frac{1}{1 - p^{-\eps}}.
  \end{equation*}
  Note that
  \begin{equation*}
    \frac{1}{1 - p^{- \eps }} \ll \exp(2p^{-\eps}),
  \end{equation*}
  which becomes inequality once $p$ is large enough in terms of $\eps$.  Thus
  \begin{equation*}
  \prod_{p \in S} \frac{1}{1 - p^{-\eps}}\ll_\eps \exp(2\sum\limits_{p\in S}p^{-\eps}).
  \end{equation*}
  our task \eqref{eqn:sum_-substack-leq-tc-:-} reduces to verifying that
  \begin{equation}\label{eqn:sum_p-in-s-p-eps-lll-log-t.-}
    \sum_{p \in S} p^{-\eps} \lll \log T.
  \end{equation}
  Write $S = \{p_1,\dotsc,p_{|S|-1},\infty\}$, with $p_j < p_{j+1}$.  Then $p_j \geq j$, so
  \begin{equation*}
    \sum_{p \in S} p^{-\eps} \leq \sum_{n =1}^{\lvert S  \rvert-1} \frac{1}{ n^\eps } \asymp \lvert S \rvert^{1 - \eps}.
  \end{equation*}
  Using that $\lvert S \rvert \lll \log T$, 
  we deduce the required bound~\eqref{eqn:sum_p-in-s-p-eps-lll-log-t.-}.
\end{proof}

\section{Proof of Theorem \ref{MainTheorem}} \label{sec:cnpv1hj5sh}
\subsection{General adaptions}\label{Sec:generaladaption}
In this section we give the proof of Theorem \ref{MainTheorem}, which is mostly parallel to \cite[Section 5]{2021arXiv210915230N}, with some general adaptions/modifications as follows:

\begin{enumerate}
\item We have introduced an auxiliary parameter $R$  in \eqref{Eq:auxiliaryR} to replace $T_\infty$ in \cite{2021arXiv210915230N}, with $\delta_0$ small enough and fixed, to be optimized later on.
\item
  In \cite{2021arXiv210915230N}, the set $S$  consists of all places at which $\pi$ is ramified (including the archimedean place), and a special role  was played by the subset $\{\infty\} \subseteq S$. In this paper we use instead the notation $S\subseteq S_\pi$ where $S$ denote the set of interesting places as in Theorem \ref{MainTheorem} among the set $S_\pi$ of all places with ramifications.
  Local factors or integrals at $\infty$ in \cite{2021arXiv210915230N} are thus replaced with products over $S$ of analogous quantities.
  The factors/integrals away from $\infty$ will be replaced by those away from $S$, taking the conductor  $C(\pi^S)$ for an example. Another example is $\Phi[f_\infty]$ defined in \cite[(2.49)]{2021arXiv210915230N} and used for constructing pseudo Eisenstein series, which is now replaced by (see Section \ref{Sec:completePseudoEisenstein})
  \begin{equation}\label{Eq:adaptedPhif}
    \Phi[f_S](s)=\mathcal{P}_G(s)\zeta^{(S)}(N,s)\otimes _{\mathfrak{p} }f_\mathfrak{p}[s].
  \end{equation}

\item \cite[Theorem 3.1]{2021arXiv210915230N} is replaced by Theorem \ref{theorem:main-local-noncompact} for p-adic places $p\in S$, and replaced by \cite[Theorem 3.1]{2021arXiv210915230N}/Lemma \ref{Lemma:archimedeananalogueforR} for the archimedean place, depending on  whether $R=T_\infty$ or not (recall \eqref{Eq:auxiliaryR}).
\item \cite[Theorem 4.1]{2021arXiv210915230N} is replaced by Theorem \ref{theorem:eisenstein-growth-bound-general-level}.
\end{enumerate}

\subsection{Finishing the proof}
We may carry out the rest of the proof as in \cite[Section 5]{2021arXiv210915230N}, keeping in mind the adaptions of Section \ref{Sec:generaladaption}.
Note that we can still directly apply the bound for the matrix counting in \cite[Section 5.9]{2021arXiv210915230N}, as the local testing function $\omega_p^\sharp$ used in Theorem \ref{theorem:main-local-noncompact} at $p\in S$ is still contained in $\overline{G}(\mathbb{Z}_p)$.

In the case $R=T_\infty$ we then obtain Theorem \ref{MainTheorem} directly; In the case $T_\infty<R$, all the resulting upper bound we get will have an additional coefficient $R^{O(1)}$ due to the ineffective controls in Lemma \ref{Lemma:archimedeananalogueforR}, and in particular we get
\begin{equation*}
  L(\pi,1/2)\ll R^{O(1)}T^{n/4-\delta}C(\pi^S)^{1/2}
\end{equation*}
for some fixed $\delta>0$.  By picking now $\delta_0$ to be small enough but fixed, we conclude the proof of Theorem \ref{MainTheorem}.


\appendix
\section{Proof of Proposition \ref{Prop:WhittakerIntDomain}}\label{Sec:appendix}
We introduce some notions first.  For a positive integer $\rho$, denote
\begin{equation*}
  A(\rho):=
  \begin{pmatrix}
    p^{(n-1)\rho} & & & \\
               & p^{(n-2)\rho} & & \\
               & & \cdots & \\
               & & & 1
  \end{pmatrix}
  .
\end{equation*}
For any matrix $g\in G$, denote by $g^{A(\rho)}$ the matrix $A(\rho)gA(\rho)^{-1}$.  Denote the valuation of the $(i,j)-$th entry of a matrix by
\begin{equation*}
  v_{i,j}(g)=v((g)_{i,j}).
\end{equation*}

\begin{definition}\label{Defn:niceNdomain}
  By a \textbf{nice domain} of slope 0, we mean the subset $N_0=N(\mathfrak{o})$ of $N$.  By a nice domain of slope $\rho> 0$ and remainder $ l\geq 0$, we mean a subset $N_0$ of $N$ consisting of unipotent matrices $u$ satisfying the following:
  \begin{enumerate}
  \item\label{Item:u0} There exists $u_0\in N(\mathfrak{o}/p^n\mathfrak{o})$ such that $N_0^{A(\rho)} =u_0 K_N(p^n\mathfrak{o})$. (Here $K_N(p^n\mathfrak{o})=\{u\in N| v_{i,j}(u-I)\geq n\}$.) In particular $N_0$ is a translate of an open compact subgroup of $N$;
  \item\label{Item:remainder} For any pair of indices $i<j$, we have
    \begin{equation*}
      v_{i,j}(u_0)\geq l;
    \end{equation*}
  \item\label{Item:achievingremainder} There exists a pair $(i_0,j_0)$ with $j_0-i_0>0$ minimal such that
    \begin{equation*}
      v_{i_0,j_0}(u_0)=l.
    \end{equation*}
  \item\label{Item:i0maximal} Furthermore if there are multiple pairs $(i,j)$ with $j-i=j_0-i_0$ and $v_{i,j}(u_0)=l$, we assume $i_0$ is maximal.
  \item\label{Item:rhominimal} $\rho$ is minimal in the sense that $N_0^{A(\rho-1)}\not\subset N(\mathfrak{o})$. In particular we have $l<n$.
  \end{enumerate}
\end{definition}
\begin{lemma}
  $N$ can be decomposed as the disjoint union of nice domains.
\end{lemma}
\begin{proof}
  It is clear that any $n\in N$ lands in some nice domain. Actually the slope of the corresponding domain is uniquely determined by item \eqref{Item:rhominimal} above.  So will $u_0$ be determined by $n^{A(\rho)} \mod p^n\mathfrak{o}$.

  It is also easy to check that the nice domains with different slopes are disjoint from each other, and so are those with the same slope $\rho$ and different $u_0\in N(\mathfrak{o}/p^n\mathfrak{o})$.
  (Note that some choices of $u_0$ may not satisfy Item \eqref{Item:rhominimal}, but this does not affect the argument above.)
\end{proof}
\begin{example}
  The matrix
  \begin{equation*}
    \begin{pmatrix}
      1 & p^{-k+1} & p^{-2k+1}\\
        &1 & p^{-k+2}\\
        & & 1
    \end{pmatrix}
  \end{equation*}
  for $k> 0$ is in a nice domain $N_0$ with slope $k$ and remainder $1$, such that
  \begin{equation*}
    N_0^{A(k)}=
    \begin{pmatrix}
      1 & p+p^3\mathfrak{o} & p+p^3\mathfrak{o}\\
        &1 & p^2+ p^3\mathfrak{o}\\
        & & 1
    \end{pmatrix}
    .
  \end{equation*}
\end{example}
As in Proposition \ref{Prop:WhittakerIntDomain}, let $a\in A$ satisfy $|a_n|=1$ and $|a^{\alpha}|\leq T^{d_2}$ for any $\alpha\in \Delta_N$; let $f[s]\in \mathcal{I}[s]$ be $K(\mathfrak{q}^2)-$invariant.  Now Proposition \ref{Prop:WhittakerIntDomain} can be easily reduced to the following lemma:

\begin{lemma}\label{Lem:vanishingforlargeslope}
  For $a$ and $f[s]$ as above, and for any nice domain $N_0$ with slope $\rho\gg_G v(T)$, we have
  \begin{equation*}
    \int\limits_{u\in N_0}f[s](w_Guak)\psi^{-1}(u)du=0.
  \end{equation*}
\end{lemma}

To prove such vanishing result, our strategy is by the following lemma.  For simplicity denote
\begin{equation*}
  h(g)=f[s](w_G g ak).
\end{equation*}
\begin{lemma}\label{Lem:twoparafamily}
  Let $a,h$ be as above.  Let $N_0$ be a nice domain with slope $\rho\gg_G v(T)$ and remainder $l$ such that $N_0=u_0^{A(-\rho)}K_N(p^n\mathfrak{o})^{A(-\rho)}$.  Then there exists a function $n(u,x)$ for $u\in N_0$ and $x\in \mathfrak{o}$ satisfying the following:
  \begin{enumerate}
  \item\label{Vanishing:item1} $n(u,x)\in K_N(p^n\mathfrak{o})^{A(-\rho)}$;
  \item\label{Vanishing:item2} For all $u\in N_0$, $x\in \mathfrak{o}$,
    \begin{equation*}
      h(u)=h(un(u,x));
    \end{equation*}
  \item\label{Vanishing:item3} We can make a change of variable $un(u,x)=w$,
  so that for the function $n'(w,x)$ satisfying $u=wn'(w,x)$, we have
    \begin{align*}
      &\int\limits_{u\in N_0}h(un(u,x))\psi^{-1}(u)du\\
      =&\int\limits_{w\in N_0}h(w)\psi^{-1}(wn'(w,x))dw;
    \end{align*}
  \item\label{Vanishing:item4} $n'(w,x)$ satisfies
    \begin{equation*}
      \left(n'(w,x) \right)_{i,i+1}
      \begin{cases}
        \equiv -p^{-l-1}(u_0)_{i_0,j_0}x \mod\mathfrak{o}, &\text{\ if }i=i_0\\
        \in \mathfrak{o}, &\text{\ otherwise.}
      \end{cases}
    \end{equation*}
  \end{enumerate}

\end{lemma}
\begin{proof}
 [Proof of Lemma \ref{Lem:vanishingforlargeslope}]
  \begin{align*}
    &\int\limits_{u\in N_0}h(u)\psi^{-1}(u)du\\
                =&\int\limits_{x\in \mathfrak{o}}\int\limits_{u\in N_0}h(un(u,x))\psi^{-1}(u)dudx\\
                =&\int\limits_{w\in N_0}\int\limits_{x\in \mathfrak{o}}h(w)\psi^{-1}(wn'(w,x))dxdw\\
                =&0.
\end{align*}
Here the 2nd, 3rd and 4th line follows respectively from item \eqref{Vanishing:item2}, \eqref{Vanishing:item3}, \eqref{Vanishing:item4} of Lemma \ref{Lem:twoparafamily} and $v_{i_0,j_0}(u_0)=l$.
\end{proof}

To devise the function $n(u,x)$ as in Lemma \ref{Lem:twoparafamily}, we need to utilize the invariance properties of $f[s]$. Let $U$ be the set of lower unipotent matrices, $Q=AU$. Let $K_A$ be the set of diagonal matrices with entries in $\mathfrak{o}^\times$, and $K_Q(p^d\mathfrak{o})=\{x\in Q| v_{i,j}(x-I)\geq d, \forall i,j\}$.  The following result is clear:
\begin{lemma}\label{Lemma:hinvariance}
        For $a$ and $f[s]$ as in Lemma \ref{Lem:vanishingforlargeslope}, and $h(g)=f[s](w_G g ak)$, we have that $h$ is left invariant by $K_AU$ and right invariant by $K_Q(p^d\mathfrak{o})$ for $d\gg_G v(T)$.
\end{lemma}
\begin{proof}
  The left invariance follows directly from the left invariance of $f[s]$ by $K_AN$ and the translation by $w_G$.

  On the other hand $f[s]$ is right invariant by $K(\mathfrak{q}^2)$, which is normalized by $k\in K$. Using that $a^\alpha \in \mathfrak{q}^{-2}$, $aK(\mathfrak{q}^2)a^{-1}$ contains $K_Q(p^d\mathfrak{o})$ for $d\gg_G v(T)$, hence the conclusion.
\end{proof}

It is also more convenient to work with the matrices after conjugation by $A(\rho)$. Thus let $u'=u^{A(\rho)}\in u_0 K_N(p^n\mathfrak{o})$.  Let $(i_0,j_0)$ be as in Definition \ref{Defn:niceNdomain}(3)(4) and
\begin{equation*}
        q_1=I+p^{\rho-l-1}x E_{j_0,i_0+1},
\end{equation*}
where $E_{i,j}$ has a single nonzero entry $1$ at the $(i,j)$ position. Then
\begin{equation*}
        q_1^{A(-\rho)}= I+p^{(j_0-i_0)\rho-l-1}x E_{j_0,i_0+1}.
\end{equation*}
Suppose now that $\rho\gg_G d$ where $d$ is as in Lemma \ref{Lemma:hinvariance}, so that
\begin{equation*}
q_1^{A(-\rho)}\in K_Q(p^d\mathfrak{o}).
\end{equation*}
This is possible as $j_0-i_0\geq 1$.
\begin{lemma}\label{Lemma:conjugatedcongruence}
  When $\rho\gg_G d$															 for $d$ as in Lemma \ref{Lemma:hinvariance}, for any $u'\in u_0 K_N(p^n\mathfrak{o})$ there exists  $q_2\in K_Q(p^{\rho-l-1}\mathfrak{o})$ depending on $u'$ such that
  \begin{equation*}
    w'=q_2u'q_1\in N,
  \end{equation*}
  satisfying
  \begin{enumerate}
  \item[(i)]\label{Item:w'range} $w'\equiv u'\mod p^n$;
  \item[(ii)]\label{Item:2nddiag} When mod $p^{\rho}$
    \begin{equation*}
      (w')_{i,i+1}\equiv
      \begin{cases*}
        (u')_{i,i+1}+ p^{\rho-l-1}(u_0)_{i_0,j_0}x, &
                                                      \text{ if  $i=i_0$}\\
        (u')_{i,i+1}, &\text{\ otherwise}.
      \end{cases*}
    \end{equation*}
  \item[(iii)]\label{Item:changofvar} There exists proper ordering of parameters for $u'$ and $w'$ so that for each fixed parameter $x$, the corresponding Jacobian matrix
    for the change of variables $u'\rightarrow w'$ is upper-triangular and has determinant norm 1.
  \end{enumerate}
\end{lemma}

\begin{proof}
  [Proof of Lemma \ref{Lem:twoparafamily}]
  We construct $n(u,x)$ to be the unipotent matrix satisfying
  \begin{equation*}
    un(u,x)=q_2^{A(-\rho)}u q_1^{A(-\rho)},
  \end{equation*}
  which is equivalent to that
  \begin{equation*}
    w=un(u,x)=(w')^{A(-\rho)}.
  \end{equation*}
  Then item \eqref{Vanishing:item2} follows from Lemma \ref{Lemma:hinvariance} and the range for $q_i$ above.

  Item \eqref{Vanishing:item1}/ \eqref{Vanishing:item4} are direct translation from Item (i)/ (ii) respectively after conjugation.

  Item \eqref{Vanishing:item3} is also straightforward from Item (iii).

\end{proof}

\begin{proof}
  [Proof of Lemma \ref{Lemma:conjugatedcongruence}]
  Take the congruence mod $p^{\rho-l-1}$, we have that the image of $u'q_1$ lies in $N(\mathfrak{o}/p^{\rho-l-1}\mathfrak{o})$. Thus there exists unique $q_2\in K_Q(p^{\rho-l-1}\mathfrak{o})$ such that $w'=q_2u'q_1\in N$.

  Taking the congruence mod $p^n$ and using that $\rho\gg n$, we have $q_i\equiv I$ and the equality $w'=q_2u'q_1$ implies (i).

  For (ii) note first that to land $u'q_1$ back to $N$, $q_2$ corresponds to row operations involving only the rows between $i_0-$th row and $j_0-$th row. To simplify the matrix computations we shall only focus on those rows in the following.

  Consider now the easier case where $j_0-i_0=1$. Then $q_1$ is diagonal and
  %

  \begin{equation*}
    u'q_1
    =
    \begin{pmatrix}
      \cdots &	\cdots &\cdots  & \cdots				& \cdots& \cdots \\
      \cdots &	0 & 1            &(1+p^{\rho-l-1}x)u'_{i_0,i_0+1} & * & \cdots   \\
      \cdots &	0 & 0           &1+p^{\rho-l-1}x  & u'_{i_0+1,i_0+2}& \cdots   \\
      \cdots &	\cdots &\cdots  & \cdots                                                        & \cdots& \cdots
    \end{pmatrix}
    .
  \end{equation*}
  Then we take $q_2$ to correspond to multiplying $(i_0+1)-$th row with $(1+p^{\rho-l-1}x)^{-1}$.
  The $(i_0,i_0+1)-$entry remains
  $$u'_{i_0,i_0+1}+p^{\rho-l-1}xu'_{i_0,i_0+1}$$
  as required,
   and the $(i_0+1,i_0+2)-$entry (if exists) becomes
  \begin{equation*}
    \frac{u'_{i_0+1,i_0+2}}{1+p^{\rho-l-1}x}\equiv u'_{i_0+1,i_0+2} \mod p^\rho
  \end{equation*}
  using Item \eqref{Item:i0maximal} of Definition \ref{Defn:niceNdomain}.  (ii) follows in this case.

  Consider now the case $j_0-i_0>1$. Taking the  congruence mod $p^{\rho}$, we have

  \begin{equation*}
    u'q_1
    \equiv
    \begin{pmatrix}
      \cdots &	\cdots &\cdots  & \cdots				& \cdots& \cdots & \cdots& \cdots\\
      \cdots &	0 & 1            &u'_{i_0,i_0+1}+p^{\rho-l-1}xu'_{i_0,j_0} & * & \cdots   & u'_{i_0,j_0}& \cdots\\
      \cdots &	0 & 0           &1& *  &\cdots & u'_{i_0+1,j_0}& \cdots \\
      \cdots &	0 & 0           &0& 1   &\cdots & u'_{i_0+2,j_0}& \cdots \\
      \cdots &	\cdots &\cdots  & \cdots                                                        & \cdots& \cdots        & \cdots& \cdots\\
      \cdots &	0 & 0           & p^{\rho-l-1}x & 0   &\cdots & 1& \cdots \\
      \cdots &	\cdots &\cdots  & \cdots                                                        & \cdots& \cdots        & \cdots& \cdots\\
    \end{pmatrix}
    .
  \end{equation*}
  Here all terms in $(i_0+1)-$th column below the diagonal are congruent to $0$ using $p^{\rho-l-1}xu'_{i,j_0}\equiv 0\mod p^\rho$ for $i_0<i<j_0$ by Item \eqref{Item:achievingremainder} of Definition \ref{Defn:niceNdomain}, except the $p^{\rho-l-1}x$ term appearing on $(j_0, i_0+1)-$entry. To land back to $N$, we can make use of $(i_0+1)-$th row to $j_0-$th row and have
  \begin{equation*}
    q_2\equiv
    \begin{pmatrix}
      \cdots  &\cdots  & \cdots				& \cdots& \cdots & \cdots& \cdots &\cdots \\
      \cdots &	 1            &0 & 0&0 & \cdots   & 0& \cdots\\
      \cdots &	 0           &1& 0 &0 &\cdots & 0& \cdots \\
      \cdots &	 0           &0& 1 &0 &\cdots & 0& \cdots \\
      \cdots &	 0           &0 &0& 1 &\cdots & 0& \cdots \\
      \cdots &	\cdots  & \cdots                                                        & \cdots& \cdots        & \cdots&\cdots & \cdots\\
      \cdots & 0           & 0 & a_{j_0,i_0+1} &a_{j_0,i_0+2} &\cdots & a_{j_0,j_0}& \cdots \\
      \cdots &\cdots&\cdots  & \cdots                                                   & \cdots& \cdots        & \cdots& \cdots\\
    \end{pmatrix}
  \end{equation*}
  $\mod p^\rho$ for proper $a_{j_0,s}$ with $i_0+1\leq s\leq j_0$.  For $\rho\gg n$, we have $a_{j_0,s}\in p^{\rho-l-1}\mathfrak{o}$ for $s<j_0$, and $a_{j_0,j_0}\equiv 1\mod p^{\rho-l-1}\mathfrak{o}$ by $(u_0)_{i,j}\in \mathfrak{o}$.
  These row operations mod $p^{\rho}\mathfrak{o}$ would only affect $(j_0,j_0+1)-$entry of interest in (ii).
  Thus when mod $p^\rho$,
  \begin{equation*}
w'_{i,i+1}\equiv u'_{i,i+1}, \forall i\neq i_0,j_0;
\end{equation*}
  \begin{equation*}
w'_{i_0,i_0+1}\equiv u'_{i_0,i_0+1}+p^{\rho-l-1}xu'_{i_0,j_0}.
\end{equation*}
  Meanwhile
  \begin{align*}
    w'_{j_0,j_0+1}&\equiv a_{j_0,i_0+1}u'_{i_0+1,j_0+1}+a_{j_0,i_0+2}u'_{i_0+2,j_0+1}+\cdots +a_{j_0,j_0}u'_{j_0,j_0+1}\\
                  &\equiv u'_{j_0,j_0+1}\mod p^\rho
  \end{align*}
  as $v(u'_{i,j_0+1})\geq l+1$ for $i_0+1\leq i\leq j_0$	 by Item \eqref{Item:achievingremainder}\eqref{Item:i0maximal} of Definition \eqref{Defn:niceNdomain}. This concludes the proof of (ii).

  For (iii) note first that when reducing $u'q_1$ back to a unipotent matrix by row operations corresponding to $q_2$, the entries of the resulting matrix $w'$ are rational functions of those in $u'$ for each fixed $x$. We shall order the entries of $u'$ (and same with $w'$) as follows:
  \begin{enumerate}
  \item[1.]$(u')_{i,j}$ will be ordered before $(u')_{k,l}$ as long as $i<k$;
  \item[2.] For entries within the same row, $(u')_{i,j_0}$ will be ordered first, the remaining $(u')_{i,j}$ can be ordered in any way.
  \end{enumerate}
  With this ordering, it is not difficult to see that the Jacobian matrix for the change of variable $u'\rightarrow w'$ is upper triangular. Indeed the multiplication on right by $q_1$ will make $(i,i_0+1)$ entries depending on $(i,j_0)$ entries for $1\leq i\leq n$; Then row operations corresponding to the lower-triangular matrix $q_2$ will make $(i,j)-$entry depending on entries from rows above and $(i,j_0)-$entry.

  To tell the determinant of the Jacobi matrix, take the congruence of  $q_2u'q_1=w'$ modulo, say $p^{n+1}\mathfrak{o}$. As $\rho$ is sufficiently large compared to $n$, we have
\begin{equation*}
(u')_{i,j}\equiv (w')_{i,j} \mod p^{n+1}\mathfrak{o}.
\end{equation*}
This implies that the Jacobi matrix has determinant norm 1.
%
\end{proof}

\begin{example}
  We give an explicit example to show what is happening in the proof of Item (ii)(iii) above. Consider the case $G=\GL_4$, $N_0$ is a nice domain of slope $\rho$ and remainder $l$, which is achieved at $(1,4)-$entry satisfying Item \eqref{Item:achievingremainder} \eqref{Item:i0maximal} of Definition \ref{Defn:niceNdomain}. Then we have
  \begin{equation*}
    u' q_1=
    \begin{pmatrix}
      1 & u'_{1,2}+p^{\rho-l-1}xu'_{1,4} & u'_{1,3} & u'_{1,4}\\
      0 & 1+p^{\rho-l-1}xu'_{2,4}		& u'_{2,3} & u'_{2,4}\\
      0 & p^{\rho-l-1}xu'_{3,4}		&1                      & u'_{3,4}\\
      0 &p^{\rho-l-1}x                          &0			&1
    \end{pmatrix}
  \end{equation*}
  The row operations involves only 2nd-4th rows. Thus
  \begin{equation}
    w'_{1,4}=u'_{1,4}, \ \ w'_{1,2}=u'_{1,2}+p^{\rho-l-1}xu'_{1,4},\ \  w'_{1,3}=u'_{1,3}.
  \end{equation}
  Next we multiply the second row with $(1+p^{\rho-l-1}xu'_{2,4}	)^{-1}$, giving
  \begin{equation*}
    w'_{2,4}=\frac{u'_{2,4}}{ 1+p^{\rho-l-1}xu'_{2,4}}, \ \ w'_{2,3}=\frac{u'_{2,3}}{ 1+p^{\rho-l-1}xu'_{2,4}}.
  \end{equation*}
  This operation further reduces $u'q_1$ to
  \begin{equation*}
    \begin{pmatrix}
      1 & w'_{1,2} & w'_{1,3} & w'_{1,4}\\
      0 & 1			& w'_{2,3} & w'_{2,4}\\
      0 & p^{\rho-l-1}xu'_{3,4}		&1                      & u'_{3,4}\\
      0 &p^{\rho-l-1}x                          &0			&1
    \end{pmatrix}
    .
  \end{equation*}
  Then we use the second row to get rid of $(3,2)-$entry, and multiply the third row by a scalar to make $(3,3)-$entry $1$. This gives
  \begin{equation*}
    w'_{3,4}=\frac{u'_{3,4}-p^{\rho-l-1}xu'_{3,4}w'_{2,4}}{1-p^{\rho-l-1}xu'_{3,4}w'_{2,3}}.
  \end{equation*}
  One also need to change the 4th row, but that won't affect the change of variables in this case. It is clear now the corresponding Jacobi matrix is upper triangular with diagonal entries congruent to $1$ provided $\rho\gg n$.
\end{example}

\def\cprime{$'$} \def\cprime{$'$} \def\cprime{$'$} \def\cprime{$'$}

\end{document}